\documentclass[12pt,reqno]{amsart}

\newcommand{\defeq}{\stackrel{\rm{def}}{=}}
\newcommand{\spn}{\operatorname{span}}

\usepackage{amssymb}
\usepackage{amsthm}
\usepackage{amsxtra}

\usepackage{fancyhdr}
\usepackage{graphicx}
\usepackage{color}
\usepackage{amsmath}
\usepackage{listings} 
\usepackage{float}
\usepackage{cases} 
\usepackage{threeparttable}
\usepackage{dcolumn}
\usepackage{multirow}
\usepackage{booktabs}

\DeclareMathOperator{\supp}{supp}

\usepackage{color}
\usepackage{hyperref}
\hypersetup{
    colorlinks = true,
    linkcolor=black,
}

\newtheorem{theorem}{Theorem}[section]
\newtheorem{definition}[theorem]{Definition}
\newtheorem{proposition}[theorem]{Proposition}
\newtheorem{lemma}[theorem]{Lemma}
\newtheorem{corollary}[theorem]{Corollary}

\theoremstyle{remark}
\newtheorem{remark}[theorem]{Remark}

\numberwithin{equation}{section}

\title[Stability and instability in fKdV]{Stability and instability of solitary waves\\ in  fractional generalized KdV equation\\ in all dimensions}

\author[O. Ria\~no]{Oscar Ria\~no}
\address{Departamento de Matem\'aticas, Universidad Nacional de Colombia, Bogot\'a, Colombia}
\curraddr{}
\email{ogrianoc@unal.edu.co}

\author[S. Roudenko]{Svetlana Roudenko}
\address{Department of Mathematics \& Statistics\\Florida International University,  Miami, FL, USA}
\curraddr{}
\email{sroudenko@fiu.edu}



\begin{document}

\begin{abstract}
We study stability of solitary wave solutions for the fractional generalized Korteweg-de Vries equation
$$
\partial_t u- \partial_{x_1} D^{\alpha}u+  \tfrac{1}{m}\partial_{x_1}(u^m)=0, ~ (x_1,\dots,x_d)\in \mathbb{R}^d, \, \, t\in \mathbb{R}, \, \, 0<\alpha <2,
$$
in any spatial dimension $d\geq 1$ and nonlinearity $m>1$. The arguments developed here are independent of the spatial dimension and rely on the {\it new} estimates for
spatial decay of ground states and their regularity.

In the $L^2$-subcritical case, we prove the orbital stability of solitary waves using the concentration-compactness argument, the commutator estimates and expansions of nonlocal operator $D^\alpha$ in several variables. 

In the $L^2$-supercritical case, we show that solitary waves are unstable. More precisely, the instability is obtained by constructing an explicit sequence of initial conditions that move away from a soliton orbit in finite time, this is shown in conjunction with the modulation and truncation arguments, and incorporating the decay and regularity of the ground states. 

As a consequence, in 1D we show the instability of solitary waves of the supercritical generalized Benjamin-Ono equation ($\alpha=1$) and the dispersion-generalized Benjamin-Ono equation ($1<\alpha<2$); furthermore, new results on the instability are obtained in the weaker dispersion regime when $\frac{1}{2}<\alpha<1$. 
This work should be of interest in studying stability of solitary waves and other coherent structures in a variety of dispersive equations that involve {\it nonlocal} operators. 
\end{abstract}

\maketitle

\tableofcontents

\section{Introduction}\label{S:Introduction}

We study the fractional generalized KdV (fKdV) equation
\begin{align}\label{EQ:fKdV}
\qquad \partial_t u- \partial_{x_1} D^{\alpha}u+  \frac{1}{m}\partial_{x_1}(u^m)=0, ~~ x=(x_1,\dots,x_d)\in \mathbb{R}^d, \, \, t\in \mathbb{R}, \, \, 0<\alpha <2,
\end{align}
where the spatial dimension $d\geq 1$, $m> 1$, $u = u(x_1, \dots, x_d, t)$ is real-valued, and the operator $D^{\alpha}$ denotes the Riesz potential of order $-\alpha$ defined via the singular integral operator
\begin{equation}
D^{\alpha}f(x)=c_{d,\alpha} \, p.v.\int_{\mathbb{R}^d} \frac{f(x)-f(y)}{|x-y|^{d+\alpha}} \, dy,
\end{equation}
for some constant $c_{d,\alpha}>0$, which depends on $\alpha$ and $d$. Equivalently, $D^{\alpha}$ is defined as the Fourier multiplier operator with the symbol $|\xi|^{\alpha}$.

The family of the fractional KdV equations describes several physical models and processes, for example, see \cite{ABFS1989} and references therein. 
Setting the dispersion $\alpha=2$ and the dimension $d=1$ in \eqref{EQ:fKdV}, one can easily recognize the generalized KdV (gKdV) equation, which appears in various physical contexts, for instance, the most famous one, the KdV ($m=2$), that describes the long time evolution of small amplitude dispersive waves that propagate on the surface of shallow water, and also appear in lattices, elastic rods, and plasma physics, e.g. see \cite{Miura1976}. When $\alpha=1$ and $d=1$, fKdV \eqref{EQ:fKdV} coincides with the generalized Benjamin-Ono (BO) equation, a specific case of which with $m=2$ describes waves in stratified fluids and deep water, see \cite{B1967, O1975}. 
Fixing $\alpha=1$ and $m=2$ in dimension $d=2$ in \eqref{EQ:fKdV}, the fKdV yields a higher-dimensional BO equation (called Shrira equation, and often referred to as HBO), which was originally proposed by V. Shrira to describe 
long-wave perturbations in a boundary-layer type shear flow \cite{Shrira1989} (see also a variant of this model in \cite{PS1994}). When $\alpha=2$ and $m=2$ in \eqref{EQ:fKdV}, the resulting equation agrees with Zakharov-Kuznetsov (ZK) equation, which is  used to model weakly nonlinear ion-acoustic waves in the presence of a uniform magnetic field (see, for instance, \cite{ZakharovKuznet1974}). 
The generalized ZK equation (gZK, $m \neq 2$) offers a rich variety of traveling solitary wave phenomena (stability/instability, depending on $m$ and the dimension), e.g., see \cite{Bouard1996, FHRY2020, FHR2019c, FHR2019b, KRS2020, KRS2021}.
The one-dimensional case of fKdV with $1<\alpha<2$ and $m=2$ is known as the dispersion generalized Benjamin-Ono equation \cite{LFA2014, FonsLinaresPonce2013}. In general, the fKdV equation in dimension $d=1$ has been used in a variety of wave phenomena, and the lower dimensional case $0<\alpha<1$ has been investigated as a model to measure the influence of dispersive effects on the dynamics of Burger's equation. For other studies on the $1d$ fKdV equation, for example, see \cite{AnguloBonaLinaresScialom2002, KenigPonceVega1991, Argenis2020, 
Riano2021, Argenis2020I, RWY2021}, 
for the higher dimensional fKdV, see \cite{Oscar2020,HLORW2019,Schippa2020,OscSvetKai2021}, and reference therein. 

Solutions of \eqref{EQ:fKdV} formally satisfy three conservation laws: the $L^2$-norm (or sometimes referred as the momentum or mass) conservation
\begin{align}\label{E:mass}
M[u(t)] \defeq \frac{1}{2}\int_{\mathbb{R}^d} |u(x,t)|^2 \, dx =M[u(0)],
\end{align}
the energy (or Hamiltonian) conservation 
\begin{align}\label{E:energy}
E[u(t)] \defeq \frac{1}{2} \int_{\mathbb{R}^d} |D^{\frac{\alpha}{2}}u(x,t)|^2 \, dx  -\frac{1}{m(m+1)} \int_{\mathbb{R}^d} \big(u(x,t)\big)^{m+1} \,  dx  = E[u(0)],
\end{align}
and the time invariant $L^1$-type integral in $x_1$-direction
\begin{equation}\label{E:l1inte}
\int_{\mathbb{R}} u(x_1,\dots,x_d,t)\, dx_1 =\int_{\mathbb{R}} u(x_1,\dots,x_d,0)\, dx_1,
\end{equation}
which can also be stated in the $d$-dimensional form as
\begin{equation}\label{E:l1inte-g}
\int_{\mathbb{R}^d} u(x,t)\, dx =\int_{\mathbb{R}^d} u(x,0)\, dx.
\end{equation} 
The equation \eqref{EQ:fKdV} is invariant under the scaling 
$$
u_{\lambda}(x,t)=\lambda^{\frac{\alpha}{m-1}} u(\lambda x,\lambda^{1+\alpha} t)
$$
for any positive $\lambda$. Thus, \eqref{EQ:fKdV} is invariant in the Sobolev space $\dot{H}^{s_{c}}(\mathbb{R}^d)$ with 
\begin{equation}\label{CritiInd}  
s_{c}=\frac{d}{2}-\frac{\alpha}{m-1}.
\end{equation}
The critical index $s_{c}$ is convenient for classification of the equation \eqref{EQ:fKdV}  according to the values of $m>1$. 
Namely, when $1<m<\frac{2\alpha}{d}+1$ ($s_{c}<0$), the equation \eqref{EQ:fKdV} is referred to as the $L^2$-subcritical; if $m=\frac{2\alpha}{d}+1$ ($s_{c}=0$), it 
is $L^2$-critical; when $m>\frac{2\alpha}{d}+1$ ($s_{c}>0$), the equation \eqref{EQ:fKdV} is $L^2$-supercritical. We also note that it is energy-critical if $m=\frac{\alpha+d}{d-\alpha}$ (or $s_c=\frac{\alpha}{2}$).

Regarding the well-posedness of the Cauchy problem for \eqref{EQ:fKdV} in Sobolev spaces $H^s(\mathbb{R}^d)$ (a Cauchy problem is well-posed in a given function space if 
existence, uniqueness of solution, and continuous dependence of the data-to-solution flow map are established), we recall the following results. 

$\blacklozenge$ 
{Dimension $d=1$:}

The initial value problem for \eqref{EQ:fKdV} was initially considered by Kenig, Ponce and Vega in \cite{KenigPonceVega1991} in the range $1 \leq \alpha <2^{\,}$\footnote{In \cite{KenigPonceVega1991} for $\alpha=2$ the lwp for $s>\frac34$ and gwp for $s  \geq 1$ are obtained, however, we don't consider this case here. }: for $m=2$ the local well-posedness  was obtained in $H^{s}(\mathbb{R})$ for $s>\frac{9-3\alpha}{4}$ and global well-posedness in the energy space $H^{\frac{\alpha}{2}}(\mathbb{R})$ for $\alpha \geq \frac{9}{5}$ (and $\alpha<2$); 
the same authors in \cite{KenigPonceVega1994} showed 
the local well-posedness for $m=3$ in $H^s$, $s \geq \frac{7-3\alpha}4$; for 
$m=4$ in $H^s$, $s \geq \frac{19-9\alpha}{12}$ and their approach generalizes for $m \geq 4$. Several improvements of these results have been done over the recent years, below we state the best known results, splitting them into three cases $0<\alpha<1$, $\alpha=1$ and $1<\alpha<2$.

\underline{Case:} $ 1<\alpha<2$ 
\begin{itemize}
\item[$\bullet$]
$m=2$: Herr, Ionescu, Kenig, and Koch \cite{HerrIoneKeniKoch2010} improved the global well-posedness of \eqref{EQ:fKdV} down to\footnote{For a simpler approach 
in this range of $\alpha$, see \cite{MolinetVento2015}, though that result is a little weaker: $s \geq 1-\frac{\alpha}2$. } $L^2(\mathbb{R})$. Recently, the local well-posedness was improved in \cite{AiLiu2024} to $s>\frac34(1-\alpha)$.

\item[$\bullet$]  
$m=3$: Guo \cite{Guo2012} obtained the local well-posedness in $H^{s}(\mathbb{R})$ for $s\geq \frac{3-\alpha}{4}$ and global well-posedness for $s\geq \frac{\alpha}{2}$.\footnote{For further progress on global existence and scattering in this case, see \cite{KleinSautWang2022,SautWang2021}.}    

\item[$\bullet$] 
$m>2\alpha +1$: Farah, Linares, and Pastor \cite{LFA2014} obtained the global solutions  
in the energy space $H^{\frac{\alpha}{2}}(\mathbb{R})$ with mass-energy-gradient restrictions on the initial data.
\end{itemize}

\underline{Case:}
$\alpha=1$  
\begin{itemize}
\item[$\bullet$]
$m=2$ (BO): the original question of well-posedness in $H^s(\mathbb R)$ was investigated by Saut \cite{Saut1979}, Iorio \cite{Iorio1986}, Ginibre-Velo \cite{GV89}, Abdelouhab et al \cite{ABFS1989}, with improvements via local smoothing by Ponce \cite{Ponce1991} (see also further developments in \cite{KochTzvetkov2003, KenigKoenig2003}). A significant advance using a gauge transform was by Tao \cite{Tao2004} obtaining global well-posedness in $H^1(\mathbb R)$ (with further advances in \cite{BurqPlanchon2008}). The global well-posedness in $L^2(\mathbb{R})$ was first obtained by Ionescu and Kenig \cite{KenigIonescu2007}, with several alternative methods and proofs in
\cite{molinetPilodBO},\cite{IfrimTata2019} and \cite{Talbut2021}. The 
sharp well-posedness in $H^s(\mathbb{R})$, $s>-\frac{1}{2}$, was recently established by Killip, Laurens, and Vi{\c s}an, \cite{KillipLaurensVisan2024} (see also ill-posedness results in \cite{BiagioniLinares2001}, \cite{KochTzvetkov2005}).  

\item[$\bullet$]
$m=3$: Kenig and Takaoka \cite{KenigTakaoka2006} showed the global well-posedness in the energy space $H^{\frac{1}{2}}(\mathbb{R}^d)$. 

\item[$\bullet$]
$m=4$: Vento \cite{Vento2010}, Molinet and Ribaud  \cite{MolinetRibaud2004} established local well-posedness in $H^{s}(\mathbb{R})$, $s>\frac{1}{3}$. 

\item[$\bullet$]
$m\geq 5$: same authors in \cite{Vento2010} and  \cite{MolinetRibaud2004} established local wellposedness for $s\geq \frac{1}{2}-\frac{1}{m-1}$.
\end{itemize}

\underline{Case:}
$0<\alpha<1$
\begin{itemize}
\item[$\bullet$]
$m=2$: initially the local well-posedness was established by Linares, Pilod and Saut \cite{LinaresPilodSaut2014} in $H^{s}(\mathbb{R})$ with $s >\frac{3}{2}-\frac{3\alpha}{8}$, and a few years later was improved by Molinet, Pilod and Vento \cite{MolPilodVen2018} to $s > \frac{3}{2}-\frac{5\alpha}{4}$, which allowed to obtain the global well-posedness in the energy space $H^{\frac{\alpha}{2}}(\mathbb{R})$ for $\frac{6}{7}<\alpha<1$ 
(it is conjectured that in this case the global well-posedness holds for all $\frac{1}{2}<\alpha<1$).
\end{itemize}

$\blacklozenge$
{Dimension $d \geq 2$}:

By a standard parabolic regularization argument (e.g., \cite{ABFS1989,Iorio1986}, see also the proof of Lemma \ref{L:Regular}), the local well-posedness holds in a general case of $d\geq 1$, $0< \alpha <2$, $m\geq 2$ ($m$ is an integer) in $H^s(\mathbb R^d)$ for $s>\frac{d}{2}+1$. To the best of our knowledge, when $d\geq 2$, there are no results addressing the global well-posedness of the Cauchy problem \eqref{EQ:fKdV}.
For $s \leq \frac{d}2+1$, the known local well-posedness results are the following.

\underline{Case:} $1\leq \alpha <2$:
\begin{itemize} 
\item[$\bullet$]
$m=2$:  when $\alpha=1$, Hickman, Linares et al. \cite{HLORW2019} studied the local well-posedness in $H^s(\mathbb{R}^d)$, $s>\frac{5}{3}$ if $d=2$, and $s>\frac{d}{2}+\frac{1}{2}$ if $d\geq 3$. Later, for $\alpha=1$ and $d=2$ , the local well-posedness was improved by Schippa \cite{Schippa2020}, who also extended the local well-posedness to $1\leq \alpha<2$ in $H^s(\mathbb R^d)$ with $s>\frac{d+3}{2}-\alpha$.  
\end{itemize}

We also mention that for $d\geq 1$ and $0<\alpha<2$, the well-posedness of the fKdV equation has been studied in {\it weighted} spaces $H^{s}(\mathbb{R}^d)\cap L^2(\omega(x)\, dx)$ with certain weights $\omega(x)$, e.g., $(1+|x|)^a$, 
see  \cite{Iorio1986, FP2011, FonsLinaresPonce2013, Oscar2020, Riano2021}.

Now, that we have reviewed the existence of solutions, we discuss a specific type of solutions, traveling waves, which preserve their shape. 
The {\it traveling solitary wave} solutions 
for the equation \eqref{EQ:fKdV} are of the form 
\begin{equation}\label{travequ}
u(x_1,\dots,x_d,t)=Q_c(x_1-c\,t,x_2,\dots,x_d), ~~c>0.
\end{equation}
Note that the solitary waves travel to the right in the (positive) $x_1$-direction, since $c>0$.
Substituting \eqref{travequ} 
into \eqref{EQ:fKdV}, we deduce that $Q_c$ satisfies 
\begin{equation}\label{EQ:groundState}
c\,Q_c+D^{\alpha}Q_c-\tfrac{1}{m}Q^{m}_c=0,
\end{equation}
where 
\begin{equation}\label{E:Qscaled}
Q_c(x)=c^{\frac{1}{(m-1)}} Q(c^{\frac{1}{\alpha}}x), ~~ x\in \mathbb{R}^d    
\end{equation} 
with $Q \equiv Q_1$ solving \eqref{EQ:groundState} with $c=1$. The equation \eqref{EQ:groundState} has many solutions, the one which is relevant to the traveling solitary waves is a non-trivial, positive, radial $H^{\frac{\alpha}{2}}$ (thus, vanishing at infinity) solution called a ``ground state". 
The existence of ground state solutions (via minimization of a certain functional) goes back to the work of Weinstein \cite{W87}, \cite{W83}, Albert, Bona and Saut \cite{ABS1997}, see also Strauss \cite{S77} and Beresticky and Lions \cite{BL83}. The uniqueness in the fractional case of $\alpha$ (and decay properties) were obtained by Frank and Lenzmann in \cite{FL2013} for the $1d$ case and by Frank, Lenzmann and Silvestre in \cite{FLS2016} in higher dimensions (see also \cite{Maris2002} for the uniqueness when $\alpha=1$ in dimension $d=2$ and Kwong \cite{Kw1989} for $\alpha=2$ and references therein to prior work). For brevity of the introduction, we provide a more detailed discussion about the ground state in Section \ref{sectionPrelimDecay}.

In this work, we investigate {\it stability} and {\it instability} of solitary wave solutions to the equation \eqref{EQ:fKdV} in {\it subcritical} and {\it supercritical} regimes, in {\it any} spatial dimension. We analyze these (in)stability properties according to the critical scaling index $s_c$ \eqref{CritiInd}, which depends on the nonlinearity $m$, the dispersion parameter $\alpha$, and the dimension $d\geq 1$. More precisely, in the subcritical case, $1<m<\frac{2\alpha}{d}+1$, we show the {\it orbital stability} of solitary waves, and in the supercritical case, $\frac{2\alpha}{d}+1<m<m_{\ast}$ (for definition of $m_{\ast}$ see \eqref{defcritm} below), 
we prove their instability, providing an example of data which shows the instability.   
We emphasize that the construction of the unstable solution is new in the nonlocal dispersion case, where we invent several approaches to handle nonlocal differential operators. The proof of stability relies on Albert's method based on concentration compactness and extra handling of nonlocal operators as well. Before providing further details and comments, we give notation and definitions.

We begin with the definition of orbital stability and then instability.  For a fixed $c>0$ and a given $\omega>0$, we define a neighborhood (or ``tube") of radius $\omega$ around $Q_c$ (modulo translation) by
\begin{equation}\label{tubdef}
U_{\omega}=\{u\in H^{\frac{\alpha}{2}}(\mathbb{R}^d): \, \inf_{z\in \mathbb{R}^d}\|u(\cdot)-Q_c(\cdot+z)\|_{H^{\frac{\alpha}{2}}}\leq \omega\}.
\end{equation}

\begin{definition}\label{orbStabdef}
Let $c>0$. Then $Q_c$ is {\rm orbitally stable} in $H^{\frac{\alpha}{2}} (\mathbb{R}^d)$ if for any $\omega>0$ there exists $\delta>0$ such that $u_0 \in U_{\delta}$ implies that the solution $u(t)$ with $u(x,0)=u_0$ is in $U_{\omega}$ for all $t\in \mathbb{R}$, i.e., 
\begin{equation}
\sup_{t \in \mathbb{R}} \inf_{z\in \mathbb{R}^d}\|u(\cdot,t)-Q_c(\cdot+z)\|_{H^{\frac{\alpha}{2}}}<\omega.
\end{equation}
\end{definition}

\begin{definition}\label{orbInstabdef}
Let $c>0$. Then 
$Q_c$ is {\rm unstable} if there exists $\omega>0$ such that for all $\delta>0$ the following holds: if $u_0 \in U_{\delta}$, then there exists $t_0=t_0(u_0)>0$ such that $u(t_0)\notin U_{\alpha}$.
\end{definition}

Our first result establishes the orbital stability of traveling solitary waves $Q_c$ as in \eqref{travequ}-\eqref{EQ:groundState}-\eqref{E:Qscaled} in the range\footnote{We note that our methods are applicable for the full range of $\alpha \in (0,2)$, though due to the integer restriction on $m$ in the nonlinearity $u^{m-1}u_{x_1}$, the range is shortened (both in $\alpha$ and $m$). When using the nonlinearity with the absolute value $|u|^{m-1} u_x$, we get the full range.} of $0< \alpha<2$.

\begin{theorem}\label{OrbstabilThm}
Let $0< \alpha<2$, $d \geq 1$ and 
$c>0$. Take an integer $m$ such that 
$1<m<\frac{2\alpha}{d}+1$.
For any $\omega>0$ there exists $\delta>0$ such that if $u_0 \in H^s(\mathbb{R}^d)$, $s>\frac{d}{2}+1$, satisfy
\begin{equation}
\|u_0-Q_c\|_{H^{\frac{\alpha}{2}}}<\delta,
\end{equation}
then the solution $u(x,t)$
to \eqref{EQ:fKdV} with  
$u(x, 0)= u_0$ satisfies
\begin{equation}\label{E:stab-conclusion}
\sup_{t \in \mathbb R} \inf_{z\in \mathbb{R}^d} \|u(\cdot,t) - Q_c(\cdot+z) \|_{H^{\frac{\alpha}{2}}}< \omega.
\end{equation}
\end{theorem}
We remark that in those cases, where the global well-posedness is not yet available, the supremum over all $t \in \mathbb R$ in \eqref{E:stab-conclusion} should be understood as $t \in (-T_\ast, T^{\ast})$, 
the maximal time interval of existence of the solution $u(t)$. To clarify even further, the imposed conditions on the range of $\alpha$ as well as $m$ being integer (i.e., $0< \alpha<2$ and $1<m<\frac{2\alpha}{d}+1$, $m \in \mathbb N$) make our stability result to be valid in dimensions $d=1, 2, 3$ with $\frac{d}{2}<\alpha<2$. 
When the nonlinearity is modified to allow non-integer $m$, Theorem \ref{OrbstabilThm} should still hold; for example, in Remark \ref{remarkstabilianym} below, we can consider an arbitrary $m>1$ and obtain orbital stability for {\it any} arbitrary dimension and dispersion $\alpha\in(0,2)$ (as long as the equation is $L^2$-subcritical).

The proof of  Theorem \ref{OrbstabilThm} is based on the concentration-compactness ideas of Lions \cite{LionsI1984,LionsII1984}, which were developed in the KdV context by Albert \cite{Albert1999} for the equation \eqref{EQ:fKdV}, $1\leq \alpha <2$, $d=1$ and integer $m \geq 2$, 
and extended by Linares, Pilod and Saut \cite{LinPilodSaut2015} 
for $0< \alpha<1$, $d=1$, and $m=2$. 
In this work we prove the orbital stability of solitons in a general setting (including higher dimensions), and we stress that our results are independent of the spatial dimension $d\geq 1$ as long as the conditions $0< \alpha<2$, $1<m<\frac{2\alpha}{d}+1$, $m \in \mathbb N$ are met (i.e., in $d=1,2,3$). The idea of proving the Theorem \ref{OrbstabilThm} is to study the minimizers of the energy functional \eqref{E:energy} subject to a fixed mass constrain (the existence of such minimizers is known, see \cite{FLS2016}).  
Then the concentration-compactness technique gives the orbital stability of the set of these minimizers, implying the orbital stability of the solitary waves of \eqref{EQ:fKdV}. 
One major difficulty that we are faced in proving this general result, including in higher dimensions, is the presence of the {\it nonlocal} operator $D^{\alpha}$ in several variables. We overcome this issue by carefully using the commutator estimates and expansions of the operator $D^{\alpha}$, as stated in Proposition \ref{fractionalDeriv}. 
Several remarks are due.  

\begin{remark}
The stability of solitons obtained in Theorem \ref{OrbstabilThm} holds on the known maximal interval of existence of the perturbation $u$, and thus, conditional up to the global well-posedness in $H^{s}(\mathbb{R}^d)$, $s\leq \frac{\alpha}{2}$. 
For simplicity, we stated Theorem \ref{OrbstabilThm} for initial data $u_0\in H^s(\mathbb{R}^d)$, $s>\frac{d}{2}+1$, however, the same conclusion is valid for {\it any} subcritical fKdV, which is well-posedness in $H^s(\mathbb{R}^d)$, $s\geq \frac{\alpha}{2}$. 
In particular, when $d=1$, due to the available global theory, Theorem \ref{OrbstabilThm} recovers the (unconditional) orbital stability of solitary waves from \cite{BonaSouganidisStrauss1987,Albert1992,BBSSB1983,AlbertBona1991, 
Benjamin1972,AlbertBonaHenry1987,Albert1999,LinPilodSaut2015}. More precisely, the following is the consequence of the proof of Theorem \ref{OrbstabilThm} in $1d$. 
\end{remark}

\begin{corollary}
Let $d=1$.
\begin{itemize}
    \item[(a)]  For $m=2$ consider $\frac{6}{7}<\alpha<2$ in \eqref{EQ:fKdV}.
    \item[(b)]  For an integer $m$ such that $3\leq m <2\alpha+1$, consider $\max\{\frac{m-1}{2},\frac{9}{5}\}<\alpha<2$ in \eqref{EQ:fKdV}.
\end{itemize}
Under either (a) or (b), $Q_c$ is orbitally stable in $H^{\frac{\alpha}{2}}$ (in the sense of Definition \ref{orbInstabdef}).  
\end{corollary}
\smallskip

\begin{remark}\label{remarkstabilianym}
All arguments developed in the proof of Theorem \ref{OrbstabilThm} are valid for  $0<\alpha<2$, $1<m<\frac{2\alpha}{d}+1$, where $m$ is not necessarily an integer, in the following sense. Consider the Cauchy problem associated to the equation
\begin{align}\label{EQ:fKdV2}
\qquad \partial_t u- \partial_{x_1} D^{\alpha}u+  \frac{1}{m}\partial_{x_1}(|u|^{m-1}u)=0, ~ (x_1,\dots,x_d)\in \mathbb{R}^d, \, \, t\in \mathbb{R}, \, \, 0<\alpha <2.
\end{align}
(The one dimensional version of \eqref{EQ:fKdV2} with $\alpha=2$ was studied by the authors and collaborators in \cite{FRRSY}, where \eqref{EQ:fKdV2} was termed as GKdV to distinguish it from the standard gKdV equations without absolute value in the nonlinearity; there, a local well-posedness was obtained on some subset of a weighted Sobolev space).   
The following is a consequence of Theorem \ref{OrbstabilThm} for \eqref{EQ:fKdV2} with a non-integer $m$.
\end{remark}

\begin{corollary}
Let $0<\alpha<2$, $1<m<\frac{2\alpha}{d}+1$, and $c>0$. Assume that the Cauchy problem associated to \eqref{EQ:fKdV2} is locally well-posed in $H^{s_0}(\mathbb{R}^d)$ for some $s_0\geq \frac{\alpha}{2}$. For any $\omega>0$ there exists $\delta>0$ such that if $u_0 \in H^{s_0}(\mathbb{R}^d)$, satisfy
\begin{equation}
\|u_0-Q_c\|_{H^{\frac{\alpha}{2}}}<\delta,
\end{equation}
then the solution $u(x,t)$ to \eqref{EQ:fKdV} with $u(x,0)=u_0$ satisfies
\begin{equation}\label{E:GKdV-supinf}
\sup_{t \in \mathbb {R}} \inf_{z\in \mathbb{R}^d}\|u(\cdot,t)-Q_c(\cdot+z)\|_{H^{\frac{\alpha}{2}}}<\omega.
\end{equation}
\end{corollary}
(Similar to the note after Theorem \ref{OrbstabilThm}, the supremum on time $t$ should be understood conditionally, on the 
maximal time interval of existence of $u(t)$).

The assumption  of the local well-posedness in $H^{s}(\mathbb{R}^d)$ for the Cauchy problem associated to \eqref{EQ:fKdV2} is not always guaranteed, since the lack of regularity of the nonlinearity $|u|^{m-1}$ for the fractional $m>1$ yields extra difficulty when applying classical methods used for the well-posedness. Nevertheless, for results in this direction for \eqref{EQ:fKdV2} in 1$d$, refer to \cite{FRRSY} and \cite{LinaresMiyazakiPonce2019}.

\begin{remark}
The spectral stability in a general KdV-type setting is discussed by D. Pelinovsky in \cite{Pelinovsky2014}. 
A verification of a certain index\footnote{Sometimes called the Vakhitov-Kolokolov index.} being negative 
is shown by Angulo in \cite{Angulo2018}, i.e., $(\mathcal L^{-1}_c Q_c, Q_c)<0$, where $\mathcal L_c=D^\alpha+c-Q_c^{m-1}$. 
This is used then to conclude orbital stability via the strategy from \cite{W87}, 
however, as explained in \cite{ABS1997} additional steps have to be verified in the case of nonlocal dispersion or nonlocal operators for that strategy to be valid, and thus, a concentration-compactness argument is developed instead in \cite{ABS1997} to show the orbital stability. This is the approach we use in this paper. Certainly, the negativity of the index implies spectral stability in the subcritical ($s_c<0$) fractional gKdV. 
\end{remark}

Our next goal is to address the instability of solitons (as stated in Definition \ref{orbInstabdef}) in the supercritical case, $s_c >0$. We first introduce a specific sequence of initial data and then state the  result. For each $n \in \mathbb N$, let
\begin{equation}\label{defseqe}
u_{0,n}(x)=\lambda_n Q_c(\lambda_n^{\frac{2}{d}} x), \text{ where }   \lambda_n=1+\tfrac{1}{n}.
\end{equation}

\begin{theorem}\label{InstMainThm}
Let $\alpha \in \mathbb R$ be such that $\max\{1-\frac{d}{2},0\}<\alpha<2$ and 
let $m \in \mathbb N$ be such that $\frac{2\alpha}{d}+1<m< m_{\ast}$. Let $u_n$ be the solution of \eqref{EQ:fKdV} with initial data $u_{0,n}$ for each $n \in \mathbb N$. Then there exist $\omega>0$ such that for every $n \in \mathbb N$, there exists 
$T_n=T_n(u_{0,n})>0$ such that $u_n(T_n)\notin U_{\omega}$, or explicitly,
\begin{equation}
\inf_{z\in \mathbb{R}^d}\|u_n(\cdot,T_n)-Q_c(\cdot-z)\|_{H^{\frac{\alpha}{2}}}\geq \omega \quad \mbox{for ~ some}~ c >0.
\end{equation}
\end{theorem}
The proof of Theorem \ref{InstMainThm} uses the explicit sequence \eqref{defseqe} of the  initial data to obtain the contradiction to the stability of $Q_c$ in the $L^2$-supercritical setting. The reason we use that form in \eqref{defseqe} is to avoid unconditional global existence of solutions $u_n(t)$. Overall, our proof combines several techniques such as truncation, modulation, Weinstein's functional, linearized equation, among others, which can be useful in other settings of dispersive equations involving nonlocal operators. (We note that in our proof we don't need to use monotonicity).

The technical condition on the dispersion $\max\{1-\frac{d}{2},0\}<\alpha<2$ in Theorem \ref{InstMainThm} is only relevant in dimension $d=1$, where it reduces to  $\frac{1}{2}<\alpha<2$. Note that for any dimension $d\geq 2$, Theorem \ref{InstMainThm} works for the full range of the dispersion $0<\alpha<2$, and integer power $m$ with $\frac{2\alpha}{d}+1<m<m_{\ast}$. The restriction $\alpha>1-\frac{d}{2}$ is technical, and it appears to assure that $x_1Q\in L^2(\mathbb{R}^d)$.

The main idea of the proof of Theorem \ref{InstMainThm} is inspired by the instability results in \cite{FHR2019b} obtained for Zakharov-Kuznetsov equation in the $L^2$-supercritical setting: $\alpha = 2$, $d\geq 2$, $s_c>0$ in \eqref{EQ:fKdV} (see also Combet \cite{Combet2010}). 
\footnote{Originally, the instability in the supercritical ZK equation was obtained by de Bouard \cite{Bouard1996} following the ideas of  Bona-Souganidis-Strauss \cite{BonaSouganidisStrauss1987} instability in the one-dimensional generalized KdV setting.} 

However, in contrast with the previous works on gKdV or gZK, in our case the ground state solution of \eqref{EQ:fKdV} has a spatial decay of polynomial order (not of exponential decay), see Theorem \ref{existTHR} below. This brings several additional difficulties to obtain instability as our estimates require tracking a precise connection of polynomial decay with fractional derivatives. A key {\it new} ingredient is the deduction of regularity and spatial decay properties of the unique positive radial ground state solution $Q$ (with speed $c=1$) when $m$ is an integer. More precisely, for any dimension $d\geq 1$, $0<\alpha<2$, and $m>1$ integer, we show (see Lemma \ref{LemmaQprop}) that 
\begin{equation}\label{decayQIn}
|\partial^{\beta}Q(x)|\lesssim (1+|x|)^{-(\alpha+d)-|\beta|},
\end{equation}
for any multi-index $\beta$. This result is obtained by a careful analysis of the equation \eqref{NLEQ}, and the distributional kernel arising from it, see the proof of Lemma \ref{LemmaQprop} below and the proof of Lemmas \ref{decayQm}, \ref{conmident} and \ref{kernDecay} in Appendix. We remark that the arguments developed here are different from those used by Kenig, Martel, and Robbiano \cite[Proposition 1]{KenigMartRobb2011} for the case $d=1$, $1<\alpha<2$, $m=2\alpha+1$ in \eqref{NLEQ}, and those of Li and Bona \cite{LiBona1997}. We also note that in the case $d=1$, $1<\alpha<2$, $m=3$, the estimate \eqref{decayQIn} is the same as the one obtained recently by 
Eychenne and Valet \cite{EychenneValet2022}, who for that specific case deduced a more detailed asymptotic expansion of the ground state $Q$ as $|x|\to \infty$; that approach only applicable in 1d. Our estimate \eqref{decayQIn} covers all integers $m>1$ and also works in higher dimensions. We believe that our arguments in Lemma \ref{LemmaQprop} can be adapted to the case when $m$ is fractional  and to more general types of dispersion given by nonlocal operators. Moreover, as a further consequence of the proof of Lemma \ref{LemmaQprop}, we obtain several decay properties for the radial symmetric eigenfunction associated to the single negative eigenvalue of the operator \eqref{linearizedop}, see Lemma \ref{LemmaChiprop}.

\begin{remark}
The currently available local well-posedness theory 
for dimension $d\geq 1$, dispersion $\alpha>0$, and nonlinearity power $m > 1$, 
does not always guarantee
the existence of solutions of \eqref{EQ:fKdV} for arbitrary initial data in $H^{\frac{\alpha}{2}}(\mathbb{R}^d)$ (unlike the $L^2$-subcritical case of gKdV). However, in the case of the sequence $\{u_{0,n}\}$ given by \eqref{defseqe}, we can always obtain the existence of solutions of \eqref{EQ:fKdV} with the initial condition \eqref{defseqe}. This is a consequence of Lemma \ref{LemmaQprop}, which implies that the sequence $\{u_{0,n}\}\in H^{s}(\mathbb{R}^d)$ for all $s\in \mathbb{R}$, thus, we can use classical methods of local well-posedness to assure the existence of a solution $u_n(t)$ of \eqref{EQ:fKdV} with the initial condition $u_{0,n}$.  
\end{remark}

\begin{remark} 
Our proof of instability strongly depends on the fact that $m>1$ is an integer. It would be interesting to extend Theorem \ref{InstMainThm} to the case of a non-integer $m$, that is, the equation \eqref{EQ:fKdV2}.
\end{remark}

\begin{remark}\label{remarkinsta bility}
This is the first result of soliton instability (in the sense of Definition \ref{orbInstabdef}) for the fractional generalized KdV equation \eqref{EQ:fKdV} for $\alpha<2$, where we provide an explicit example of the initial conditions leading to the instability (thus, a truly nonlinear behavior: starting as close as possible to the soliton and escaping its orbit up to the shifts as far as possible in finite time; we show that via a constructive proof). In one dimension $d=1$, in the fractional setting the statements concerning {\it spectral} instability of solitons can be found in  \cite{Lin2008,KapitulaStefanov2014,Angulo2018,LinZeng2022}. We note that some statements about soliton instability are given in \cite[Theorem 10.19]{Angulo2009} with some ideas of the proofs outlined in \cite{Angulo2003}, however, that only shows how to verify the sign (positive in this case) of the VK (Vakhitov-Kolokolov) index, which is insufficient for the full proof of nonlinear instability and as mentioned in \cite{VWW2020} one has to be especially careful in the fractional case due to the slow decay of the tails of the ground states, and hence, more steps would be required to obtain the instability via those methods. The proof of Theorem \ref{InstMainThm} is given via a different, constructive approach (as in \cite{FHR2019a, FHR2019b} for the local dispersion equations such as the gKdV and gZK equations) for the $L^2$-supercritical fractional gKdV equation and its generalization to higher dimensions, with $m$ being an integer in \eqref{EQ:fKdV}, the only restriction in this paper and to keep things clearer for the reader (we do mention the non-integer $m$ case in the appropriately modified gKdV equation \eqref{EQ:fKdV2}). 
\end{remark}

The paper is organized as follows. In Section \ref{sectionPrelimDecay}, we review the ground state and deduce some key properties concerning spatial decay and regularity of solutions of \eqref{EQ:groundState}. 
The orbital stability result of Theorem \ref{OrbstabilThm} is proved in Section \ref{S:OrbitalS}. There, we use a concentration-compactness approach for a constrained variational problem and show the stability of the minimizers of that problem (Theorem \ref{stabilTheorem-2}). Then in Lemma \ref{caraGround} we observe that these minimizers are the same minimizers as for the unconstrained problem of the Gagliardo-Nirenberg inequality (as discussed in Section \ref{sectionPrelimDecay}), which are the ground states $Q_c$. 
The next two sections deal with the deduction of instability of solitary wave solutions of fKdV, i.e., the proof of Theorem \ref{InstMainThm}. In Section \ref{InstaPrelimC}, we first review useful properties of the linearized operator $L$ in \eqref{linearizedop} as well as some properties of the eigenfunction associated to the unique negative eigenvalue of $L$. We then deduce the modulation theory for solutions $u$ of \eqref{EQ:fKdV} close to $Q$, obtain some virial type estimates and introduce a virial-type functional \eqref{Jfunct}, which will be the main tool in showing the instability of traveling solitary waves in the supercritical case. Section \ref{SectInstabProof} collects the results of the previous sections to prove the instability  Theorem \ref{InstMainThm}, where we use the explicit sequence \eqref{defseqe}. We conclude with Appendix, where we establish the proof of Lemmas \ref{decayQm}, \ref{conmident} and \ref{kernDecay}, which are fundamental in obtaining the spatial decay properties of $Q$ in Lemma \ref{LemmaQprop}.


{\bf Notation.}
Given two positive numbers $a$ and $b$,  $a\lesssim b$ means that there exists a positive constant $c>0$ such that $a\leq c b$. We write $a\sim b$ to symbolize that $a\lesssim b$ and $b\lesssim a$. We will use the standard multi-index notation, $\beta=(\beta_1,\dots,\beta_d) \in (\mathbb{Z}\cup\{0\})^d$, $\partial^{\beta}=\partial^{\beta_1}_{x_1}\cdots \partial_{x_d}^{\beta_d}$, $|\beta|=\sum_{j=1}^d \beta_j$,  and $\gamma \leq \beta$ if $\gamma_j \leq \beta_j$ for all $j=1,\dots,d$. $[A,B]$ denotes the commutator operator associated to $A$ and $B$, i.e.,
$$[A,B]=AB-BA.$$
We define $\langle x\rangle=(1+|x|^2)^{\frac{1}{2}}$, $x\in \mathbb{R}^d$. $C^{\infty}_c(\mathbb{R}^d)$ denotes the set of smooth functions with compact support and $\mathcal{S}(\mathbb{R}^d)$ denotes the Schwartz class functions. Given $s\in \mathbb{R}$, the Bessel operator $J^s$ is defined via the Fourier transform by the multiplier $\langle \xi \rangle^{s}=(1+|\xi|^2)^{\frac{s}{2}}$. For arbitrary $s>0$, the Riesz potential $D^s$ is defined by the Fourier multiplier $|\xi|^{s}$.

{\bf Acknowledgments.} The research of both authors was partially funded by the NSF grants DMS-1927258 and DMS-2055130 (PI: S. Roudenko).

\section{Preliminaries on the ground state} 
\label{sectionPrelimDecay}

We consider $0<\alpha <2$ and  
$1<m<m_{\ast}$, where
\begin{equation}\label{defcritm}
m_{\ast}=\left\{
\begin{aligned}
&\frac{d+\alpha}{d-\alpha},  \hspace{.3cm} \mbox{if} \quad 0<\alpha<d, \\
& +\infty,    \hspace{.5cm}   \mbox{if} \quad \alpha \geq d. 
\end{aligned}
\right.
\end{equation}
This range of parameters ensures that the equation \eqref{EQ:fKdV} is energy-subcritical. This also guarantees that the following nonlinear elliptic equation
\begin{equation}\label{NLEQ}
 Q+D^{\alpha}Q-\tfrac{1}{m}Q^{m}=0,
\end{equation}
admits non-trivial solutions in $H^{\frac{\alpha}2}(\mathbb R^d)$, for example, via Pohozhaev identities, see Lemma \ref{PHident} and further details in \cite{OscSvetKai2021}  
(here, we are only interested in real-valued and more precisely positive solutions).\footnote{In the context of \eqref{EQ:fKdV2}, instead of the equation \eqref{NLEQ} one considers 
$ Q+D^{\alpha}Q-\tfrac{1}{m}|Q|^{m-1} Q=0$,
and their non-trivial real-valued $H^{\frac{\alpha}2} (\mathbb R^d)$ solutions.}

We recall the following fractional Gagliardo-Nirenberg inequality 
(\cite[p.168]{BL1976}, \cite[Theorem 1.3.7]{Caz-book}, \cite{FLS2016}), which shows that any $f \in H^{\frac{\alpha}2} (\mathbb R^d)$ also belongs to $L^{m+1}(\mathbb R^d)$ by interpolation:
\begin{equation}\label{GNineq}
   \|f\|_{L^{m+1}(\mathbb{R}^d)}^{m+1} \leq C_{GN} \|D^{\frac{\alpha}{2}}f\|_{L^2(\mathbb{R}^d)}^{\frac{d(m-1)}{\alpha}}\|f\|_{L^2(\mathbb{R}^d)}^{(m+1)-\frac{d(m-1)}{\alpha}}.
\end{equation}
The sharp constant $C_{GN}$ is obtained from minimizing 
the Weinstein functional 
\begin{equation}\label{Wfunctional}
   \mathcal{W}(u) \stackrel{def}{=}\frac{\| D^{\frac{\alpha}{2}}u\|_{L^2(\mathbb{R}^d)}^{\frac{d(m-1)}{\alpha}}\|u\|_{L^2(\mathbb{R}^d)}^{(m+1)-\frac{d(m-1)}{\alpha}}}{\|u\|_{L^{m+1}(\mathbb{R}^d)}^{m+1}}
\end{equation}
over $u\in H^{\frac{\alpha}{2}}(\mathbb{R}^d) \setminus \{0 \}$. If $Q$ is such a minimizer, i.e., $\mathcal W (Q) = \inf \{ \mathcal{W}(u): u \in H^{\frac{\alpha}{2}} \setminus \{0\} \}$, which can also be chosen non-negative since $\mathcal W(|u|) \leq \mathcal W (u)$, then the sharp constant in the Gagliardo-Nirenberg inequality \eqref{GNineq} is $C_{GN}^{-1} = \mathcal W(Q)$. 

The existence of minimizers can be obtained via the concentration-compactness arguments, see \cite{W87}, \cite{ABS1997}, or via symmetrization and Strauss radial lemma \cite[Theorem B]{W83}, \cite[Proposition 3.1]{FLS2016} (see also \cite[Lemma 4.1]{AR2022} in the context of nonlocal convolution nonlinearity instead of nonlocal dispersion). 

Note that such minimizers $Q \in H^{\frac{\alpha}{2}}(\mathbb R^d)$ of $\mathcal W (u)$ solve \eqref{NLEQ}, this guarantees the existence of non-negative solutions to the nonlinear elliptic equation \eqref{NLEQ}. Moreover, every minimizer $u \in H^{\frac{\alpha}{2}}(\mathbb{R}^d)$ of the Weinstein functional $\mathcal W$ is of the form $u = \beta \, Q(\gamma\, (\cdot+z))$ with some $\beta\in \mathbb{C}\setminus\{0\}$, $\gamma>0$ and $z\in \mathbb{R}^d$.
We also mention that several other constrained variational problems are equivalent to the unconstrained problem of minimizing the functional $\mathcal W$ on $H^{\frac{\alpha}2}(\mathbb{R}^d) \setminus \{0\} $. For example, minimizing the Hamiltonian in $H^{\frac{\alpha}2}$ with a fixed mass, a variation of this approach is used below to prove orbital stability (see Section \ref{S:OrbitalS}).
The following properties of such minimizers hold.  

\begin{proposition}[\cite{FLS2016}]\label{existTHR}
Let $0<\alpha<2$ and $1<m<m_{\ast}$, where $m_{\ast}$ is defined in \eqref{defcritm}. If a non-trivial non-negative $Q \in H^{\frac{\alpha}2}(\mathbb R^d)$ satisfies \eqref{NLEQ}, then there exists an $x_0\in \mathbb{R}^d$ such that $Q(\cdot-x_0)$ is radial, positive, and strictly decreasing in $|x-x_0|$. Additionally, the function $Q$ belongs to $H^{\alpha+1}(\mathbb{R}^d) \cap  C^{\infty}(\mathbb{R}^d)$ and satisfies
\begin{equation}\label{poldecayGS}
\frac{C_1}{1+|x|^{d+\alpha}} \leq Q(x) \leq \frac{C_2}{1+|x|^{d+\alpha}}
\end{equation}
for all $x \in \mathbb{R}^d$ with some constants $C_2 \geq C_1>0$ (which can depend on $\alpha$, $m$, $Q$). 
\end{proposition}

We now have positive radial $H^{\frac{\alpha}2}(\mathbb R^d)$ solutions to the \eqref{NLEQ}, which are called the {\it ground state} solutions, the next 
question to address  is the uniqueness of such solutions. 
The uniqueness is proved in \cite{FL2013} for the $1d$ case and in \cite{FLS2016} for any dimension. In fact, in \cite{FLS2016} a more general result is proved, via defining a weaker version of a ground state. More precisely, the minimizers that are discussed above are the global minimizers of the Weinstein functional. This notion can be relaxed to the local minimizers via Morse index, in the following sense. 

\begin{definition}(Weaker version of a ground state, \cite[Definition 3.2]{FLS2016})
For a real-valued solution  $Q \in H^{\frac{\alpha}{2}}(\mathbb R^d)$ of \eqref{NLEQ}, define the corresponding linearized operator $L$ acting on $L^2(\mathbb{R}^d)$ by
\begin{equation}\label{linearizedop}
L=D^{\alpha}+1-|Q|^{m-1}.
\end{equation}
Then a non-trivial non-negative $Q$ is referred to as (a weaker version of) a {\it ground state} solution of the equation \eqref{NLEQ} if $L$ has Morse index equal to $1$, i.e., $L$ has exactly one strictly negative eigenvalue (counting multiplicity).
\end{definition}
If $Q \geq 0$ is a local minimizer of $\mathcal W(u)$, then $L$ has Morse index 1 (see \cite{W85, FLS2016}), and thus, according to the above definition would be called a ground state. While the notion of a ground state now is more general, the uniqueness even in this more general case holds, which was shown in \cite{FLS2016} by proving the nondegeneracy of the operator $L$, namely, that the ker $L$ is spanned only by $\{\partial_{x_j} Q\}_{j=1, \dots, d}$.  

Observe that the ground state $Q$ in the equation \eqref{NLEQ} is $Q \equiv Q_1$, 
the uniqueness of the ground state solution to the rescaled version 
\eqref{EQ:groundState} follows (up to translation). 

For convenience we state the sharp constant of \eqref{GNineq} in terms of $Q_c$, the ground state solution of \eqref{EQ:groundState} with $c>0$:
\begin{equation}
C_{GN}=\frac{m\alpha(m+1) \, c^{\frac{2\alpha-d(m-1)}{2\alpha}}}{d^{\frac{d(m-1)}{2\alpha}} (m-1)^{\frac{d(m-1)}{2\alpha}} \big(2d-(d-\alpha)(m+1)\big)^{\frac{2\alpha-d(m-1)}{2\alpha}}} \frac{1}{\|Q_c\|_{L^2(\mathbb{R}^d)}^{m-1}}.
\end{equation}

In what follows we need some additional spatial decay and regularity properties of the ground state, which we prove next. 

\begin{lemma}\label{LemmaQprop}
Assume $0< \alpha<2$ and $1<m<m_{\ast}$ is an integer, where $m_{\ast}$ is defined in \eqref{defcritm}. Let 
$Q$ be a positive radial $H^{\frac{\alpha}2}$ solution of \eqref{NLEQ}. 
Then $\displaystyle Q\in H^{\infty}(\mathbb{R}^d) \equiv \bigcap\limits_{s>0} H^{s}(\mathbb{R}^d)$. Moreover, for any multi-index $\eta$,
\begin{equation}\label{decderQ}
    |\partial^{\eta} Q(x)|\lesssim \langle x \rangle^{-\alpha-d-|\eta|}.
\end{equation}
\end{lemma}

\begin{proof}
We begin by claiming that if for some $s>0$, $Q\in H^{s}(\mathbb{R}^d)$, then $Q^{m}\in H^s(\mathbb{R}^d)$ for any integer $m\geq 1$. By the spatial decay properties of $Q$ in \eqref{poldecayGS}, we observe that it suffices to show
\begin{equation}\label{sestimQ}
\begin{aligned}
\|D^{s}(Q^m)\|_{L^2}  \lesssim \|Q\|_{L^{\infty}}^{m-1}\|D^sQ\|_{L^{2}}.
\end{aligned}
\end{equation}
To prove \eqref{sestimQ}, we will use the fractional Leibniz rule
\begin{equation}\label{fLR}
\begin{aligned}
\|D^s(f g)\|_{L^2}\lesssim \|D^s f\|_{L^2}\|g\|_{L^{\infty}}+\|f\|_{L^{\infty}}\|D^s g\|_{L^2}, 
\end{aligned}
\end{equation}
which for arbitrary $s>0$ is a consequence of the results in \cite[Theorem 1]{GrafakosOh2014}. Thus, the previous identity yields
\begin{equation*}
\begin{aligned}
\|D^s(Q^m)\|_{L^2}=\|D^s(Q Q^{m-1})\|_{L^2}\lesssim \|D^s Q\|_{L^2}\|Q\|^{m-1}_{L^{\infty}}+\|Q\|_{L^{\infty}}\|D^{s}(Q^{m-1})\|_{L^2}. 
\end{aligned}
\end{equation*}
Consequently, if $m\geq 2$, by writing $Q^{m-1}=Q Q^{m-2}$, we can iterate the argument above, using \eqref{fLR} in each step until we arrive at \eqref{sestimQ}.

Now, the equation \eqref{NLEQ} and the fact that $Q\in H^{\alpha+1}(\mathbb{R}^d)$ yield the following identity
\begin{equation}\label{eqSNLFS}
D^{\beta}Q=\frac{1}{m}(1+D^{\alpha})^{-1}D^{\beta}(Q^m)
\end{equation}
for any $\beta\geq 0$. Given that $\alpha>0$, the right-hand side of \eqref{eqSNLFS} is at least $\alpha$ derivatives smoother than its left-hand side. For example, if $Q\in H^{k\alpha}(\mathbb{R}^d)$ for some integer $k\geq 0$, the estimate \eqref{sestimQ} shows that $Q^{m}\in H^{k\alpha}(\mathbb{R}^d)$, and in turn \eqref{eqSNLFS} yields  $Q\in H^{(k+1)\alpha}(\mathbb{R}^d)$. Thus, the fact that $Q\in H^{\infty}(\mathbb{R}^d)$ follows by an inductive argument on $k\geq 1$ with $\beta = k \alpha$ in \eqref{eqSNLFS}.

Next, we establish \eqref{decderQ}. It is enough to prove that for any integer $\kappa\geq 0$, it follows that
\begin{equation}
    |\partial^{\eta}Q(x)|\lesssim \langle x \rangle^{-\alpha-d-\kappa}\, \,  \text{ for all } \, \, |\eta|\geq \kappa.
\end{equation}
The above inequality is deduced by using an inductive argument on $\kappa \geq 0$ integer. We  start with the base case $\kappa=0$.

\underline{\bf Case $\kappa=0$}. We first recall the following lemma proved in \cite[Lemma C.3]{FLS2016}.
\begin{lemma}\label{lemmadeca}
Let $d\geq 1$, $K,f\in L^1(\mathbb{R}^d)$ such that $|K(x)|\lesssim |x|^{-r}$ and $|f(x)|\lesssim \langle x \rangle^{-r}$ for some $r>d$. Moreover, assume that
\begin{equation*}
    \begin{aligned}
    \lim_{|x|\to \infty}|x|^{r}K(x)=K_0 \, \, \text{ and }\, \, \lim_{|x|\to \infty} |x|^r f(x)=0.
    \end{aligned}
\end{equation*}
Then 
\begin{equation*}
    \lim_{|x|\to \infty}|x|^r(K\ast f)(x)=K_0\int_{\mathbb R^d} f(x)\, dx.
\end{equation*}
\end{lemma}

To continue with the proof, we come back to \eqref{eqSNLFS} and differentiate it to obtain 
    \begin{equation}\label{Decayderi}
        \begin{aligned}
        \partial^{\eta}Q=& \, \frac{1}{m}(1+D^{\alpha})^{-1}\Big(\sum_{\gamma_1+\gamma_2=\eta}c_{\gamma_1,\gamma_2}\partial^{\gamma_1}(Q^{m-1})\partial^{\gamma_2}Q\Big)\\
        =& \, \frac{1}{m} \, \widetilde{G}_{\alpha}\ast \Big(\sum_{\gamma_1+\gamma_2=\eta}c_{\gamma_1,\gamma_2}\partial^{\gamma_1}(Q^{m-1})\partial^{\gamma_2}Q\Big),
        \end{aligned}
    \end{equation}
where $\widetilde{G}_{\alpha}\in L^{1}(\mathbb{R}^d)$ denotes the kernel representation of the operator $(1+D^{\alpha})^{-1}$, which satisfies
\begin{equation}\label{Decayderieq1}
    \lim_{|x|\to \infty}|x|^{d+\alpha}\widetilde{G}_{\alpha}(x)=c_{\alpha}
\end{equation}
for some $c_{\alpha}>0$. For a proof of the above properties of $\widetilde{G}_{\alpha}$, we refer to \cite[Lemma C.1]{FLS2016}. Now, if $|\eta|=1$, up to some constant, the right-hand side of \eqref{Decayderi} reduces to $Q^{m-1}\partial^{\eta}Q$, and this factor satisfies
\begin{equation}
    |(Q^{m-1}\partial^{\eta}Q)(x)|\lesssim |\partial^{\eta}Q(x)|\langle x \rangle^{-(d+\alpha)(m-1)},
\end{equation}
which goes to zero as $|x|\to \infty$ provided that $Q\in H^{\infty}(\mathbb{R}^d)$. Gathering these facts, we can apply Lemma \ref{lemmadeca} together with the identity \eqref{Decayderi} to deduce
\begin{equation}
    |\partial^{\eta}Q(x)|\lesssim \langle x\rangle^{-\alpha-d},
\end{equation}
whenever $|\eta|=1$, which is exactly \eqref{decderQ} for indices of order one. On the other hand, we notice that 
\begin{equation}
\begin{aligned}
\big|\sum_{\gamma_1+\gamma_2=\eta}c_{\gamma_1,\gamma_2}(\partial^{\gamma_1}(Q^{m-1})\partial^{\gamma_2}Q)(x) \big|\lesssim \sum_{\substack{\gamma_1+\gamma_2=\eta\\ \gamma_2 \neq \eta}}  
& |\partial^{\gamma_1}(Q^{m-1})(x)||(\partial^{\gamma_2}Q)(x)|\\
& +|(Q^{m-1})(x)||\partial^{\eta}Q(x)|.
\end{aligned}
\end{equation}
Consequently, by the previous estimate, the facts that $|\partial^{\eta'}(Q^{m-1})(x)|\to 0$ and $|\partial^{\eta'}Q(x)|\to 0$ as $|x|\to \infty$ for any multi-index $\eta'$, and arguing as in the proof of the case $|\eta|=1$ above, 
an inductive argument on the order of the multi-index $\beta$ gives  
the desired result.

\underline{\bf Case $\kappa\geq 1$}. We assume by inductive hypothesis that 
\begin{equation}\label{decayinduction}
    |\partial^{\eta}Q(x)|\lesssim \langle x \rangle^{-\alpha-d-\min\{|\eta|,\kappa-1\}},
\end{equation}
for any multi-index $\eta$. We need the following series of lemmas to continue.

\begin{lemma}\label{decayQm}
Let $m>1$ be an integer, $0<\alpha<2$, and assume that \eqref{decayinduction} holds for some integer $\kappa\geq 1$. Let $\widetilde{\gamma}, \widetilde{\beta}, \widetilde{\eta}$ be multi-indices with $|\widetilde{\beta}|\leq \kappa$ and such that
\begin{equation*}
    |\widetilde{\beta}|-|\widetilde{\gamma}|-|\widetilde{\eta}|-(m-1)\alpha-(m-1) d<0.
\end{equation*}
Then 
\begin{equation*}
    |\partial^{\widetilde{\gamma}}\big(x^{\widetilde{\beta}}\partial^{\widetilde{\eta}}(Q^m)\big)(x)|\lesssim H(x) \, \langle x \rangle^{-\alpha-d} \quad \mbox{for~~ any} \quad  x \in \mathbb{R}^d,
\end{equation*}
where $H(x)$ is a continuous, bounded function such that  $H(x)\to 0$ as $|x|\to \infty$.
\end{lemma}

\begin{lemma}\label{conmident}
Let $0<\alpha<2$, $\beta, \eta$ be multi-indices with $|\beta|=|\eta|\geq 1$. Assume that $\partial^{\beta_1'}(x^{\beta_2'}f)\in L^2(\mathbb{R}^d)$ for any
multi-indices $\beta_1'$, $\beta_2'$ such that  $0\leq |\beta_1'|, |\beta_2'|\leq |\beta|-1$. Then, the following identity holds (in the $L^2$-sense)
\begin{equation}\label{identdecay}
\begin{aligned}
\, [x^{\beta},\partial^{\eta}(1+D^{\alpha})^{-1}]f =
& \sum_{\substack{|\beta_1|+|\beta_2|=|\beta|\\ |\beta_1|\geq 1}}
\, \sum_{\substack{|\gamma_1|=2|\beta_1|\\ |\gamma_2|=|\beta_2|}}
\, \sum_{1\leq l\leq|\beta_1|}c_{\beta_1,\beta_2,\gamma_1,\gamma_2,l} \, \frac{D^{\alpha l-2|\beta_1|}\partial^{\gamma_1}}{(1+D^{\alpha})^{1+|\beta_1|}}\partial^{\gamma_2}(x^{\beta_2}f)\\
&+\sum_{\substack{0\leq |\beta_2|\leq |\beta|-1\\ |\gamma|=|\beta_2|}}c_{\beta_2,\gamma} \, \frac{1}{(1+D^{\alpha})}\partial^{\gamma}(x^{\beta_2}f)
\end{aligned}    
\end{equation}
for some universal constants $c_{\beta_1,\beta_2,\gamma_1,\gamma_2,l}, c_{\beta_2,\gamma}\in \mathbb{R}$.
\end{lemma}

We consider a function $\psi\in C^{\infty}_c(\mathbb{R}^d)$ with $0\leq \psi\leq 1$,  $\psi(\xi)=1$ if $|\xi|\leq 1$, and $\psi(\xi)=0$ if $|\xi|\geq 2$. We denote by $P_\psi$ the operator defined by $\widehat{P_{\psi}f}(\xi)=\psi(\xi)\widehat{f}(\xi)$ 
and by $(1-P_{\psi})$ the operator defined by the Fourier multiplier with the symbol $1-\psi(\xi)$.

\begin{lemma}\label{kernDecay}
Let $\lambda>0$, $0<\alpha<2$, $\rho\geq 1$, and $\mu>\max\{d-\alpha,0\}$. Consider $1\leq l\leq \rho$, and let $\gamma$ be a multi-index such that $|\gamma|=2\rho$. Define
\begin{equation*}
 K_{\rho,l,\gamma}(x)=\frac{D^{\alpha l-2\rho }\partial^{\gamma}}{(\lambda+D^{\alpha})^{1+\rho}}P_{\psi}\in S'(\mathbb{R}^d)
\end{equation*}
and
\begin{equation*}
\widetilde{K}_{\mu,\rho,l,\gamma}(x)=\frac{D^{\alpha l-2\rho }\partial^{\gamma}}{(\lambda+D^{\alpha})^{1+\rho}}D^{-\mu}(1-P_{\psi})\in S'(\mathbb{R}^d).
\end{equation*}
Then 
\begin{equation}\label{Kdecay1}
    |K_{\rho,l,\gamma}(x)|\lesssim \langle x \rangle^{-\alpha -d},
\end{equation}
and for all $a>0$,
\begin{equation}\label{Kdecay2}
    |\widetilde{K}_{\mu,\rho,l,\gamma}(x)|\lesssim \langle x \rangle^{-a} \quad \mbox{for~~ any} \quad x\in \mathbb{R}^d.
\end{equation}
\end{lemma}

The proof of Lemmas \ref{decayQm}, \ref{conmident} and \ref{kernDecay} is given in Appendix. 

We are now in a position to finish the proof of the case $\kappa\geq 1$. Let $\eta, \beta$ be two multi-indices with $|\beta|=\kappa$ and $|\eta|\geq \kappa$. Since $Q\in H^{\infty}(\mathbb{R}^d)$ solves \eqref{NLEQ}, we have the following pointwise identity
\begin{equation}\label{concludecayeq1}
\begin{aligned}
x^{\beta}\partial^{\eta}Q=\frac{x^{\beta}}{m}\Big(\frac{\partial^{\eta}}{(1+D^{\alpha})}\Big)Q^m.
\end{aligned}
\end{equation}
Writing $\eta=\eta_1+\eta_2$ with $|\eta_1|=|\beta|=\kappa$, we use Lemma \ref{conmident} to deduce 
\begin{equation}\label{concludecayeq2}
\begin{aligned}
x^{\beta}\Big(\frac{\partial^{\eta}}{(1+D^{\alpha})}\Big)Q^m=&[x^{\beta},\partial^{\eta_1}(1+D^{\alpha})^{-1}]\partial^{\eta_2}(Q^m)+\frac{\partial^{\eta_1}}{(1+D^{\alpha})}(x^{\beta}\partial^{\eta_2}Q^m)\\
=& \sum_{\substack{|\beta_1|+|\beta_2|=|\beta|\\ |\beta_1|\geq 1}}\sum_{\substack{|\gamma_1|=2|\beta_1|\\ |\gamma_2|=|\beta_2|}}\sum_{1\leq l\leq|\beta_1|}c_{\beta_1,\beta_2,\gamma_1,\gamma_1,l}\frac{D^{\alpha l-2|\beta_1|}\partial^{\gamma_1}}{(1+D^{\alpha})^{1+|\beta_1|}}\partial^{\gamma_2}(x^{\beta_2}\partial^{\eta_2}(Q^m))\\
&+\sum_{\substack{0\leq |\beta_2|\leq |\beta|\\ |\gamma|=|\beta_2|}}c_{\beta_2,\gamma}\frac{1}{(1+D^{\alpha})}\partial^{\gamma}(x^{\beta_2}\partial^{\eta_2}(Q^m)).
\end{aligned}    
\end{equation}
The validity of the above expression is justified by Lemma \ref{decayQm}, which assures the conditions in Lemma \ref{conmident}. Moreover, we can increase the number of the derivatives in the above expression to use the continuous Sobolev embedding  and Lemma \ref{decayQm} to conclude that the above identity holds pointwise for all $x\in \mathbb{R}^d$. We prove that the right-hand side of \eqref{concludecayeq2} has a decay of order $-\alpha-d$ on the spatial variable.

Since $|\gamma|=|\beta_2|$, $|\beta_2|\leq \kappa$, Lemma \ref{decayQm} yields
\begin{equation*}
    |\partial^{\gamma}(x^{\beta_2}\partial^{\eta_2}(Q^m))(x)|\lesssim \langle x \rangle^{-\alpha-d} 
\end{equation*}
with $||x|^{\alpha+d}\partial^{\gamma}(x^{\beta_2}\partial^{\eta_2}(Q^m))(x)|\to 0$ as $|x|\to \infty$. 
Then, recalling that $\widetilde{G}_{\alpha}$ denotes the kernel of the operator $(1+D^{\alpha})^{-1}$ (see \eqref{Decayderi}), we use Lemma \ref{lemmadeca} to deduce
\begin{equation}\label{concludecayeq3}
\begin{aligned}
\Big|\frac{1}{(1+D^{\alpha})}\partial^{\gamma}(x^{\beta_2}\partial^{\eta_2}(Q^m))
\Big|=|\widetilde{G}_{\alpha}\ast \big(\partial^{\gamma}(x^{\beta_2}\partial^{\eta_2}(Q^m))\big)|\lesssim \langle x \rangle^{-\alpha-d},
\end{aligned}
\end{equation}
for all $0\leq |\beta_2|\leq |\beta|=\kappa$ and $|\gamma|=|\beta_2|$.

On the other hand, let $\beta_1,\beta_2,\gamma_1,\gamma_2$ be multi-indices such that $|\beta_1|+|\beta_2|=|\beta|$ with $|\beta_1|\geq 1$, $|\gamma_1|=2|\beta_1|$, $|\gamma_2|=|\beta_2|$, and $1\leq l\leq|\beta_1|$. Then  we write
\begin{equation*}
\begin{aligned}
\frac{D^{\alpha l-2|\beta_1|}\partial^{\gamma_1}}{(1+D^{\alpha})^{1+|\beta_1|}}\partial^{\gamma_2}(x^{\beta_2}\partial^{\eta_2}(Q^m)) & =\frac{D^{\alpha l-2|\beta_1|}\partial^{\gamma_1}}{(1+D^{\alpha})^{1+|\beta_1|}}P_{\psi}\partial^{\gamma_2}(x^{\beta_2}\partial^{\eta_2}(Q^m))\\
&+\frac{D^{\alpha l-2|\beta_1|}\partial^{\gamma_1}}{(1+D^{\alpha})^{1+|\beta_1|}}(1-P_{\psi})\partial^{\gamma_2}(x^{\beta_2}\partial^{\eta_2}(Q^m)).
\end{aligned}
\end{equation*}
By Lemma \ref{kernDecay}, we have
\begin{equation*}
\begin{aligned}
\frac{D^{\alpha l-2|\beta_1|}\partial^{\gamma_1}}{(1+D^{\alpha})^{1+|\beta_1|}}P_{\psi}\partial^{\gamma_2}(x^{\beta_2}\partial^{\eta_2}(Q^m))=K\ast \big(\partial^{\gamma_2}(x^{\beta_2}\partial^{\eta_2}(Q^m))\big),
\end{aligned}    
\end{equation*}
where $|K(x)|\lesssim \langle x \rangle^{-\alpha-d}$. Since Lemma \ref{decayQm} implies that $|\partial^{\gamma_2}(x^{\beta_2}\partial^{\eta_2}(Q^m))|\lesssim \langle x \rangle^{-\alpha-d}$, we conclude
\begin{equation}\label{concludecayeq4}
\begin{aligned}
\Big|\frac{D^{\alpha l-2|\beta_1|}\partial^{\gamma_1}}{(1+D^{\alpha})^{1+|\beta_1|}}P_{\psi}\partial^{\gamma_2}(x^{\beta_2}\partial^{\eta_2}(Q^m))(x)\Big|=& |K\ast \big(\partial^{\gamma_2}(x^{\beta_2}\partial^{\eta_2}(Q^m))\big)(x)|
\lesssim & \langle x\rangle^{-\alpha-d}.
\end{aligned}
\end{equation}
Next, we let $p\geq 1$ integer be such that $2p>\max\{d-\alpha,0\}$ and we write
\begin{equation*}
\begin{aligned}
\frac{D^{\alpha l-2|\beta_1|}\partial^{\gamma_1}}{(1+D^{\alpha})^{1+|\beta_1|}}&(1-P_{\psi})\partial^{\gamma_2}(x^{\beta_2}\partial^{\eta_2}(Q^m))\\
&=\frac{D^{\alpha l-2|\beta_1|}\partial^{\gamma_1}(-\Delta)^{-p}}{(1+D^{\alpha})^{1+|\beta_1|}}(1-P_{\psi})(-\Delta)^{p}\partial^{\gamma_2}(x^{\beta_2}\partial^{\eta_2}(Q^m))\\
&=\widetilde{K}\ast\big((-\Delta)^{p}\partial^{\gamma_2}(x^{\beta_2}\partial^{\eta_2}(Q^m))\big),
\end{aligned}    
\end{equation*}
where Lemma \ref{kernDecay} shows that $|\widetilde{K}(x)|\lesssim \langle x\rangle^{-\alpha-d}$. Given that $(-\Delta)^{p}$ is a local operator ($p$ is an integer), Lemma \ref{decayQm} implies that $|(-\Delta)^{p}\partial^{\gamma_2}(x^{\beta_2}\partial^{\eta_2}(Q^m))(x)|\lesssim \langle x\rangle^{-\alpha-d}$. Summarizing, we obtain
\begin{equation}\label{concludecayeq5}
\begin{aligned}
\Big|\frac{D^{\alpha l-2|\beta_1|}\partial^{\gamma_1}}{(1+D^{\alpha})^{1+|\beta_1|}}&(1-P_{\psi})\partial^{\gamma_2}(x^{\beta_2}\partial^{\eta_2}(Q^m))(x)\Big|\lesssim \langle x\rangle^{-\alpha-d}.
\end{aligned}
\end{equation}
Plugging \eqref{concludecayeq3}, \eqref{concludecayeq4} and \eqref{concludecayeq5} into \eqref{concludecayeq2} and \eqref{concludecayeq1}, we get
\begin{equation*}
\begin{aligned}
|x^{\beta}\partial^{\eta}Q(x)|\lesssim \langle x \rangle^{-\alpha-d}.
\end{aligned}
\end{equation*}
Since $\beta$ is an arbitrary multi-index with $|\beta|=\kappa$, we conclude
\begin{equation*}
\begin{aligned}
|\partial^{\eta}Q(x)|\lesssim \langle x \rangle^{-\alpha-d-\kappa}
\end{aligned}
\end{equation*}
for all $|\eta|\geq \kappa$. This completes the inductive step, consequently, finishing the proof of the lemma. 
\end{proof}

\begin{remark}
The proof of Lemma \ref{LemmaQprop} shows that the spatial decay of $Q$ is limited by that of the kernel $\widetilde{G}_{\alpha}$ associated with the operator $(1+D^{\alpha})^{-1}$. For example, this can be seen in the decomposition \eqref{concludecayeq2}, where the operator $(1+D^{\alpha})^{-1}$ ultimately determines the maximum spatial decay of $Q$ and its derivatives.
\end{remark}

We next state the Pokhozhaev identities of the equation \eqref{EQ:groundState}, which will be useful later. 
\begin{lemma}\label{PHident}
Let $d\geq 1$. Assume $0<\alpha<2$, $c>0 $, $1<m< m_{\ast}$. Let $Q_c$ be a smooth vanishing at infinity solution of \eqref{EQ:groundState}. Then the following identities hold
\begin{align}
&\|D^{\frac{\alpha}{2}}Q_c \|_{L^2}^2 =\frac{d\,c\,(m-1)}{(2d-(d-\alpha)(m+1))}\| Q_c \|_{L^2}^2, \label{Phideneq1} \\
& \int_{\mathbb{R}^d} Q^{m+1}_c \, dx=\frac{\alpha\, c\,m\,(m+1)}{(2d-(d-\alpha)(m+1))}\| Q_c \|_{L^2}^2. \label{Phideneq2}
\end{align}
As a consequence, 
\begin{equation}\label{E:Energy-phi}
E[Q_c] = s_c \, \frac{c(m-1)}{(2d-(d-\alpha)(m+1))} \| Q_c \|_{L^2}^2,
\end{equation}
where $s_c$ is defined in \eqref{CritiInd}.
\end{lemma}

\begin{proof}
Lemma \ref{PHident} follows by deducing Pokhozhaev identities via the equation \eqref{EQ:groundState}. These identities can be obtained by similar arguments as in the proof of  \cite[Lemma 1]{OscSvetKai2021}. We remark that to adapt the proof from \cite{OscSvetKai2021}, we require the following identities
\begin{equation}
[D^{\frac{\alpha}{2}},x_j]\partial_{x_j}f=-\frac{\alpha }{2}D^{\frac{\alpha}{2}-2}\partial_{x_j}^2f,
\end{equation}
for all $j=1,\dots,d$. 
\end{proof}

Before closing this section, we mention a commutator estimate for the homogeneous fractional derivatives that we need later. 
\begin{proposition}\label{fractionalDeriv}
Let $d \geq 1$, $s>0$ and $1<p<\infty$. Then  
\begin{equation}
\|D^s(fg)-\sum_{|\beta|\leq s} \frac{1}{\beta!} \partial^{\beta}f D^{s,\beta} g\|_{L^p(\mathbb R^d)}\lesssim \|D^{s}f\|_{L^{\infty}}\|g\|_{L^{p}(\mathbb R^d)}.
\end{equation}
The operator $D^{s,\beta}$ is defined via the Fourier transform as
\begin{equation*}
\widehat{D^{s,\beta}g}(\xi)=\widehat{D^{s,\beta}}(\xi)\widehat{g}(\xi) \, \, \text{ and } \, \,  \widehat{D^{s,\beta}}(\xi)=i^{-|\beta|}\partial_{\xi}^{\beta}(|\xi|^s).
\end{equation*}
\end{proposition}
Proposition \ref{fractionalDeriv} is a particular case of \cite[Theorem 1.2]{DonLi2019}. For other fractional commutator estimates in the one-dimensional setting, refer to \cite{KPV1993}. 

\section{Orbital Stability}\label{S:OrbitalS}

In this section we prove Theorem \ref{OrbstabilThm}. We use the concentration-compactness method to develop the ideas from \cite{Albert1999,LinPilodSaut2015,LionsI1984,LionsII1984} in our setting, covering the entire range of $\alpha$ and also extending to higher dimensions in our model \eqref{EQ:fKdV}. Recalling \eqref{E:energy} and \eqref{E:mass}, we define for any $q>0$
\begin{equation}\label{min:prob}
\mathcal{I}_q=\inf_{u\in H^{\frac{\alpha}{2}}(\mathbb{R}^d)}\left\{ E[v]\, :\, M[v]=q \right\}.
\end{equation}  
Denote by $\mathcal{G}_q=\{v\in H^{\frac{\alpha}{2}}(\mathbb{R}^d)\,:\, M[v]=q \, \text{ and } \, E[v]=\mathcal{I}_q \}$ the set (possibly empty) of minimizers of \eqref{min:prob}. 

\begin{lemma}\label{Lemmanegat}
Let $0<\alpha<2$ and $1<m<\frac{2\alpha}{d}+1$. For any $q>0$ it follows $-\infty<\mathcal{I}_q<0$.
\end{lemma}

\begin{proof}
We use the Gagliardo-Nirenberg \eqref{GNineq} and Young's inequalities to deduce 
\begin{equation}\label{eq:estab1}
\begin{aligned}
\Big|\int v^{m+1} \big| \leq \delta \|v\|_{H^{\frac{\alpha}{2}}}^2+C_{\delta}\|v\|_{L^2}^{\frac{2(\alpha(m+1)-d(m-1))}{2\alpha-d(m-1)}}
\end{aligned}
\end{equation}
for any $\delta>0$. The previous inequality and taking $\delta>0$ small enough yield
\begin{equation}
\begin{aligned}
E[v]&=E[v]+M[v]-M[v]\\
&=\frac{1}{2}\|v\|_{H^{\frac{\alpha}{2}}}^2-\frac{1}{m(m+1)}\int v^{m+1} -M[v] \\
&\geq \Big(\frac{1}{2}-\frac{\delta}{m(m+1)} \Big)\|v\|_{H^{\frac{\alpha}{2}}}-q-\frac{C_{\delta}}{m(m+1)}q^{\frac{2(\alpha(m+1)-d(m-1))}{2\alpha-d(m-1)}}\\
&\geq  -q-\frac{C_{\delta}}{{m(m+1)}}q^{\frac{2(\alpha(m+1)-d(m-1))}{2\alpha-d(m-1)}}.
\end{aligned}
\end{equation}
Thus, $\mathcal{I}_q>-\infty$. On the other hand, let $\varphi \in H^{\frac{\alpha}{2}}(\mathbb{R}^d)$ with $\int \varphi^{m+1} > 0$ and $M[\varphi]=q$. For each $\theta>0$, we define $\varphi_{\theta}(x)=\theta\varphi(\theta^{\frac{2}{d}} x)$, then $M[\varphi_{\theta}]=M[\varphi]=q$, and
\begin{equation*}
E[\varphi_{\theta}]=\frac{\theta^{\frac{2\alpha}{d}}}{2}\int |D^{\frac{\alpha}{2}}\varphi|^2 -\frac{\theta^{m-1}}{m(m+1)}\int \varphi^{m+1}.
\end{equation*} 
Since $m-1<\frac{2\alpha}{d}$, we can take some $0<\theta_0<1$ sufficiently small such that $E[\varphi_{\theta}]<0$, which implies $\mathcal{I}_q<0$.
\end{proof}

We now establish a subadditivity property of $\mathcal{I}_q$. 

\begin{lemma}\label{subaddprop} 
Let $1<m<\frac{2\alpha}{d}+1$. For any $q_1,q_2>0$, it follows $\mathcal{I}_{q_1+q_2}< \mathcal{I}_{q_1}+\mathcal{I}_{q_2}$.
\end{lemma}

\begin{proof}
Given $v\in H^{\frac{\alpha}{2}}(\mathbb{R}^d)$ with $M[v]=q$, we consider the scaling (for $\theta>0$)
\begin{equation}
v_{\theta}(x)=\theta^{\frac{\alpha}{2\alpha-d(m-1)}}v\big(\theta^{\frac{(m-1)}{2\alpha-d(m-1)}}x\big).
\end{equation}
Then $\|v_{\theta}\|_{L^2}^{2}=\theta\|v\|_{L^2}^{2}$ and $E[v_{\theta}]=\theta^{\frac{\alpha(m+1)-d(m-1)}{2\alpha-d(m-1)}}E[v]$ yield
\begin{equation*}
\mathcal{I}_{\theta q}=\theta^{\frac{\alpha(m+1)-d(m-1)}{2\alpha-d(m-1)}} \mathcal{I}_q.
\end{equation*}
Choosing $q=1$, $\theta=q_1+q_2$ in the above identity, and using that $\mathcal{I}_1<0$, we get
\begin{equation*}
\mathcal{I}_{q_1+q_2}
=(q_1+q_2)^{\frac{\alpha(m+1)-d(m-1)}{2\alpha-d(m-1)}} \,\mathcal{I}_1 
< \Big( q_1^{\frac{\alpha(m+1)-d(m-1)}{2\alpha-d(m-1)}}+q_2^{\frac{\alpha(m+1)-d(m-1)}{2\alpha-d(m-1)}}\Big) \mathcal{I}_1 
= \mathcal{I}_{q_1}+\mathcal{I}_{q_2}.
\end{equation*}
\end{proof}
Next, we show that the minimizing sequences for $\mathcal{I}_q$ are bounded in $H^{\frac{\alpha}{2}}(\mathbb{R}^d)$.
\begin{lemma}\label{lemmaUnibound}
Let $0<\alpha<2$, $1<m<\frac{2\alpha}{d}+1$, $q>0$, and $\{v_n\}$ be a minimizing sequence for $\mathcal{I}_q$ in $H^{\frac{\alpha}{2}}(\mathbb{R}^d)$. Then there exist a constant $C>0$ and $\delta>0$ such that
\begin{itemize}
\item[(a)] 
$\|v_n\|_{H^{\frac{\alpha}{2}}}\leq C$ for all $n$, and

\item[(b)] 
$\|v_n\|_{L^{m+1}} \geq \delta$ for all $n$ sufficiently large. 
\end{itemize} 
\end{lemma}

\begin{proof}
We apply \eqref{eq:estab1} to get the following upper bound
\begin{equation}
\begin{aligned}
\frac{1}{2}\|v_n\|_{H^{\frac{\alpha}{2}}}^2&=E[v_n]+\frac{1}{2}\int |v_n|^2+\frac{1}{m(m+1)}\int v_n^{m+1}  \\
& \leq E[v_n]+\delta_0 \|v_n\|_{H^{\frac{\alpha}{2}}}^2+C(q),
\end{aligned}
\end{equation}
for any $\delta_0 > 0$, which establishes (a).  
Next we deduce (b). Arguing by contradiction, we assume that there exists a subsequence $\{v_{n_k}\}$ of $\{v_{n}\}$ such that $\|v_{n_k}\|_{L^{m+1}} \to 0$ as $n_k \to \infty$. Then
\begin{equation}
\begin{aligned}
\mathcal{I}_q=\lim_{n_k \to \infty} E[v_{n_k}] \geq \lim_{n_k \to \infty} -\frac{1}{m(m+1)}\|v_{n_k}\|_{L^{m+1}}^{m+1}=0,
\end{aligned}
\end{equation}
which, according to Lemma \ref{Lemmanegat} contradicts the fact that $\mathcal{I}_q<0$. Consequently, (b) holds. 
\end{proof}

We apply the concentration-compactness argument to a minimizing sequence $\{v_n\}$.
Let a sequence of non-decreasing functions $\rho_n: [0,\infty) \rightarrow [0,q]$ be defined by
\begin{equation*}
\rho_n(r)=\sup_{z \in \mathbb{R}^d} \, \frac{1}{2}\int_{B(z,r)}|v_n|^2 \, dx,
\end{equation*}
where $B(z,r)=\{x \in \mathbb{R}^d\, :\, |x-z|<r\}$. Then there exists a subsequence of $\{\rho_n\}$, still denoted by $\{\rho_n\}$, which converges uniformly on compact sets to a non-decreasing function $\rho:[0,\infty) \rightarrow [0,q]$. Define $\lambda=\lim\limits_{r\to \infty}\rho(r)$. Then $0\leq \lambda \leq q$. The three mutually exclusive cases of  $\lambda$ values are
\begin{itemize}
\item vanishing: $\lambda=0$,
\item dichotomy: $0<\lambda<q$,
\item compactness: $\lambda=q$.
\end{itemize}
We show that vanishing and dichotomy do not hold (Lemmas \ref{lemmavanish}-\ref{rulevanish}). In the case of compactness, the following result shows that the set of minimizers $\mathcal{G}_q$ is not empty.

\begin{lemma}\label{lemmacompactres}
Let $0<\alpha<2$, $1<m<\frac{2\alpha}{d}+1$. Assume $\lambda=q$. Then there exists a sequence of real numbers $\{z_n\}$ such that
\begin{itemize}
\item[(a)] for every $\eta<q$, there exists $r=r(\eta)$ such that
\begin{equation}
\frac{1}{2} \int_{B(z_n,r)}|v_n|^2 \, dx > \eta,
\end{equation}
for all $n$ large enough.
\item[(b)] The sequence $\{\widetilde{v}_n\}$ defined by $\widetilde{v}_n(x)=v_n(x+z_{n})$ has a subsequence, which converges in $H^{\frac{\alpha}{2}}(\mathbb{R}^d)$ to a function $g \in \mathcal{G}_q$. In other words, $\mathcal{G}_q$ is not empty. 
\end{itemize}
\end{lemma}

\begin{proof}
We incorporate the idea from \cite[Lemma 2.5]{Albert1999}. 
Since  $\lambda=q$, there exists $r_0>0$ such that for 
all $n$ sufficiently large
\begin{equation}
\rho_n(r_0)=\sup_{z\in \mathbb{R}^d}\frac{1}{2}\int_{B(z,r_0)}|v_n|^2   >\frac{q}{2}.
\end{equation}
Then for each sufficiently large $n$, there exists $z_n\in \mathbb{R}^d$, such that
\begin{equation}
\frac{1}{2}\int_{B(z_n,r_0)}|v_n|^2 >\frac{q}{2}.
\end{equation}
Let $\eta<q$. By the previous argument, we may assume that $\eta>\frac{q}{2}$. Since  $\lambda=q$, we can find a subsequence $\{ z_n(\eta) \} \in \mathbb{R}^d$, $r_0(\eta)>0$, and $N(\eta)>0$ such that if $n\geq N(\eta)$, we get
\begin{equation}
\frac{1}{2}\int_{B(z_n(\eta),r_0(\eta))}|v_n|^2>\eta.
\end{equation}
Given that $\|v_n\|_{L^2}^2=2q$, it must follow that for large $n$, the balls $B(z_n,r_0)$ and $B(z_n(\eta),r_0(\eta))$ must intersect. Thus, setting $r=2r_0(\eta)+r_0$, we get $B(z_n(\eta),r_0(\eta))\subset B(z_n,r)$. This observation completes the proof of (a).

Now, by part (a), we have that for every $k>0$ integer, there exists $r_k \in \mathbb{R}^{+}$ such that for all $n$ large
\begin{equation}
\frac{1}{2} \int_{B(0,r_k)}|\widetilde{v}_n|^2   > q-\frac{1}{k}.
\end{equation}
By Lemma \ref{lemmaUnibound}, the sequence $\{\widetilde{v}_n\}$ is uniformly bounded in $H^{\frac{\alpha}{2}}(\mathbb{R}^d)$, thus, by compactness embedding of fractional Sobolev spaces on open Lipschitz domains with bounded boundary in $\mathbb{R}^d$ (see 
\cite[Section 7]{NezzaPalatucciValdinoci2012} and references therein), there exists a subsequence $\{\widetilde{v}_n\}$ (still denoted by $\{\widetilde{v}_n\}$), converging in $L^2(B(0,r_k))$ to a function $v \in L^2(B(0,r_k))$ such that
\begin{equation}
\frac{1}{2}\int_{B(0,r_k)}|v|^2  > q-\frac{1}{k}.
\end{equation}
Therefore, we can apply a Cantor diagonalization argument and the fact that $\|\widetilde{v}_n\|_{L^2}^2=2q$ for all $n$, to find a subsequence $\{\widetilde{v}_n\}$ converging in $L^2(\mathbb{R}^d)$ to some $v\in L^2(\mathbb{R}^d)$ with $\|v\|_{L^2}^2=2q$.  Moreover, by uniqueness of the weak limit, up to some  subsequence, $\{v_n\}$ converges weakly to $v$ in $H^{\frac{\alpha}{2}}(\mathbb{R}^d)$, and
\begin{equation*}
\|v\|_{H^{\frac{\alpha}{2}}} \leq \liminf_{n\to \infty} \|\widetilde{v}_n\|_{H^{\frac{\alpha}{2}}}\leq C,
\end{equation*}
where the constant is determined by the part (a) of Lemma \ref{lemmaUnibound}. Hence, an application of the Gagliardo-Nirenberg inequality \eqref{GNineq} yields
\begin{equation*}
   \|v-\widetilde{v}_n\|_{L^{m+1}}^{m+1} \lesssim  \|v-\widetilde{v}_n\|_{H^{\frac{\alpha}{2}}}^{\frac{d(m-1)}{\alpha}}\|v-\widetilde{v}_n\|_{L^2}^{(m+1)-\frac{d(m-1)}{\alpha}}\lesssim \|v-\widetilde{v}_n\|_{L^2}^{(m+1)-\frac{d(m-1)}{\alpha}} \to 0,
\end{equation*}
as $n \to \infty$. Collecting the previous estimates for the obtained above sequence $\{\tilde{v}_n\}$ and recalling $M[\tilde{v}_n]=M[v]$, we get
\begin{equation}
\begin{aligned}
E[v]&=\frac{1}{2}\|v\|_{H^{\frac{\alpha}{2}}}^2- 
M[v]-\frac{1}{m(m+1)}\int v^{m+1} \\
&\leq \lim_{n \to \infty} \Big( E[\widetilde{v}_n]+\frac{1}{m(m+1)}\big|\|v\|_{L^{m+1}}^{m+1}-\|\widetilde{v}_n\|_{L^{m+1}}^{m+1}\big|\Big) 
&=\lim_{n\to \infty}  E[\widetilde{v}_n]=\mathcal{I}_q.
\end{aligned}
\end{equation}
Hence, $v\in \mathcal{G}_q$, and thus, $\mathcal{G}_q$ is non-empty. Furthermore, since by definition $\mathcal{I}_q \leq E[{v}_n]$, we have  $\lim\limits_{n\to \infty }E[\widetilde{v}_n]=E[v]$, and  using the fact that $\widetilde{v}_n\to v$ in the $L^{m+1}$-norm, we find that $\|D^{\frac{\alpha}{2}}\widetilde{v}_n\|_{L^2}\to \|D^{\frac{\alpha}{2}}v\|_{L^2}$ as $n \to \infty$. 
This convergence, also the weak convergence of $\tilde{v}_n$ to $v$ in $H^{\frac{\alpha}{2}}(\mathbb{R}^d)$ and the properties of a Hilbert space yield strong convergence of the sequence $\{\widetilde{v}_n\}$ to $v$ in the $H^{\frac{\alpha}{2}}$-norm. This completes the proof of (b).
\end{proof}

The next results will be useful to rule out the vanishing case $\lambda=0$.

\begin{lemma}\label{lemmavanish}
Let $0<\alpha<2$, $1<m<\frac{2\alpha}{d}+1$, $B>0$ and $\delta>0$. Then there exists $\eta=\eta(B,\delta)$ such that if $v\in H^{\frac{\alpha}{2}}(\mathbb{R}^d)$ with $\|v\|_{H^{\frac{\alpha}{2}}}\leq B$ and $\|v\|_{L^{m+1}}\geq \delta$,
we have 
\begin{equation}
\sup_{z\in \mathbb{R}^d} \int_{B(z,2 \sqrt{d}\, )}|v|^{m+1}  \geq \eta.
\end{equation}
\end{lemma}
\begin{proof}
We take the idea from \cite{ABS1997,Albert1999} and develop it for the entire range of $\alpha$ and the higher dimensional setting. 
Let $\varphi \in C^{\infty}_c(\mathbb{R})$ such that $0\leq \varphi \leq 1$, it is supported on $[-2,2]$, and $\sum_{j \in \mathbb{Z}} \varphi_j(x)=1$, where $\varphi_j(x) \stackrel{def}{=} \varphi(x-j)$, $j\in \mathbb{Z}$. We define the map $T: H^{s}(\mathbb{R}^d) \to l^2 (H^s(\mathbb{R}^d))$ by
\begin{equation*}
Tf=\{\varphi_{1,j_1}\dots \varphi_{d,j_d} f\}_{j_k \in \mathbb Z, \, k=1,\dots,d \,},
\end{equation*}
where $\varphi_{l,j_l}(x_1,\dots,x_d)=\varphi(x_l-j_{l})$, for all $l=1,\dots,d$. It is not difficult to see that $T$ is bounded when $s=0$ and 
$s=1$. 
Then by interpolation, e.g., see \cite{BL1976}, 
it is bounded when $s=\frac{\alpha}{2}$, in other words, there exits $C_0>0$ such that for all $f \in H^{\frac{\alpha}{2}}(\mathbb{R}^d)$,
\begin{equation}\label{E:T1}
\begin{aligned}
\qquad \sum_{ j_k \in \mathbb Z, \, k=1,\dots,d}
\|\varphi_{1,j_1}\dots \varphi_{d,j_d} f\|_{H^{\frac{\alpha}{2}}}^2 \leq C_0 \|f\|_{H^{\frac{\alpha}{2}}}^2.
\end{aligned}
\end{equation}
Now, let $0<C_1 \leq \sum_{j_k \in \mathbb Z, \, k=1,\dots,d}$  $|\varphi(x_1-j_1)|^{m+1}\dots|\varphi(x_d-j_d)|^{m+1}$ for all $(x_1,\dots,x_d) \in \mathbb{R}^d$. Such a constant exists as the support of $\varphi$ reduces the above sum to a finite number of elements independent of $x=(x_1, \dots, x_d)\in \mathbb{R}^d$. Setting $C_2=\frac{C_0B^2}{C_1}$, we claim that for every function $v \in H^{\frac{\alpha}{2}}(\mathbb{R}^d)$ with $v \neq 0$, and satisfying the hypothesis of Lemma \ref{lemmavanish}, there exist $j_{0,l} \in \mathbb{Z}$, $1\leq l \leq d$, such that 
\begin{equation}\label{eq:estab2}
\|\varphi_{1,j_{0,1}}\dots \varphi_{d,j_{0,d}} v\|_{H^{\frac{\alpha}{2}}}^2\leq (1+C_2\|v\|_{L^{m+1}}^{-(m+1)})\|\varphi_{1,j_{0,1}}\dots \varphi_{d,j_{0,d}} v\|_{L^{m+1}}^{m+1}.
\end{equation}
Otherwise, it must be the case that
\begin{equation}\label{E:T2}
\|\varphi_{1,j_{1}}\dots \varphi_{d,j_{d}} v\|_{H^{\frac{\alpha}{2}}}^2>(1+C_2\|v\|_{L^{m+1}}^{-(m+1)})\|\varphi_{1,j_{1}}\dots \varphi_{d,j_{d}} v\|_{L^{m+1}}^{m+1},
\end{equation}
for every $j_k \in \mathbb Z, \, k=1,\dots,d$. Summing over $j_{1},\dots,j_d$, the inequality \eqref{E:T2} together with \eqref{E:T1} yield
\begin{equation*}
\begin{aligned}
C_0\|v\|_{H^{\frac{\alpha}{2}}}^2&>(1+C_2\|v\|_{L^{m+1}}^{-(m+1)})\sum_{j_k \in \mathbb Z, \, k=1,\dots,d}
\|\varphi_{1,j_{1}}\dots \varphi_{d,j_{d}} v\|_{L^{m+1}}^{m+1}\\
&\geq (1+C_2\|v\|_{L^{m+1}}^{-(m+1)})C_1\|v\|_{L^{m+1}}^{m+1}.
\end{aligned}
\end{equation*}
Since $C_1C_2=C_0B^2$, we get
\begin{equation*}
C_0B^2>C_1\|v\|_{L^{m+1}}^{m+1}+C_0B^2,
\end{equation*}
which is a contradiction. Thus, \eqref{eq:estab2} holds true for some integers $j_{0,1},\dots j_{0,d}$.  By Sobolev embedding, there exists a uniform constant $C_3>0$ such that
\begin{equation*}
\begin{aligned}
\|\varphi_{1,j_{0,1}}\dots \varphi_{d,j_{0,d}} v\|_{L^{m+1}}\leq C_3 \|\varphi_{1,j_{0,1}}\dots \varphi_{d,j_{0,d}} v\|_{H^{\frac{\alpha}{2}}}.
\end{aligned}
\end{equation*}
Then the above inequality and \eqref{eq:estab2} yield
\begin{equation*}
\frac{1}{C_3^2}\|\varphi_{1,j_{0,1}}\dots \varphi_{d,j_{0,d}} v\|_{L^{m+1}}^2\leq \Big( 1 +\frac{C_2}{\delta^{m+1}} \Big)\|\varphi_{1,j_{0,1}}\dots \varphi_{d,j_{0,d}} v\|_{L^{m+1}}^{m+1},
\end{equation*}
which can be rewritten as  
\begin{equation*}
\eta \leq \|\varphi_{1,j_{0,1}}\dots \varphi_{d,j_{0,d}} v\|_{L^{m+1}},
\end{equation*}
with $\eta>0$ defined by $\eta^{m-1}=\frac{\delta^{m+1}}{C_3^2(\delta^{m+1}+C_2)}$. 
Checking the support of $\varphi_{1,j_{0,1}}\dots \varphi_{d,j_{0,d}}$
in $d$ dimensions, the desired conclusion is now a consequence of the above inequality. 
\end{proof}

The next lemma will be used to exclude the case of dichotomy ($0<\lambda<q$).

\begin{lemma}\label{dichotomyLemma} 
Given $0<\alpha < 2$, assume $1<m<\frac{2\alpha}{d}+1$. Let $\{v_n\}$ be the minimizing sequence as in Lemma \ref{lemmaUnibound}. Then for every $\delta>0$ there exist an integer $N\geq 1$ integer and sequences $\{g_N,g_{N+1},\dots\}$ and $\{h_N,h_{N+1},\dots\}$ of functions in $H^{\frac{\alpha}{2}}(\mathbb{R}^d)$ such that for every $n\geq N$,
\begin{itemize}
\item[(a)] $|M[g_n]-\lambda|<\delta$,
\item[(b)] $|M[h_n]-(q-\lambda)|<\delta$,
\item[(c)] $E[v_n]\geq E[g_n]+E[h_n]-\delta$.
\end{itemize}
\end{lemma}

\begin{proof}
Let $\vartheta \in C^{\infty}_c(\mathbb{R}^d)$ be such that $0\leq \vartheta \leq 1$ supported on the set $|x|\leq 2$, with $\vartheta(x)\equiv 1$ when $|x|\leq 1$. Take $\varphi \in C^{\infty}(\mathbb{R}^d)$ be such that $\vartheta^2+\varphi^2=1$ on $\mathbb{R}^d$. For each $r > 0$, we set $\vartheta_r(x)=\vartheta(\frac{x}{r})$, and similarly, we define $\varphi_r(x)$. Recall the definition of the function $\rho(r)$. Then for any fixed $\delta>0$ and for all sufficiently large values of $r$, one has 
\begin{equation*}
\lambda- \delta< \rho(\tfrac{r}{2})\leq \rho(3r)\leq \lambda.
\end{equation*}
For such value of $r$ fixed, we can choose $N$ large such that for all $n \geq N$ we have
\begin{equation*}
\lambda- \delta< \rho_n(\tfrac{r}{2})\leq \rho_n(3r)\leq \lambda+\delta.
\end{equation*}
Thus, for each $n \geq N$, there exists $z_n$ such that 
\begin{equation}\label{eq:estab3}
\frac{1}{2}\int_{B(z_n,r)}|v_n|^2 \geq \frac{1}{2}\int_{B(z_n,\frac{r}{2})}|v_n|^2 >\lambda-\delta,
\end{equation}
and
\begin{equation}\label{eq:estab4}
\frac{1}{2}\int_{B(z_n,2r)}|v_n|^2\leq \frac{1}{2}\int_{B(z_n,3r)}|v_n|^2 <\lambda+\delta.
\end{equation}
In particular, one has
\begin{equation}\label{eq:estab4.0}
    \Big|\frac{1}{2}\int_{\{\frac{r}{2}< |x-z_n|< 3r\}} |v_n|^2\Big|\leq 2\delta.
\end{equation}

Define $g_n(x)=\vartheta_r(x-z_{n})v_n(x)$ and $h_n(x)=\varphi_r(x-z_{n})v_n(x)$. Consequently, the previous estimates show that $g_n$ and $h_n$ satisfy (a) and (b), respectively. 

On the other hand, we have
\begin{equation}\label{eq:estab4.1}
\begin{aligned}
E[g_n]&+E[h_n]=E[v_n]+\frac{1}{2} \int [D^{\alpha},\vartheta_r]v_n\, \vartheta_r v_n +\frac{1}{2} \int [D^{\alpha},\varphi_r]v_n\, \varphi_r v_n  \\
&-\frac{1}{m(m+1)}\int (\vartheta_r^{m+1}-\vartheta_r^2)v_n^{m+1} 
-\frac{1}{m(m+1)}\int (\varphi_r^{m+1}-\varphi_r^2)v_n^{m+1}.
\end{aligned}
\end{equation}
We proceed to estimate each term on the right hand-side of the above identity. Recalling the operators $D^{\alpha, \beta}$ defined in Proposition \ref{fractionalDeriv}, we write
\begin{equation}\label{eq:estab4.2}
\begin{aligned}
\left[D^{\alpha},\vartheta_r \right]v_n=&  
\underbrace{D^{\alpha}(\vartheta_r v_n)-\sum_{|\beta|\leq \alpha}\frac{1}{\beta !}\partial^{\beta}(\vartheta_r)D^{\alpha, \beta}v_n}_{R(\vartheta_r,v_n)}
+\sum_{1 \leq |\beta|\leq \alpha}\frac{1}{\beta !}\partial^{\beta}(\vartheta_r)D^{\alpha, \beta}v_n\\
\end{aligned}
\end{equation}
where we denoted the first two terms by $R(\vartheta_r,v_n)$ and for the second, when needed, we follow the zero convention for the empty summation, that is, if $\alpha<1$, then $\sum_{1\leq |\beta|\leq \alpha}(\cdots)=0$. By Proposition \ref{fractionalDeriv}, we have
\begin{equation}
\begin{aligned}
\|R(\vartheta_r,v_n)\|_{L^{2}}&\lesssim \|D^{\alpha}(\vartheta_r)\|_{L^{\infty}}\|v_n\|_{L^{2}}\lesssim r^{-\alpha}\|v_n\|_{H^{\frac{\alpha}{2}}}\sim r^{-\alpha},
\end{aligned}
\end{equation}
where we used Lemma \ref{lemmaUnibound}. Now, since $|\beta|\leq \alpha$, there exist a multi-index $\beta'=(\beta_1',\dots,\beta_d')$ and $c\neq 0$ such that
\begin{equation}
D^{\alpha, \beta}v_n=c \, \mathcal{R}_1^{\beta_1'}\dots\mathcal{R}_d^{\beta_d'}D^{\alpha-|\beta|}v_n,
\end{equation} 
where $\mathcal{R}_j=\partial_{x_j}D^{-1}$ denotes the Riesz transform operator in the $j$th-direction. Thus, we get
\begin{equation}\label{eq:estab4.3}
\begin{aligned}
\|\sum_{1 \leq |\beta|\leq \alpha}\frac{1}{\beta !}\partial^{\beta}(\vartheta_r)D^{\alpha, \beta}v_n\|_{L^2} \lesssim & \sum_{1 \leq |\beta|\leq \alpha} \frac{1}{r^{|\beta|}}\|\partial^{\beta}\vartheta\|_{L^{\infty}}\|\mathcal{R}_1^{\beta_1'}\dots \mathcal{R}_d^{\beta_d'}D^{\alpha-|\beta|}v_n\|_{L^2}\\
\lesssim & \, \frac{1}{r} \,\|v_n\|_{H^{\frac{\alpha}{2}}}\sim \frac{1}{r},
\end{aligned}
\end{equation} 
provided that $r\geq 1$. Similarly, we estimate $[D^{\alpha},\varphi_r]v_n$: a key observation 
here is to recall the definition of $\varphi$ and $\varphi^2+\vartheta^2=1$, which yield
\begin{equation}
\varphi=1-\widetilde{\varphi}, \, \, \text{ where }\, \, \widetilde{\varphi}=1-\sqrt{1-\vartheta^2}\in C^{\infty}_c(\mathbb{R}^d),
\end{equation}
which implies that in the distributional sense $D^{s,\beta}\varphi=-D^{s,\beta}\widetilde{\varphi}$ for all $s>0$, and $0\leq |\beta| \leq s$. Summarizing, there exists some $m_1>0$ such that
\begin{equation}
\Big|\int [D^{\alpha},\vartheta_r]v_n\, \vartheta_r v_n \Big|+\Big| \int [D^{\alpha},\varphi_r]v_n\, \varphi_r v_n\Big|=O(r^{-m_1}).
\end{equation}
Next, let $\widetilde{\vartheta}\in C^{\infty}_c(\mathbb{R}^d)$ be such that $\supp(\widetilde{\vartheta})\subset \{x:\, \frac{1}{2}\leq |x|\leq 3\}$, $0\leq \widetilde{\vartheta}\leq 1$, and $\widetilde{\vartheta}=1$ on $\{x: \, 1\leq |x|\leq 2\}$. Then, 
recalling the support of ${\vartheta}$, we have 
$\widetilde{\vartheta}^{m+1}(\vartheta^2-\vartheta^{m+1})=(\vartheta^2-\vartheta^{m+1})$. With this, setting $\widetilde{\vartheta}_r(x)=\widetilde{\vartheta}(\frac{x-z_n}{r})$ for all $x\in \mathbb{R}^d$ and applying the Gagliardo-Nirenberg inequality \eqref{GNineq}, we obtain
\begin{equation}\label{eq:estab5}
\begin{aligned}
\Big|\int (\vartheta_r^{2}-\vartheta_r^{m+1})v_n^{m+1}  \Big|\lesssim & \int |\widetilde{\vartheta}_r v_n|^{m+1}  \\ \lesssim  &\|\widetilde{\vartheta}_r v_n\|_{L^2}^{(m+1)-\frac{d(m-1)}{\alpha}}\|D^{\frac{\alpha}{2}}(\widetilde{\vartheta}_r v_n)\|_{L^{2}}^{\frac{d(m-1)}{\alpha}} \\
\lesssim &\,  \delta^{\frac{(m+1)}{2}-\frac{d(m-1)}{2\alpha}}\|D^{\frac{\alpha}{2}}(\widetilde{\vartheta}_r v_n)\|_{L^2}^{\frac{d(m-1)}{\alpha}},
\end{aligned}
\end{equation}
where we used \eqref{eq:estab4.0} to get
\begin{equation}\label{eq:estab6}
\begin{aligned}
\|\widetilde{\vartheta}_r v_n\|_{L^2}^2 \lesssim \int_{\{\frac{r}{2}<|x-z_n|< 3r\}} |v_n|^2 \lesssim \delta.
\end{aligned}
\end{equation}
The estimate for $\|D^{\frac{\alpha}{2}}(\widetilde{\vartheta}_r v_n)\|_{L^2}$ follows from the same argument as in \eqref{eq:estab4.2}. Indeed, we have
\begin{equation}\label{eq:estab6.0}
    \begin{aligned}
    \|D^{\frac{\alpha}{2}}(\widetilde{\vartheta}_r v_n)\|_{L^2}\lesssim \|[D^{\frac{\alpha}{2}},\widetilde{\vartheta}_r]v_n\|_{L^2}+\|\widetilde{\vartheta}_r D^{\frac{\alpha}{2}}v_n\|_{L^2}.
    \end{aligned}
\end{equation}
To use Proposition \ref{fractionalDeriv},  we write
\begin{equation*}
\begin{aligned}
\, [D^{\frac{\alpha}{2}},\widetilde{\vartheta}_r]v_n=R^{\ast}(\widetilde{\vartheta}_r,v_n)
\end{aligned}
\end{equation*}
with the upper bound
\begin{equation*}
    \|R^{\ast}(\widetilde{\vartheta}_r,v_n)\|_{L^2}\lesssim \|D^{\frac{\alpha}{2}}(\widetilde{\vartheta}_r)\|_{L^{\infty}}\|v_n\|_{L^2}\lesssim r^{-\frac{\alpha}{2}}\|v_n\|_{L^2}.
\end{equation*}
Once again, Lemma \ref{lemmaUnibound} provides a uniform constant for the above inequality independent of $n$.

By noticing that $\supp(\varphi_r^{m+1}-\varphi_r^2)\subset\{x: \, r\leq |x|\leq 2r\}$, the above arguments also establish the same control for the fifth term on the right-hand side of \eqref{eq:estab4.1}. More precisely, we obtain 
\begin{equation}
\Big|\int (\vartheta_r^{2}-\vartheta_r^{m+1})v_n^{m+1}  \Big|+\Big|\int (\varphi_r^{2}-\varphi_r^{m+1})v_n^{m+1}  \Big| \lesssim \delta^{\frac{m+1}{2}-\frac{d(m-1)}{2\alpha}}.
\end{equation}
Going back to \eqref{eq:estab4.1}, collecting the previous estimates, and taking $r$ large enough, we complete the proof of part (c). 
\end{proof}

\begin{corollary}\label{contrsubaddprop} 
Let $0<\alpha < 2$ and $1<m<\frac{2\alpha}{d}+1$. If $0<\lambda<q$, then $\mathcal{I}_q\geq \mathcal{I}_{\lambda}+\mathcal{I}_{q-\lambda}$.
\end{corollary}

\begin{proof}
We first notice that if $|M[g]-\lambda|<\delta$, then setting $\beta^2=\frac{\lambda}{M[g]}$, we have $M[\beta g]=\lambda$, and $|\beta-1|<C_1 \delta$, where $C_1$ is independent of $g$ and $\delta>0$. Thus, we get
\begin{equation}
\mathcal{I}_{\lambda} \leq E[\beta g]=E[g]+C_2 \delta,
\end{equation}
where $C_2$ depends only on $\|g\|_{H^{\frac{\alpha}{2}}}$ and $C_1$. A similar reasoning works for a function $h\in H^{\frac{\alpha}{2}}(\mathbb{R}^d)$ with $|M[h]-(q-\lambda)|<\delta$. Thus, the previous argument and Lemma \ref{dichotomyLemma} imply that there exists a subsequence $\{v_{n_k}\}$ of $\{v_{n}\}$ with corresponding subsequences $\{h_{n_k}\}$ of $\{h_{n}\}$ and $\{g_{n_k}\}$ of $\{g_{n}\}$ such that
\begin{equation}
\begin{aligned}
&E[g_{n_k}]\geq \mathcal{I}_{\lambda}-\frac{1}{k}, \\
&E[h_{n_k}]\geq \mathcal{I}_{q-\lambda}-\frac{1}{k}, \\
&E[v_{n_k}] \geq  E[g_{n_k}]+E[h_{n_k}]-\frac{1}{k}.
\end{aligned}
\end{equation}
Hence,
\begin{equation}
E[v_{n_k}] \geq  \mathcal{I}_{\lambda}+\mathcal{I}_{q-\lambda}-\frac{3}{k}.
\end{equation}
Taking the limit $n_k \to \infty$ in the inequality above yields the desired conclusion. 
\end{proof}

Observe that Corollary \ref{contrsubaddprop} contradicts the subadditivity property established in Lemma \ref{subaddprop}. This rules out the case of dichotomy. 
We now  show that the vanishing case  does not occur either.

\begin{lemma}\label{rulevanish}   
Let $0<\alpha < 2$ and $1<m<\frac{2\alpha}{d}+1$. For every minimizing sequence, $\lambda>0$.
\end{lemma}

\begin{proof}
By Lemma \ref{lemmaUnibound} and \ref{lemmavanish} there exist $\eta>0$ and a sequence $\{z_n\} $ such that
\begin{equation}
\begin{aligned}
\int_{B(z_n,2 \sqrt{d}) }|v_n|^{m+1}  \geq \eta, \, \, \text{ for all } \, \, n.
\end{aligned}
\end{equation}
Let $\phi\in C^{\infty}_c(\mathbb{R}^d)$ be such that $0\leq \phi \leq 1$, $\phi(x)=1$ when $|x|\leq 2\sqrt{d}$,  
and $\supp \phi \subseteq \{ |x|\leq 3 \sqrt{d} \}$. 
Then by Lemma \ref{lemmaUnibound} and Gagliardo-Nirenberg inequality (Proposition \ref{GNineq}), we get
\begin{equation}
\begin{aligned}
\eta \leq \int_{B(z_n,2\sqrt{d})} |v_n|^{m+1}  &\lesssim \|\phi(\cdot-z_n) v_n\|_{L^2}^{(m+1)-\frac{d(m-1)}{\alpha}}\|D^{\frac{\alpha}{2}}(\phi(\cdot-z_n)v_n)\|_{L^{2}}^{\frac{d(m-1)}{\alpha}} \\
&\leq C \|v_n\|_{L^2(B(z_n,3\sqrt{d} \,))}^{(m+1)-\frac{d(m-1)}{\alpha}},
\end{aligned}
\end{equation}
where the term $\|D^{\frac{\alpha}{2}}(\phi(\cdot-z_n)v_n)\|_{L^{2}}$ is estimated the same as in \eqref{eq:estab6.0}. Therefore, we obtain
\begin{equation}
\lambda= \lim_{r\to \infty} \rho(r) \geq \rho(3 \sqrt{d}\, )=\lim_{n \to \infty} \rho_n(3 \sqrt{d} \,)\geq \frac{1}{2}\Big(\frac{\eta}{C}\Big)^{\frac{2\alpha}{(m+1)\alpha-d(m-1)}}>0.
\end{equation}
\end{proof}
We summarize that the cases of vanishing and dichotomy do not occur and draw the following conclusion.

\begin{theorem}\label{stabilTheorem-1}
Let $0 <\alpha<2$ and $1<m<\frac{2\alpha}{d}+1$.  
For every $q>0$ there exists a nonempty set $\mathcal{G}_q$ of minimizers of \eqref{min:prob}. 
Moreover, if $\{v_n\}$ is a minimizing sequence for $\mathcal{I}_q$, then 
\begin{itemize}
    \item[(i)] there exist a sequence $\{z_n\}$ and an element $g\in \mathcal{G}_q$ such that $\{v_n(\cdot+z_n)\}$ has a subsequence converging strongly in $H^{\frac{\alpha}{2}}(\mathbb{R}^d)$ to $g$,
\item[(ii)] 
$$
\lim_{n\to \infty}\inf_{g\in \mathcal{G}_q, z\in \mathbb{R}^d}\|v_n(\cdot+z)-g\|_{H^{\frac{\alpha}{2}}}=0,
$$
\item[(iii)] 
$$
\lim_{n\to \infty}\inf_{g\in \mathcal{G}_q}\|v_n-g\|_{H^{\frac{\alpha}{2}}}=0.
$$
\end{itemize}

\end{theorem}

\begin{proof}
Gathering Lemmas \ref{subaddprop}, \ref{dichotomyLemma}, \ref{rulevanish}, and Corollary \ref{contrsubaddprop}, it must be the case $\lambda=q$. Thus, Lemma \ref{lemmacompactres} establishes part (i). The part (iii) follows from (ii) and the fact that the functionals $E[\cdot]$ and $M[\cdot]$ are invariant under translation, more precisely, if $g \in \mathcal{G}_q$, then  any translation of $g$ is also in $\mathcal{G}_q$. 

The part (ii) follows by contradiction generalizing the proof of \cite[Theorem 2.9]{Albert1999} to all dimensions.
If (ii) does not hold, then there exists a subsequence $\{v_{n_k}\}$ of $\{v_n\}$ and a number $\delta>0$ such that
\begin{equation}\label{contradieq1}
    \inf_{g\in \mathcal{G}_q, z\in \mathbb{R}^d}\|v_{n_k}(\cdot+z)-g\|_{H^{\frac{\alpha}{2}}}\geq \delta
\end{equation}
for all $k\geq 1$ integer. However, $\{v_{n_k}\}$ is a minimizing sequence of $\mathcal{I}_q$, and thus, from part (i), there exists a sequence $\{z_k\}$ and $g_0\in \mathcal{G}_q$ such that
\begin{equation*}
    \liminf_{k\to \infty}\|v_{n_k}(\cdot+z_k)-g_0\|_{H^{\frac{\alpha}{2}}}=0.
\end{equation*}
This is a contradiction to \eqref{contradieq1}, which completes the proof of (ii).
\end{proof}

Before we proceed to the discussion of the stability of the set $\mathcal{G}_q$, we recall some well-posedness results as it was mentioned that in some cases the desired well-posedness (in $H^{\frac{\alpha}2}$) is not yet established. 

Recall that when $m>1$ is an integer, by a standard parabolic regularization argument, the Cauchy problem associated to \eqref{EQ:fKdV} is locally well-posed in $H^{s}(\mathbb{R}^d)$ with $s> \frac{d}{2}+1$. Briefly, this argument consists of adding an extra term $\mu \Delta u$ to the right-hand side of \eqref{EQ:fKdV}, after which the limit $\mu \to 0$ is taken in an appropriate space. 

\begin{lemma}\label{L:Regular}
Let $s>\frac{d}{2}+1$, $0<\alpha<2$ and $m>1$ be an integer. Then for any $u_0 \in H^s(\mathbb{R}^d)$ there exists a maximal time of existence $T^{\ast}=T^{\ast}(\|u_0\|_{H^s})>0$ and a unique solution $u\in C([0,T^{\ast}); H^s(\mathbb{R}^d))$ of the initial value problem associated to \eqref{EQ:fKdV}. 
\end{lemma}

\begin{proof}
The proof closely follows the arguments for the one-dimensional models in \cite{ABFS1989, Iorio1986}. For the reader's convenience we provide a short overview of the proof. For a fixed initial condition $u_0\in H^s(\mathbb{R}^d)$, the first step is to find solutions to the regularized equation, obtained by adding the term $\mu \Delta u$ to the right-hand side of \eqref{EQ:fKdV}.  Such solutions exist for all $\mu>0$, by a contraction argument, where this extra factor $\mu \Delta u$ is essential to dominate the dispersive effects and to be able to control the extra derivative in the nonlinear term $\partial_{x_1}(u^m)$. 

Next, note that using the Kato-Ponce commutator estimate \cite{KatoPonce1988}, and the fractional Leibniz rule for the operator $J^s$ (e.g., see \cite[Theorem 1]{GrafakosOh2014}), from the energy estimates for \eqref{EQ:fKdV} one obtains
\begin{equation}
\frac{d}{dt}\|u(t)\|_{H^s}^2 \leq c \|u(t)\|^{m-2}_{L^{\infty}}\|\nabla u(t)\|_{L^{\infty}}\|u(t)\|_{H^s}^2,  
\end{equation}
from which Gronwall's inequality yields
\begin{equation}\label{E:G}
 \sup_{t\in [0,T]} \|u(t)\|_{H^s}^2 \leq \|u_0\|_{H^s}^2 \, e^{c \int_0^T \|u(\tau)\|^{m-2}_{L^{\infty}}\|\nabla u(\tau)\|_{L^{\infty}} \, d\tau}, 
\end{equation}
for some  constant $c>0$ and all $T>0$ (for similar ideas in the case $m=2$, $\alpha=1$, see \cite{HLORW2019}). Using the Sobolev embedding $H^{s}(\mathbb{R}^d) \hookrightarrow L^{\infty}(\mathbb{R}^d)$, $s>\frac{d}{2}+1$, the argument in the above exponential in \eqref{E:G} is controlled via the $H^s$-norm. Next, note that similar estimates are valid with a constant independent of $\mu$ for the solutions of the regularized equation \eqref{EQ:fKdV}. Hence, one can show that the existence time of solutions of the regularized equation is independent of $0<\mu<1$, and such solutions are uniformly bounded in $H^s(\mathbb{R}^d)$. Thus, using these previous estimates, the convergence as $\mu\to 0^+$ follows (for further details, see \cite{ABFS1989, Iorio1986}).
\end{proof}

We are now ready to describe the stability of the set of minimizers $\mathcal{G}_q$.

\begin{theorem}\label{stabilTheorem-2}
Let $0 <\alpha<2$ and $1<m<\frac{2\alpha}{d}+1$ is an integer. 
The set $\mathcal{G}_q$ is stable, that is, for any $\omega>0$ there exists $\delta>0$ such that if $u_0\in H^{s}(\mathbb{R}^d)$, $s>\frac{d}{2}+1$, and
\begin{equation*}
\inf_{g\in \mathcal{G}_q}\|u_0-g\|_{H^{\frac{\alpha}{2}}}<\delta,
\end{equation*}
then the corresponding solution $u \in C([0,T^{\ast});H^{s}(\mathbb{R}^d))$ of \eqref{EQ:fKdV} with the initial condition $u_0$ satisfies
\begin{equation}
\inf_{g\in \mathcal{G}_q}\|u(\cdot,t)-g\|_{H^{\frac{\alpha}{2}}}<\omega,
\end{equation}
for all $t>0$
(in cases, where the global well-posedness is not yet established in a corresponding $H^s$, then it should read $0<t<T^{\ast}$, where $0<T^{\ast} \leq \infty$ is the maximal existence time of solution $u(t)$).
\end{theorem}

\begin{proof}
Assume to the contrary that the stability of $\mathcal{G}_q$ does not hold. Then, there exist $\omega>0$, a sequence $\{u_{0,n}\}$ in $H^{\frac{\alpha}{2}}(\mathbb{R}^d)$ and a sequence of times $\{t_n\}$ such that
\begin{equation}
\inf_{g\in \mathcal{G}_q}\|u_{0,n}-g\|_{H^{\frac{\alpha}{2}}}<\frac{1}{n},
\end{equation}
and 
\begin{equation}
\inf_{g\in \mathcal{G}_q}\|u_n(\cdot,t_n)-g\|_{H^{\frac{\alpha}{2}}}\geq \omega,
\end{equation}
for all $n$, where we denote by $u_{n}(x,t)$ the solution of \eqref{EQ:fKdV} with the initial condition $u_{0,n}$. Since $u_{0,n}$ is converging in $\mathcal{G}_q$ in $H^{\frac{\alpha}{2}}(\mathbb{R}^d)$, it follows that $E[u_{0,n}] \to \mathcal{I}_q$ and $M[u_{0,n}]\to q$. Then for each $n$ we choose $\theta_n $ such that $M[\theta_n u_{0,n}]=q$, hence, it must follow that $\theta_n\to 1$. Setting $f_n=\theta_n u_n(\cdot,t_n)$, we have $M[f_n]=q$ and
\begin{equation}
\lim_{n\to \infty} E[f_n]=\lim_{n\to \infty}E[u_n(\cdot,t_n)]=\lim_{n\to \infty}E[u_{0,n}]=\mathcal{I}_q.
\end{equation}
Then $\{f_n\}$ is a minimizing sequence for $\mathcal{I}_q$. Note that Theorem \ref{stabilTheorem-1} assures that for all $n$ sufficiently large there exists $g_n \in \mathcal{G}_q$ such that $\|f_n-g_n\|_{H^{\frac{\alpha}{2}}}<\frac{\omega}{2}$. However, this implies that 
\begin{equation}
\begin{aligned}
\omega\leq \|u_n(\cdot,t_n)-g_n\|_{H^{\frac{\alpha}{2}}}&\leq \|u_n(\cdot,t_n)-f_n\|_{H^{\frac{\alpha}{2}}}+\|f_n-g_n\|_{H^{\frac{\alpha}{2}}} \\
&\leq|1-\theta_n|\|u_n(\cdot, t_n)\|_{H^{\frac{\alpha}{2}}}+\frac{\omega}{2},
\end{aligned}
\end{equation}
which yields a contradiction when $n \to \infty$ and completes the proof of the theorem.
\end{proof}

Next, we relate the set $\mathcal{G}_q$ with the minimizers of the Gagliardo--Nirenberg inequality \eqref{GNineq}.

\begin{lemma}\label{caraGround} 
Let $q>0$, $0<\alpha<2$ and $1<m<\frac{2\alpha}{d}+1$. Any minimizer $f$ of $\mathcal{I}_q$ can be written as 
\begin{equation*}
f(x)=c^{\frac{1}{m-1}}Q(c^{\frac{1}{\alpha}}(x-z)),
\end{equation*}
for some $z\in \mathbb{R}^d$ and $Q$ a ground state solution of \eqref{NLEQ}
with the parameter $c>0$ chosen such that $\frac{1}{2}\int f^2=q$. 
\end{lemma}

\begin{proof}
By Theorem \ref{existTHR} and scaling, there exists a positive minimizer $Q_{c^{\ast}}$ 
of the functional \eqref{Wfunctional}, solving \eqref{EQ:groundState} such that $M[Q_{c^{\ast}}]=q$, and
\begin{equation}\label{eq:estab6.1}
E[Q_{c^{\ast}}] = s_c \, \frac{2c^{\ast}(m-1)}{(2d-(d-\alpha)(m+1))} \, q,
\end{equation}
which holds by Lemma \ref{PHident}. Furthermore, in terms of the minimizers of the Weinstein functional \eqref{Wfunctional}, we have
\begin{equation}
\begin{aligned}
\mathcal{W}(Q_{c^{\ast}})&=\frac{\| D^{\frac{\alpha}{2}}Q_{c^{\ast}}\|_{L^2}^{\frac{d(m-1)}{\alpha}}\|Q_{c^{\ast}}\|_{L^2}^{(m+1)-\frac{d(m-1)}{\alpha}}}{\|Q_{c^{\ast}}\|_{L^{m+1}}^{m+1}}\\
&=\frac{d^{\frac{d(m-1)}{2\alpha}}(m-1)^{\frac{d(m-1)}{2\alpha}}\Big(2d-(d-\alpha)(m+1)\Big)^{\frac{2\alpha-d(m-1)}{2\alpha}}}{2^{\frac{1-m}{2}}m\alpha(m+1)(c^{\ast})^{\frac{2\alpha-d(m-1)}{2\alpha}}} \, q^{\frac{m-1}{2}}.
\end{aligned}
\end{equation}
On the other hand, if $f$ is a minimizer of $\mathcal{I}_q$, by Lagrange multiplier theory, there exists a constant $\theta_q\in \mathbb{R}$ such that
\begin{equation}
D^{\alpha}f-\frac{1}{m}f^m+\theta_q f=0.
\end{equation}
Then, Lemma \ref{PHident} yields
\begin{equation}\label{eq:estab7}
\|D^{\frac{\alpha}{2}}f\|_{L^2}^2=\frac{2d\theta_q (m-1)}{(2d-(d-\alpha)(m+1))}\, q,
\end{equation}
\begin{equation}\label{eq:estab8}
\int f^{m+1}=\frac{2\alpha\theta_q m(m+1)}{(2d-(d-\alpha)(m+1))}\, q,
\end{equation}
and
\begin{equation}\label{eq:estab9}
E[f] = s_c \, \frac{2\theta_q (m-1)}{(2d-(d-\alpha)(m+1))} \, q.
\end{equation}
From \eqref{eq:estab7} it follows that $\theta_q>0$, and by \eqref{eq:estab8} 
\begin{equation}
\int f^{m+1} \geq 0.
\end{equation}
We claim that $f$ is positive. Indeed, since $\|D^{\frac{\alpha}{2}}(|f|)\|_{L^2}\leq \|D^{\frac{\alpha}{2}}f\|_{L^2}$ (see (2.10) in \cite{FelmerQuaTan2012}), then $E[|f|]\leq E[f]$ and $M[|f|]=q$, hence, it implies that $f$ is positive, otherwise, we would have $E[|f|]<E[f]=\mathcal{I}_q$, contrary to $f$ being a minimizer of $\mathcal{I}_q$.

Applying \eqref{eq:estab7} and \eqref{eq:estab8} to compute $\mathcal{W}(f)$, the inequality $\mathcal{W}(f)\geq \mathcal{W}(Q_{c^{\ast}})$ implies that $ c^{\ast}\geq \theta_q$, however, from \eqref{eq:estab6.1} and \eqref{eq:estab9} we would have $ \theta_q\geq c^{\ast}$. This gives $\mathcal{W}(f)= \mathcal{W}(Q_{c^{\ast}})$. Then Lemma \ref{existTHR} and the uniqueness result in \cite{FLS2016, FL2013} as discussed in Section \ref{sectionPrelimDecay} show
that $f=Q_{c^{\ast}}(\cdot-z)$ for some $z\in \mathbb{R}^d$. Lastly, observe that $Q_{c^{\ast}}(x)=(c^{\ast})^{\frac{1}{m-1}}Q((c^{\ast})^{\frac{1}{\alpha}}x)$, where $Q$ is a ground state solution of \eqref{NLEQ}, thus, completing the proof.
\end{proof}

Finally, as a consequence of Theorem \ref{stabilTheorem-2} and Lemma \ref{caraGround}, we obtain the orbital stability result claimed in Theorem \ref{OrbstabilThm} in the subcritical setting of fKdV \eqref{EQ:fKdV}.


\section{Instability: Preliminaries}
\label{InstaPrelimC}

We now would like to investigate the instability of traveling solitary waves of \eqref{EQ:fKdV}. We use the approach to prove the instability of solitary waves in gKdV and its higher dimensional generalization Zakharov-Kuznetsov equations (where the dispersion is given by local operators) as in  \cite{Combet2010,FHR2019a,FHR2019b,FHR2019c}.

Recall the unique positive ground state solution $Q$ of 
\eqref{NLEQ} and the linearized around $Q$ operator, 
which has the following properties. 

\begin{theorem}[Properties of $L$] \label{propL} The following holds for the  operator $L$: \begin{itemize}
\item 
$\ker L=\spn \{\partial_{x_1} Q,\partial_{x_2} Q,\dots, \partial_{x_d} Q\}$.

\item 
$L$ has a unique single negative eigenvalue $-\lambda_0$ (with $\lambda_0>0$) associated to a positive, radially symmetric eigenfunction $\chi_0$.  Moreover,
    \begin{equation}
|\chi_0(x_1,\dots,x_d)|\lesssim \langle x \rangle^{-\alpha-d}, \, \, \text{ for all }\,\, x=(x_1,\dots,x_d)\in \mathbb{R}^d.
    \end{equation}

\item (Orthogonality Condition I) 
There exists some $c_0>0$ such that for all $f\in H^{\frac{\alpha}{2}} (\mathbb{R}^d)$ with $f$ orthogonal to the set~ $\spn\{\chi_0,\partial_{x_1}Q,\partial_{x_2}Q, \dots, \partial_{x_d}Q\}$, that is,
    \begin{equation}
    (f,\chi_0)=(f,\partial_{x_j}Q)=0, \qquad \text{ for all }\, j=1,\dots, d
    \end{equation}
one has
\begin{equation}\label{Coercond}
(Lf,f)\geq c_0 \|f\|_{H^{\frac{\alpha}{2}}}^2.
\end{equation}

\item (Orthogonality Condition II) 
There exists $k_1,k_2>0$ such that for all $\epsilon \in H^{\frac{\alpha}{2}} (\mathbb{R}^d)$, satisfying $\epsilon \bot \{\partial_{x_1} Q, \partial_{x_2} Q,\dots, \partial_{x_d} Q\}$, one has
\begin{equation}
(L\epsilon,\epsilon)=\int |D^{\frac{\alpha}{2}}\epsilon|^2+\int \epsilon^2-\int Q^{m-1}\epsilon^2 \geq k_1\|\epsilon\|_{H^{\frac{\alpha}{2}}}^2-k_2|(\epsilon,\chi_0)|^2.
\end{equation}
\end{itemize}
\end{theorem}

\begin{proof}
The first three properties in Theorem \ref{propL} are proved in \cite{FL2013,FLS2016}. We show the Orthogonality Condition II. Setting $a=(\epsilon,\chi_0)\|\chi_0\|_{L^2}^{-2}$ and $\epsilon_1=\epsilon-a \chi_0$, by \eqref{Coercond}, we have
\begin{equation}\label{E:4.5}
\begin{aligned}
(L \epsilon_1,\epsilon_1) 
\geq & \, c_0\|\epsilon\|_{H^{\frac{\alpha}{2}}}^{2}-2c_0 a(\epsilon,\chi_0)_{H^{\frac{\alpha}{2}}}+c_0 a^2\|\chi_0\|_{H^{\frac{\alpha}{2}}}^{2} \\
\geq & \, \frac{c_0}{2}\|\epsilon\|_{H^{\frac{\alpha}{2}}}^2+(c_0-c_1)a^{2}\|\chi_0\|_{H^{\frac{\alpha}{2}}}^2,
\end{aligned}
\end{equation}
where by $(\cdot,\cdot)_{H^{\frac{\alpha}{2}}}$ we denoted the inner product in  $H^{\frac{\alpha}{2}}(\mathbb{R}^d)$ and used Cauchy-Schwarz and Young's inequalities
$$
|2c_0 a(\epsilon,\chi_0)_{H^{\frac{\alpha}{2}}}|\leq c_1 a^2\|\chi_0\|_{H^{\frac{\alpha}{2}}}^2+\frac{c_0}{2}\|\epsilon\|_{H^{\frac{\alpha}{2}}}^2,
$$
for some constant $c_1>0$. On the other hand, 
\begin{equation} \label{E:4.7}
\begin{aligned}
(L\epsilon_1,\epsilon_1)=&(L\epsilon,\epsilon)-2a(\epsilon,L\chi_0)+a^2(L\chi_0, \chi_0)\\
=&(L\epsilon,\epsilon)+2a\lambda_0\, (\epsilon,\chi_0)-a^2\lambda_0\|\chi_0\|_{L^2}^2.
\end{aligned}
\end{equation}
Putting together \eqref{E:4.5} and \eqref{E:4.7}, and recalling that $a=(\epsilon,\chi_0)\|\chi_0\|_{L^2}^{-2}$, one can find a universal constant $\widetilde{c}_1 \in \mathbb{R}$ such that
\begin{equation}
\begin{aligned}
(L \epsilon, \epsilon) \geq \frac{c_0}{2}\|\epsilon\|_{H^{\frac{\alpha}{2}}}^{2}+\widetilde{c}_1|(\epsilon,\chi_0)|^2,
\end{aligned}
\end{equation}
completing the proof.
\end{proof}
We also require some regularity and spatial decay properties of the function $\chi_0$. 
\begin{lemma}\label{LemmaChiprop} 
Let $0< \alpha<2$ and $0<m<m_{\ast}$ be an integer with $m_{\ast}$ defined by \eqref{defcritm}. Let $\chi_0$ be a positive radial symmetric eigenfunction associated to the eigenvalue $-\lambda_0$, $\lambda_0>0$ of the operator $L$. Then it follows  $\chi_0\in H^{\infty}(\mathbb{R}^d)=\bigcap_{s>0}H^{s}(\mathbb{R}^d)$. Moreover, for any multi-index $\beta$,
\begin{equation}\label{decderchi}
    |\partial^{\beta} \chi_0(x)|\lesssim \langle x \rangle^{-\alpha-d-|\beta|}.
\end{equation}
\end{lemma}
\begin{proof}
Since $L\chi_0=-\lambda_0\chi_0$, we have
 \begin{equation}\label{eqChi}
        \begin{aligned}
        D^{\beta}\chi_0=&\big(D^{\alpha}+1+\lambda_0\big)^{-1}D^{\beta}(Q^{m-1}\chi_0)\\
        =&\widetilde{G}_{\alpha,\lambda_0}\ast D^{\beta}(Q^{m-1}\chi_0),
        \end{aligned}
    \end{equation}
for any $\beta>0$, where $\widetilde{G}_{\alpha,\lambda_0}\in L^1(\mathbb{R}^d)$ denotes the kernel associated to the operator $\big(D^{\alpha}+1+\lambda_0\big)^{-1}$, which satisfies \eqref{Decayderieq1} (see for example \cite[Lemma C.1]{FLS2016}). Thus, since $Q^{m-1}\chi_0\in L^2(\mathbb{R}^d)$, setting $\beta=\alpha$ in \eqref{eqChi}, we have $\chi_0\in H^{\alpha}(\mathbb{R}^d)$. We can argue by induction on $\beta=k\alpha$, with $k\geq 0$ integer 
to obtain that $\chi_0\in H^{\infty}(\mathbb{R}^d)$. A key observation here is that if $\chi_0\in H^{k\alpha}(\mathbb{R}^d)$, then $Q^{m-1}\chi_0\in H^{k\alpha}(\mathbb{R}^d)$, which can be easily obtained as a consequence of the fractional Leibniz rule \eqref{fLR} as it is done in the proof of Lemma \ref{LemmaQprop}.

Once we have established that $\chi_0\in H^{\infty}(\mathbb{R}^d)$, using \eqref{eqChi} and the spatial decay properties of $Q$, we can follow the arguments in the proof of Lemma \ref{LemmaQprop} to obtain \eqref{decderchi}. We remark that Lemmas \ref{lemmadeca}, \ref{decayQm}, \ref{conmident}, and \ref{kernDecay} are fundamental to obtain this inequality.
\end{proof}
Next, we define the scaling generator operator $\Lambda$ by
\begin{equation}\label{opeLam}
\Lambda f=\frac{1}{m-1}f+\frac{1}{\alpha} x\cdot \nabla f.
\end{equation}
\begin{lemma}\label{propGamop}
The following identities hold
\begin{itemize}
\item[(a)] $L(\Lambda Q)=-Q$,
\item[(b)] $\int Q \Lambda Q=\frac{2\alpha-d(m-1)}{2\alpha(m-1)}\int Q^2<0$, $\frac{2\alpha}{d}+1<m$, and $\int Q\Lambda Q=0$ if $m=\frac{2\alpha}{d}+1$.
\end{itemize}
\end{lemma}

\begin{proof}
Part (a) follows from the fact that $Q$ solves \eqref{NLEQ} and the identity 
\begin{equation}\label{conmIden1}
[D^{\alpha},x_j]\partial_{x_j}f=-\alpha D^{\alpha-2}\partial_{x_j}^2f,
\end{equation}
for all $j=1,\dots,d$. These identities are used to operate the fractional derivative $D^{\alpha}$. Part (b) is obtained by integration by parts.  
\end{proof}

Next, we introduce a Lyapunov-type functional or action functional $W$ (also called Weinstein functional) and its expansion around $Q$.

\begin{lemma}
\label{lemWF} Let $0<\alpha<2$, and $1<m<m_{\ast}$ be an integer with $m_{\ast}$ defined in \eqref{defcritm}. Let $Q$ be a positive $H^{\frac{\alpha}2}$ solution of \eqref{NLEQ}.  Recall \eqref{E:mass}, \eqref{E:energy}, and define
\begin{equation*}
W[u]=E[u]+M[u]
\end{equation*}
Then
\begin{equation}
W[Q+\epsilon]=W[Q]+\frac{1}{2}(L \epsilon,\epsilon)+K[\epsilon]
\end{equation}
with $K: H^{\frac{\alpha}{2}}(\mathbb{R}^d)\rightarrow \mathbb{R}$ such that
\begin{equation}\label{Kope1}
|K[\epsilon]|\lesssim \sum_{k=3}^{m+1} \|D^{\frac{\alpha}{2}}\epsilon\|_{L^2}^{\frac{d(k-2)}{\alpha}}\|\epsilon\|_{L^2}^{k-\frac{d(k-2)}{\alpha}}.
\end{equation}
\end{lemma}

\begin{proof}
A simple computation and the fact that $Q$ solves \eqref{NLEQ} yield 
\begin{equation}\label{eq:insta1}
\begin{aligned}
E[Q+\epsilon]&=\frac{1}{2}\int |D^{\frac{\alpha}{2}}(Q+\epsilon)|^2-\frac{1}{m(m+1)}\int (Q+\epsilon)^{m+1}\\
&=E[Q]+\frac{1}{2}\int D^{\alpha}\epsilon \epsilon+\int D^{\alpha}Q \epsilon-\frac{1}{m}\int Q^m \epsilon-\frac{1}{2}\int Q^{m-1}\epsilon^2+K[\epsilon]\\
&=E[Q]+\frac{1}{2}\int D^{\alpha}\epsilon \epsilon-\int Q\epsilon-\frac{1}{2}\int Q^{m-1}\epsilon^2+K[\epsilon],
\end{aligned}
\end{equation}
where 
\begin{equation}\label{E:4.15}
K[\epsilon]=-\frac{1}{m(m+1)}\sum_{k=3}^{m+1} \int  \binom{m+1}{k}Q^{m+1-k}\epsilon^k.
\end{equation}
Since $Q$ is in $L^{\infty}(\mathbb{R}^d)$, we have
\begin{equation*}
|K[\epsilon]|\lesssim \sum_{k=3}^{m+1}\|\epsilon\|^{k}_{L^k},
\end{equation*}
and the Gagliardo-Nirenberg inequality \eqref{GNineq} applied to the above right-hand side yields \eqref{Kope1}, completing the proof.
\end{proof}
Note that in the above proof it is essential that $m$ is an integer, it is used for a binomial expansion into a sum in \eqref{E:4.15}. 


\subsection{Decomposition of $u$ and modulation theory}

We consider the canonical parametrization of the solution $u(x,t)$ close to $Q$:
\begin{equation}\label{eqMT1}
\epsilon(x, t) 
=u(x_1+z_1(t),x_2+z_2(t),\dots,x_d+z_d(t),t)-Q(x),
\end{equation}
where $x=(x_1,\dots,x_d)\in \mathbb{R}^d$, $u$ solves \eqref{EQ:fKdV} and $z_j(t)$, $j=1,\dots,d$, are $C^1$ functions with bounded derivatives to be determined later. In the following lemma, we obtain the equation for $\epsilon(x,t)$.

\begin{lemma}[Equation for $\epsilon$]\label{Lemeps}
Let $0<\alpha<2$ and $1<m<m_{\ast}$ be an integer. Let $k_{\alpha}\geq 0$ be such that $\frac{d}{2}-\alpha<k_{\alpha}<d-\alpha$ if $d\geq 2$, and set $k_{\alpha}=1$ if $d=1$. Let $u\in L^{\infty}([0,\infty);H^{\frac{\alpha}{2}}(\mathbb{R}^d))$ be a solution of  \eqref{EQ:fKdV}. Then
\begin{equation}\label{eqMT2}
\partial_t\epsilon-\partial_{x_1} L\epsilon=(z_1'(t)-1) \, \partial_{x_1}(Q+\epsilon) +\sum_{j=2}^d z_j'(t) \, \partial_{x_j}(Q+\epsilon)-R(\epsilon),
\end{equation}
where the remainder $R(\epsilon)$ contains higher order terms of $\epsilon$
\begin{equation}\label{eqMT2.1}
R(\epsilon)=\frac{1}{m}\partial_{x_1}\Big( \sum_{k=2}^{m} \binom{m}{k} Q^{m-k}\epsilon^k\Big),
\end{equation}
and the equation \eqref{eqMT2} makes sense in the $H^{-(1+k_{\alpha}+\alpha)}(\mathbb{R}^d)$ topology. (In the one-dimensional case, the convention is that the second sum on the right-hand side of \eqref{eqMT2} is zero.)  
\end{lemma}

\begin{proof}
We rewrite \eqref{eqMT1} as
\begin{equation}
u(x_1,\dots,x_d,t)=Q(x_1-z_1(t),\dots,x_d-z_d(t))+\epsilon(x_1-z_1(t),\dots,x_d-z_d(t),t).
\end{equation}
Since $u$ and $Q$ solve \eqref{EQ:fKdV} and \eqref{NLEQ}, respectively, from the above decomposition and expanding $(Q+\epsilon)^m$, we get \eqref{eqMT2}. All of the previous arguments can be justified in the $H^{-(1+k_{\alpha}+\alpha)}(\mathbb{R}^d)$ topology. Indeed, since $Q\in H^{\infty}(\mathbb{R}^d)$ and
\begin{equation}
\begin{aligned}
&\partial_t u\in L^{\infty}((0,\infty);H^{-(1+k_{\alpha}+\alpha)}(\mathbb{R}^d)),\\
 &\partial_{x_1} L u \in L^{\infty}((0,\infty);H^{-(1+k_{\alpha}+\alpha)}(\mathbb{R}^d)),
\end{aligned}
\end{equation}
it is not difficult to check the validity of the left-hand side of \eqref{eqMT2} in $H^{-(1+k_{\alpha}+\alpha)}(\mathbb{R}^d)$. For example, to justify the nonlinear term in \eqref{eqMT2} for $1\leq k \leq m$, we have
\begin{equation}
\|J^{-1-k_{\alpha}-\alpha}\partial_{x_1} \big(Q^{m-k}\epsilon^k\big)\|_{L^2}\leq 
\|J^{-k_{\alpha}-\alpha}\big(Q^{m-k}\epsilon^k\big)\|_{L^2} 
= 
\|G_{k_{\alpha}+\alpha}\ast \big(Q^{m-k}\epsilon^k\big) \|_{L^2},
\end{equation}
where for any $s>0$, $G_{s}$ denotes the kernel associated to the Bessel potential $J^{-s}$, which satisfies the following asymptotics: $G_{s}(x)\sim |x|^{-d+s}$ as $|x|\to 0$, if $0<s<d$; $G_{s}(x)\sim -\ln|x|$ as $|x|\to 0$, if $s=d$; $G_{s}(x)\sim 1$ as $|x|\to 0$, if $s>d$; and $G_s(x)=O(e^{-c|x|})$ as $|x|\to \infty$, for any $s>0$ (see for instance \cite[Chapter V, page 132]{SteinSingInte1970}). Then, the definition of $k_{\alpha}$ and the decay properties of $G_{k_{\alpha}+\alpha}$ imply that $G_{k_{\alpha}+\alpha}\in L^{2}(\mathbb{R}^d)$, which combined with Young's inequality, allow us to conclude
\begin{equation}\label{E:4.22}
\|J^{-k_{\alpha}-\alpha}\big(Q^{m-k}\epsilon^k\big)\|_{L^2}\lesssim 
\|G_{k_{\alpha}+\alpha}\|_{L^2}\|Q^{m-k}\epsilon^k\|_{L^1}\\
\lesssim 
\|\epsilon\|_{L^k}^k.
\end{equation}
Using the Gagliardo-Nirenberg inequality, the inequality \eqref{E:4.22} is controlled by the right hand-side of \eqref{Kope1}.
\end{proof}

Next, we introduce the modulation theory close to the ground state solution $Q$. We will use the neighborhood $U_{\omega}$ introduced in \eqref{tubdef}.

\begin{proposition}[Modulation theory]\label{propmodThe} 
Assume $0<\alpha<2$, and $1<m<m_{\ast}$. Then there exists $\omega_1>0$, $C_1>0$ and a unique $C^1$ map
\begin{equation*}
(z_1,\dots,z_d): U_{\omega_1}\mapsto \mathbb{R}^d
\end{equation*}
such that if $u\in U_{\omega_1}$ and $\epsilon_{(z_1(u),\dots,z_d(u))}$ is given by
\begin{equation}\label{eqMT3}
\epsilon_{(z_1(u),\dots, z_d(u))}(x)=u(x_1+z_1(u),\dots,x_d+z_d(u))-Q(x),
\end{equation}
where $x=(x_1,\dots,x_d)\in \mathbb{R}^d$, then
\begin{equation}\label{eqMT3.1}
\epsilon_{(z_1(u),\dots,z_d(u))} \perp \partial_{x_j}Q,
\end{equation}
for all $j=1,\dots, d$. Moreover, if $u\in U_{\omega}$ with $0<\omega<\omega_1$, then
\begin{equation}\label{eqMT4}
\|\epsilon_{(z_1(u),\dots,z_d(u))} \|_{H^{\frac{\alpha}{2}}}\leq C_1 \omega.
\end{equation}
Additionally, when $d\geq 2$, if $u$ is cylindrically symmetric, i.e., $u(x)=u(x_1,|(x_2,\dots,x_d)|)$, then taking $\omega_1$ smaller, if necessary, we can assume $z_j(u)=0$, for all $j=2,\dots,d$.
\end{proposition}

\begin{proof}
The proof is based on the arguments in 
\cite{FHR2019b}, \cite[Lemma 4.4]{Bouard1996}, and \cite{MartelMerle2001}. 
We begin by defining the functional
\begin{equation}
\begin{aligned}
\Phi(z_1,\dots,z_d,u)=&\int_{\mathbb{R}^d} \big(u(x_1+z_1,\dots,x_d+z_d)-Q(x)\big) \nabla Q(x)\, dx \\
=&\Big(\int_{\mathbb{R}^d} \big(u(x_1+z_1,\dots,x_d+z_d)-Q(x)\big)\partial_{x_1} Q(x)\, dx,\\
&\dots,\int_{\mathbb{R}^d} \big(u(x+z_1,\dots,x_d+z_d)-Q(x)\big) \partial_{x_d}Q(x)\, dx  \Big)\\
=:&\big(\Phi_1(z_1,\dots,z_d,u),\dots,\Phi_d(z_1,\dots,z_d,u)\big),
\end{aligned}
\end{equation}
where $(z_1,\dots,z_d,u)\in \mathbb{R}^d\times \{u\in H^{\frac{\alpha}{2}}(\mathbb{R}^d)\, :\,\|u-Q\|_{H^{\frac{\alpha}{2}}}<\omega^{\ast}\}$ for any $\omega* >0$. 
Given that $u \in H^{\frac{\alpha}{2}}(\mathbb{R}^d)$,
\begin{equation}
\begin{aligned}
&\frac{\partial}{\partial z_j}(u(x+z_1,\dots,x_d+z_d)-Q(x))=\partial_{x_j}u(x+z_1,\dots,x_d+z_d)\in H^{\frac{\alpha}{2}-1}(\mathbb{R}^d),
\end{aligned}
\end{equation}
$j=1,\dots, d$, and using that $Q\in H^{\infty}(\mathbb{R}^d)$ (see Lemma \ref{LemmaQprop}),  we find
\begin{equation}
\begin{aligned}
&\frac{\partial \Phi_j}{\partial z_j}\Big|_{(0,\dots,0,Q)}=\|\partial_{x_j} Q\|_{L^2}^2, \qquad \frac{\partial \Phi_j}{\partial z_l}\Big|_{(0,\dots,0,Q)}=0, \, \, \text{ if } \, \, j\neq l, \, \, j,l=1,\dots,d.
\end{aligned}
\end{equation}
Thus, by the implicit function theorem, there exist a sufficiently small $\omega_1>0$ (and $\omega_1 <\omega^{\ast}$), a neighborhood $\mathcal{N}$ of $(0,\dots,0)$ in $\mathbb{R}^d$, and a unique $C^1$ 
(due to the regularity and decay of $Q$) 
map  $(z_1,\dots,z_d): \{u\in H^{\frac{\alpha}{2}}(\mathbb{R}^d)\, :\,\|u-Q\|_{H^{\frac{\alpha}{2}}}<\omega_1\} \mapsto \mathcal{N}$ such that  \eqref{eqMT4} holds true. Moreover, there exists $C>0$ such that if $\|u-Q\|_{H^{\frac{\alpha}{2}}}<\omega<\omega_1$, then $|(z_1,\dots,z_d)|\leq C\omega$. It also implies that $\|\epsilon_{(z_1(u),\dots,z_d(u))} \|_{H^{\frac{\alpha}{2}}}\leq C \omega$ for some $C>0$. 

By a further application of the implicit function theorem and similar reasoning as above, there exists $0<\omega_1<\omega^{\ast}$ such that there exists a $C^1$ function $(\widetilde{z}_1,\dots,\widetilde{z}_d): U_{\omega_1}\mapsto \mathbb{R}^d$, for which
\begin{equation*}
\begin{aligned}
\|u(\cdot+\widetilde{z}_1,\dots,\cdot+\widetilde{z}_d)-Q(\cdot)\|_{H^{\frac{\alpha}{2}}}=\inf_{y\in \mathbb{R}^d} \|u(\cdot+y_1,\dots,\cdot+y_d)-Q(\cdot)\|_{H^{\frac{\alpha}{2}}}\leq \omega_1,
\end{aligned}
\end{equation*}
for all $u\in U_{\omega_1}$. Thus, denoting by $\widetilde{z}(u)=(\widetilde{z}_1(u),\dots,\widetilde{z}_d(u))$, and setting $z^{\ast}_j(u)=z_j(u(\cdot+\widetilde{z}(u))) 
{+\widetilde{z}_j(u)} $, 
$j=1,\dots,d$, we extend  $(z_1,\dots,z_d)$ uniquely to a tube $U_{\omega_1}$ for $\omega_1>0$ sufficiently small. 

Now, arguing as above, we can apply the implicit function theorem to assure the existence of a $C^1$ mapping $\widetilde{z}_1$, defined for $u$ in some neighborhood of $Q$ such that 
\begin{equation*}
\begin{aligned}
\int u(x_1+\widetilde{z}_1(u),\dots,x_d)\,\partial_{x_1} Q(x)\, dx =0.
\end{aligned}
\end{equation*}
If $u$ is cylindrically symmetric, using the fact that $Q$ is radial (and thus, has cylindrical symmetry), it follows
\begin{equation*}
\begin{aligned}
\int u(x_1+\widetilde{z}_1(u),\dots,x_d)\,\partial_{x_j} Q(x)\, dx =0, \qquad \text{ for all } j=2,\dots,d.
\end{aligned}
\end{equation*}
Hence, up to taking $\omega_1$ small, by the uniqueness of the implicit function theorem, it follows that $(z_1(u),\dots,z_d(u))=(\widetilde{z}_1(u),0,\dots,0)$, whenever $u \in U_{\omega_1}$ has cylindrical symmetry, thus, completing the proof. \end{proof}

Assuming that $u(t)\in U_{\omega}$ with $\omega<\omega_1$, for all $t \geq 0$, and $\omega_1>0$ small enough, we can use Proposition \ref{propmodThe} to define the functions $\widetilde{z}_j(t)=z_j(u(t))$, $j=1,\dots,d$. 

\begin{definition}\label{epsidefi} 
For all $t\geq 0$, let $z_j(t)$, $j=1,\dots,d$, be such that $\epsilon_{(z_1(t),\dots,z_d(t))}$ (as defined in \eqref{eqMT3}) satisfies 
\begin{equation*}
\epsilon_{(z_1(t),\dots,z_d(t))} \perp \partial_{x_j}Q, \, \, j=1,\dots,d.
\end{equation*}
In this case, we set
\begin{equation}\label{eqMT4.1}
\epsilon(t) \stackrel{def}{=} \epsilon_{(z_1(t),\dots,z_d(t))}=u(x_1+z_1(t),\dots,x_d+z_d(t))-Q(x).
\end{equation}
\end{definition}
Notice that \eqref{eqMT4} yields
\begin{equation}
\|\epsilon(t)\|_{H^{\frac{\alpha}{2}}}\leq C_1\, \omega.
\end{equation}

\begin{lemma}[Control of the modulation parameters]\label{lemmacontrompar} Assume $0<\alpha<2$ and $1<m<m_{\ast}$ be an integer. Let $\omega_1>0$ be as in Proposition \ref{propmodThe}. There exists $0<\omega_2< \omega_1$ 
such that if for all $t\geq 0$ $u(t)\in U_{\omega_2}$ and $z_j(t)$, $j=1,\dots,d$, are $C^1$ functions of $t$
for each $j=1,\dots,d$, then the parameters $z_j(t)$, $j=1,\dots,d$, satisfy the equations for any $l=1,\dots, d$
\begin{equation}\label{eqMT5}
\begin{aligned}
(z_1'-1)&\big(\|\partial_{x_1} Q\|_{L^2}^2\,\delta_{1,l}-(\epsilon,\partial_{x_1}\partial_{x_l}Q)\big)+\sum_{j=2}^d z_j'(t)\big(\|\partial_{x_l}Q\|_{L^2}^2\,\delta_{j,l}-(\epsilon,\partial_{x_j}\partial_{x_l}Q)\big)\\
&=(\epsilon,L(\partial_{x_1}\partial_{x_l}Q))+(R(\epsilon),\partial_{x_l} Q),
\end{aligned}
\end{equation}
where $\delta_{j,l}$ 
stands for the Kronecker delta function and $R(\epsilon)$ is  as defined in \eqref{eqMT2.1}. 
Moreover, there exists a universal constant $C_2>0$ such that for any $0< \omega < \omega_2$ if $\|\epsilon(t)\|_{L^2}\leq \omega$ 
for all $t\geq 0$, then
\begin{equation}\label{eqMT6.0}
|z_1'(t)-1|+\sum_{j=2}^d|z'_j(t)|\leq C_2 \, \|\epsilon(t)\|_{H^{\frac{\alpha}{2}}}.
\end{equation}
(In the one-dimensional case $d=1$, we assume $\sum_{j=2}^d(\dots)=0$ in \eqref{eqMT5} and \eqref{eqMT6.0}.)
\end{lemma}

\begin{proof} Multiplying \eqref{eqMT2} by $\partial_l Q$, $l=1,\dots,d$, and integrating by parts yields
\begin{equation}\label{E:4.33}
\begin{aligned}
\int \partial_t\epsilon \,\partial_{x_l} Q & +\int L(\epsilon)\,\partial_{x_1}\partial_{x_l} Q =(z_1'(t)-1)\big(\int |\partial_{x_1}Q|^2\big) \delta_{1,l}-(z_1'(t)-1)\int \epsilon \, \partial_{x_1}\partial_{x_l}Q\\
&+\sum_{j=2}^d z_j'(t)\big(\int |\partial_{x_l}Q|^2\big) \, \delta_{j,l}-\sum_{j=2}^d z_j'(t)\int \epsilon \, \partial_{x_j}\partial_{x_l}Q
-\int R(\epsilon) \, \partial_{x_l}Q,
\end{aligned}
\end{equation}
where we also used that $\partial_{x_j} Q\perp \partial_{x_l} Q$, $j\neq l$. Given that $Q\in H^{\infty}(\mathbb{R}^d)$,
the above is justified by the duality pair $H^{-(1+k_{\alpha}+\alpha)}(\mathbb{R}^d)$ and $H^{1+k_{\alpha}+\alpha}(\mathbb{R}^d)$, where $k_{\alpha}\geq 0$ is fixed as in Lemma \ref{Lemeps}. Now, the orthogonality conditions for $\epsilon(t)$ (defined in \eqref{eqMT4.1}) reduces \eqref{E:4.33} to \eqref{eqMT5}. 

Next, there exists some $\omega_2>0$ sufficiently small such that
\begin{equation}\label{matrixIneq}
    |\mathcal{A}(\epsilon,Q)|\gtrsim \|\partial_{x_1}Q\|_{L^2}^2\|\partial_{x_2}Q\|_{L^2}^2\dots \|\partial_{x_d}Q\|_{L^2}^2,
\end{equation}
where
\begin{equation*}
\mathcal{A}(\epsilon,Q)=\left(\begin{array}{c c c c}
\|\partial_{x_1} Q\|_{L^2}^2-(\epsilon,\partial_{x_1}^2Q) &-(\epsilon,\partial_{x_1}\partial_{x_2} Q) & \dots & -(\epsilon,\partial_{x_1}\partial_{x_d} Q) \\
-(\epsilon,\partial_{x_1}\partial_{x_2} Q) &\|\partial_{x_2}Q\|_{L^2}^2-(\epsilon,\partial_{x_2}^2 Q) & \dots & -(\epsilon,\partial_{x_2}\partial_{x_d}Q) \\
 \vdots &\ddots & & \\
 -(\epsilon,\partial_{x_1}\partial_{x_d} Q) &\dots & \dots & \|\partial_{x_d}Q\|_{L^2}^2-(\epsilon,\partial_{x_d}^2Q) \\
\end{array}\right),
\end{equation*}
and the implicit constant in \eqref{matrixIneq} depends on the dimension, provided that $\|\epsilon\|_{L^2} \leq \omega <\omega_2$ for every $t\geq 0$ and  small enough  $\omega_2$. We notice that integration by parts yields 
\begin{equation*}
\begin{aligned}
    |(R(\epsilon),\partial_{x_l}Q)|\lesssim & \sum_{k=2}^m\|\partial_{x_1}\partial_{x_l}Q\|_{L^{\infty}}\|Q^{m-k}\epsilon^k\|_{L^1}\\
    \lesssim & \sum_{k=2}^m \|\partial_{x_1}\partial_{x_l}Q\|_{L^{\infty}}\|Q^{m-k}\|_{L^{\infty}}\|\epsilon^k\|_{L^1},
\end{aligned}
\end{equation*}
for each $l=1,\dots,d$. Now, an application of the Gagliardo-Nirenberg inequality establishes
\begin{equation*}
\begin{aligned}
\|\epsilon\|_{L^k}^k \lesssim \|D^{\frac{\alpha}{2}}\epsilon\|_{L^2}^{\frac{d(k-2)}{\alpha}}\|\epsilon\|_{L^2}^{k-\frac{d(k-2)}{\alpha}}\lesssim \|\epsilon\|_{H^{\frac{\alpha}{2}}}^k.
\end{aligned}
\end{equation*}
Thus, it follows
\begin{equation*}
\begin{aligned}
    |(R(\epsilon),\partial_{x_l}Q)|\lesssim & \sum_{k=2}^m\omega_2^{k-1}\|\epsilon\|_{H^{\frac{\alpha}{2}}}.
\end{aligned}
\end{equation*}
Using  the previous estimates, the matrix $\mathcal{A}=\mathcal{A}(\epsilon,Q)$, and the equations \eqref{eqMT5}, we obtain
\begin{equation*}
\begin{aligned}
\left|\left(\begin{array}{c}
    z_1'(t)-1   \\
    z_2'(t) \\
    \vdots \\
    z_d'(t)
\end{array}\right)\right|=&\frac{1}{|\mathcal{A}|}\left|\text{Adj}(\mathcal{A})\left(\begin{array}{c}
    (\epsilon,L(\partial_{x_1}^2Q))+(R(\epsilon),\partial_{x_1}Q)   \\
    (\epsilon,L(\partial_{x_1}\partial_{x_2}Q))+(R(\epsilon),\partial_{x_2}Q) \\
    \vdots \\
    (\epsilon,L(\partial_{x_1}\partial_{x_d}Q))+(R(\epsilon),\partial_{x_d}Q) 
\end{array}\right)\right|\\
\lesssim& \|\epsilon(t)\|_{H^{\frac{\alpha}{2}}},
\end{aligned}
\end{equation*}
where $\text{Adj}(\mathcal{A})$ denotes the adjoint matrix of $\mathcal{A}$. The above inequality assures the existence of a constant $C_2>0$ such that \eqref{eqMT6.0} holds true, completing the proof.
\end{proof}

\begin{remark}\label{remarkregulModPar}
In the case $d=1$ (or $d\geq 2$ with the solution $u$ of \eqref{EQ:fKdV} cylindrically symmetric), we do not have to assume the $C^1$ regularity of the parameters $z_j(t)$ in Lemma \ref{lemmacontrompar}. To see this, by Proposition \ref{propmodThe}, we have $z_j(t)=0$ for all $j=2,\dots, d$, and the proof that $z_1(t)$ is of class $C^1$ follows the same arguments as in \cite[Lemma 1]{BonaSoyeur1994}, provided that $\omega_2>0$ in Lemma \ref{lemmacontrompar} is sufficiently small.  
\end{remark}


\subsection{Virial-type estimates}

In this part, we introduce a virial-type functional, which is used to show instability of solitary waves in the supercritical case. Assuming to the contrary that a solitary wave is stable, on one hand, we show the upper bound for this functional. On the other hand, its time derivative is lower bounded, which implies growth in time, contradicting the first fact, the boundedness. This type of functional was used in the 1d context of the critical gKdV equation in  \cite{MartelMerle2001} 
(see also the instability argument reviewed in the supercritical gKdV case in \cite{Combet2010} and \cite{FHR2019a}). The two-dimensional version of the virial-type functional was introduced by the second author and her collaborators in \cite{FHR2019b, FHR2019c}.
In our approach it is sufficient to use a truncation of this functional, avoiding 
monotonicity (as, for instance, in \cite{FHR2019c}). 

We recall the scaling generator $\Lambda$ from \eqref{opeLam} and the eigenfunction $\chi_0$ of $L$ corresponding to the negative eigenvalue. We let $\widetilde{\sigma}\in C^{\infty}_c(\mathbb{R})$ be such that $\widetilde{\sigma}(x)=1$, if $|x|\leq 1$, and $\widetilde{\sigma}(x)=0$, if $|x|\geq 2$.
Given $A\geq 1$, we let
\begin{equation}\label{sigmaprim}
    \widetilde{\sigma}_A(x)=\widetilde{\sigma}\Big(\frac{x}{A}\Big).
\end{equation}
For each $x\in \mathbb{R}^d$, we define
\begin{equation}\label{defF}
F(x_1,\dots,x_d)=\int_{-\infty}^{x_1} \big(\Lambda Q(\kappa,x_2,\dots,x_d)+\beta\chi_0(\kappa,x_2,\dots,x_d) \big)\, d\kappa,
\end{equation}
where $\beta \in \mathbb{R}$ to be chosen later. We denote 
\begin{equation}\label{defFA}
    F_A(x_1,\dots,x_d) := F(x_1,\dots,x_d)\,\sigma_A(x_1,\dots,x_d),
\end{equation}
where $\sigma_A(x_1,\dots,x_d):=\widetilde{\sigma}_A(x_1)$ is a function of one variable.

The next lemma establishes some key results for the functions $F$ and $F_A$ introduced above.
\begin{lemma}\label{Fproperties}
Assume $0<\alpha<2$ and $1<m<m_{\ast}$ be an {integer}. 
Then for each $j=2,\dots,d$, 
\begin{equation}\label{boundF0}
\begin{aligned}
|\partial_{x_j}^{l}F(x_1,\dots,x_d)| \lesssim & \langle (x_2,\dots,x_d) \rangle^{-m_1}\Big| \int_{-\infty}^{x_1} \langle \kappa \rangle^{-m_2} \, d\kappa\Big|,
\end{aligned}
\end{equation}
for any integer $l\geq 0$, where $m_1+m_2=d+\alpha+l$, $m_2>1$. In particular, if $ x_1\leq 0$, we have
\begin{equation}\label{boundF}
\begin{aligned}
|\partial_{x_j}^{l}F(x_1,\dots,x_d)|\lesssim \langle (x_2,\dots,x_d) \rangle^{-m_1}\Big|\langle x_1 \rangle^{-m_2} \int_{-\infty}^{x_1} \langle \kappa \rangle^{-m_3} \, d\kappa\Big|,
\end{aligned}
\end{equation}
for any integer $l\geq 0$, 
where $m_1+m_2+m_3=d+\alpha+l$, $m_3>1$. 

In the one-dimensional setting, i.e., $j=1$, the convention is that \eqref{boundF0} and \eqref{boundF} are valid without incorporating the factor $\langle (x_2,\dots,x_d) \rangle^{-m_1}$, i.e., $m_1=0$ in each of the above inequalities.

Moreover, for any fixed $A\geq 1$, 
\begin{equation}
  F_A\in H^{\infty}(\mathbb{R}^d).  
\end{equation}
\end{lemma}

\begin{proof}
The inequality \eqref{boundF0} is a consequence of the decay properties of $Q$ and $\chi_0$ in Lemmas \ref{LemmaQprop} and \ref{LemmaChiprop}, which imply
\begin{equation*}
    \begin{aligned}
    |\partial_{x_j}^{l}\Lambda Q(\kappa,x_2,\dots,x_d)+\beta\partial_{x_j}^{l}\chi_0(\kappa,x_2,\dots,x_d)| \lesssim & \langle (\kappa,x_2,\dots,x_d) \rangle^{-\alpha-d-l} \\
    \lesssim &  \langle (x_2,\dots,x_d) \rangle^{-m_1}\langle \kappa \rangle^{-m_2},
    \end{aligned}
\end{equation*}
since $m_1+m_2=d+\alpha+l$. The condition $\kappa\leq x_1\leq 0$ implies
$\langle x_1 \rangle \lesssim \langle \kappa \rangle$, and hence, \eqref{boundF} is a consequence of \eqref{boundF0}. 

Next, by using \eqref{boundF0}, \eqref{boundF} and the support of the function $\sigma_A$, we have
\begin{equation}\label{L2estF}
\|F_A\|_{L^2(\mathbb{R}^d)}\lesssim 
\|F\chi_{\{|x_1|< 2A\}}\|_{L^2(\mathbb{R}^d)} 
\lesssim 
A^{\frac{1}{2}}\|\langle (x_2,\dots,x_d) \rangle^{-m_1}\|_{L^2(\mathbb{R}^{d-1})} 
\lesssim 
(1+A^{\frac{1}{2}}),
\end{equation}
where the above inequality
is finite provided that, in the case $d\geq 2$, there exist $m_1$ and $m_2$ such that
\begin{equation*}
\left\{\begin{aligned}
&m_1+m_2=d+\alpha\\
&m_1>\frac{d-1}{2} \, \, \text{ and } \, m_2>1.
\end{aligned}\right.
\end{equation*}
Note that such numbers $m_1$ and $m_2$ exist, if $\alpha>\frac{1}{2}-\frac{d}{2}$. In the one-dimensional case, $d=1$, the above conditions are reduced to $m_1=\alpha>0$. Consequently, since Lemmas \ref{LemmaQprop} and \ref{LemmaChiprop} show that the spatial decay of the derivatives of $Q$ and $\chi_0$ is faster in proportion to the order of the derivative taken, by using a similar argument to that in \eqref{L2estF}, we obtain  $F_A\in H^{k}(\mathbb{R}^d)$ for all $k\geq 0$, that is, $F_A\in H^{\infty}(\mathbb{R}^d)$. 
\end{proof}

Recalling $\epsilon(t)$ from \eqref{eqMT4.1}, we are finally ready to introduce 
the virial-type functional 
\begin{equation}\label{Jfunct}
\mathcal{J}_A(t)=\int_{\mathbb{R}^d}\epsilon(t) \, F_A(x)\, dx ,
\end{equation}
for $A \geq 1$.
Note that under the conditions of Lemma \ref{Fproperties},  the above functional is well-defined for any function $\epsilon(t)\in H^{s}(\mathbb{R}^d)$, $s\in \mathbb{R}$, since $F_A\in H^{\infty}(\mathbb{R}^d)$. We conclude this section by finding the time derivative of $\mathcal{J}_A(t)$.

\begin{lemma}\label{lemmaderJ}
Assume $0<\alpha<2$ and $1<m<m_{\ast}$ be an integer. Suppose for any $t\geq 0$, $\epsilon(t)\in H^{\frac{\alpha}{2}}(\mathbb{R}^d)$ and  $z_j(t) \in C^1([0,\infty) )$  for $j=1,\dots,d$. Then the function $t\mapsto \mathcal{J}_A(t)$ is $C^{1}([0,\infty))$ and
\begin{equation}\label{derivJ}
\frac{d}{dt}\mathcal{J}_A(t)=\beta \lambda_0 \int \epsilon \chi_0\sigma_A+ R_1(\epsilon),
\end{equation}
where
\begin{equation*}
\beta := \frac{-\int Q \Lambda Q}{\int Q \chi_0}
\end{equation*}
and recalling the definition of $R$ in \eqref{eqMT2.1}, we denote
\begin{equation}
\begin{aligned}
R_1(\epsilon):=&\int \epsilon Q \sigma_A-(z_1'(t)-1)\int \epsilon (\Lambda Q+\beta \chi_0)\sigma_{A}-(z_1'(t)-1)\int Q(\Lambda Q+\beta \chi_0)(\sigma_{A}-1)\\
&-\int \epsilon[L,\sigma_A] ( \Lambda Q+\beta \chi_0)-\frac{1}{A}\int \epsilon L\big(F\sigma'_A\big)\\
&-\frac{(z'_1(t)-1)}{A}\int (Q+\epsilon)F\sigma'_A-\sum_{j=2}^d z_j'(t)\int (Q+\epsilon)(\partial_{x_j} F)\sigma_A-\int R(\epsilon)F_A,
\end{aligned}
\end{equation}
where $F$and $F_A$ are defined by \eqref{defF} and \eqref{defFA}, respectively, and 
\begin{equation*}
    \sigma'_A(x_1,\dots,x_d)=(\widetilde{\sigma})'\Big(\frac{x_1}{A}\Big).
\end{equation*}
In the one-dimensional case, we assume $\sum_{j=2}^d(\cdots)=0$ in the definition of $R_1(\epsilon)$ above.
\end{lemma}

\begin{proof}
Given that $\epsilon$ solves \eqref{eqMT2} in $H^{-(1+\kappa_{\alpha}+\alpha)}(\mathbb{R}^d)$, where $\kappa_{\alpha}\geq 0$ is defined in Lemma \ref{Lemeps}, and that $F_A \in H^{\infty}(\mathbb{R}^d)$, we can integrate by parts and use the self-adjointness of the operator $L$ defined in \eqref{linearizedop} to get
\begin{equation}
\begin{aligned}
\frac{d}{dt}\mathcal{J}_A(t)=&-\int \epsilon L \partial_{x_1} F_A-(z_1'(t)-1)\int (Q+\epsilon)\partial_{x_1} F_A-\sum_{j=2}^d z'_j(t)\int (Q+\epsilon)\partial_{x_j}F_A\\
&-\int R(\epsilon)F_A\\
=&-\int \epsilon ( L\Lambda Q+\beta L\chi_0)\sigma_A-\int \epsilon[L,\sigma_A] ( \Lambda Q+\beta \chi_0)-\frac{1}{A}\int \epsilon L\big(F\sigma'_A\big)\\
&-(z'_1(t)-1)\int (Q+\epsilon)(\Lambda Q+\beta \chi_0)\sigma_A-\frac{(z'_1(t)-1)}{A}\int (Q+\epsilon)F\sigma'_A\\
&-\sum_{j=2}^d z'_j(t)\int (Q+\epsilon)(\partial_{x_j} F)\sigma_A-\int R(\epsilon)F_A.
\end{aligned}
\end{equation}
By Lemma \ref{propGamop}, we have $L (\Lambda Q) =-Q$, and our choice of $\chi_0$ shows $L\chi_0=-\lambda_0\chi$ with $\lambda_0>0$, thus, we get
 \begin{equation}
\begin{aligned}
\frac{d}{dt}\mathcal{J}_A(t)=&\beta \lambda_0 \int \epsilon \chi_0\sigma_A-(z'_1(t)-1)\int (Q+\epsilon)(\Lambda Q+\beta \chi_0)\sigma_A+\int \epsilon Q \sigma_A\\
&-\int \epsilon[L,\sigma_A] ( \Lambda Q+\beta \chi_0)-\frac{1}{A}\int \epsilon L\big(F\sigma'_A\big)\\
&-\frac{(z'_1(t)-1)}{A}\int (Q+\epsilon)F\sigma'_A-\sum_{j=2}^d z_j'(t)\int (Q+\epsilon)(\partial_{x_j} F)\sigma_A-\int R(\epsilon)F_A.
\end{aligned}
\end{equation}
Setting $\beta=-\frac{\int Q\Lambda Q}{\int Q \chi_0}>0$, we have
\begin{equation*}
\begin{aligned}
\int (Q+\epsilon)(\Lambda Q+\beta \chi_0)\sigma_A=&\int Q(\Lambda Q+\beta \chi_0)+\int Q(\Lambda Q+\beta \chi_0)(\sigma_A-1)\\
&+\int \epsilon(\Lambda Q+\beta \chi_0)\sigma_A  \\
=&\int Q(\Lambda Q+\beta \chi_0)(\sigma_A-1)+\int \epsilon(\Lambda Q+\beta \chi_0)\sigma_A.
\end{aligned}
\end{equation*}
Gathering the expressions together finishes the proof of the lemma.
\end{proof}

\begin{remark}
We emphasize that Theorem \ref{propL}, Lemma \ref{propGamop} and Proposition \ref{propmodThe} are valid for values of $m$ not necessarily an integer. In contrast, the fact that $m$ is an integer is crucial in Lemmas \ref{LemmaChiprop}, \ref{lemWF}, \ref{Lemeps}, \ref{lemmacontrompar}, \ref{Fproperties}, \ref{lemmaderJ}.
\end{remark}


\section{$H^{\frac{\alpha}{2}}$-instability} 
\label{SectInstabProof}

The aim of this section is to establish Theorem \ref{InstMainThm}. We assume $\max\{1-\frac{d}{2},0\}<\alpha<2$ and $\frac{2\alpha}{d}+1<m< m_{\ast}$ is an integer with $m_{\ast}$ defined in \eqref{defcritm}. For simplicity, 
we assume that the speed of the ground state is $c=1$. The general case follows by simple rescaling of our arguments and the linearized operator in \eqref{linearizedop}. Thus, we consider the following sequence of initial conditions
\begin{equation}\label{eqins1}
u_{0,n}(x):=Q_{\lambda_n}(x) \equiv \lambda_n Q(\lambda_n^{\frac{2}{d}} x) \, \text{ with } \, \lambda_n=1+\frac{1}{n},
\end{equation}
where $Q>0$ is the (unique) positive radial $H^{\frac{\alpha}2}$ solution of \eqref{NLEQ}, and thus, by Lemma \ref{LemmaQprop}, we have that $u_{0,n} \in H^{\infty}(\mathbb{R}^d)$ for each $n$.

\begin{proposition}\label{propApr}
Let $u_{0,n}$ be given by \eqref{eqins1}. Then for any integer $n > 0$, 
we have
\begin{equation}
\|u_{0,n}\|_{L^2}=\|Q\|_{L^2} \, \, \text{ and } \, \, E[u_{0,n}]<E[Q].
\end{equation}
Moreover,
\begin{equation}
\|u_{0,n}-Q\|_{H^{\frac{\alpha}{2}}} \to 0 \, \, \text{ as } \, \, n \to \infty.
\end{equation}
\end{proposition}

\begin{proof}
By rescaling, we have $\|u_{0,n}\|_{L^2}=\|Q\|_{L^2}$, and it is not difficult to see that $\|u_{0,n}-Q\|_{H^{\frac{\alpha}{2}}} \to 0$ as $n\to \infty$. It remains then to compute the energy of $u_{0,n}$. We have
\begin{equation}
    \begin{aligned}
    E[u_{0,n}]=&\frac{1}{2}\int |D^{\frac{\alpha}{2}}u_{0,n}(x)|^2\, dx -\frac{1}{m(m+1)}\int (u_{0,n}(x))^{m+1}\, dx \\
    =&\frac{\lambda_n^{\frac{2\alpha}{d}}}{2}\int |D^{\frac{\alpha}{2}}Q(x)|^2\, dx -\frac{\lambda_n^{m-1}}{m(m+1)}\int |Q(x)|^{m+1}\, dx.
    \end{aligned}
\end{equation}
Now, by the Pokhozhaev identities in Lemma \ref{PHident}, we have
\begin{equation}
    \|Q\|_{L^{m+1}}^{m+1}=\frac{\alpha\, m (m+1)}{d(m-1)}\|D^{\frac{\alpha}{2}}Q\|_{L^2}^2.
\end{equation}
Therefore, we obtain
\begin{equation}
    \begin{aligned}
    E[u_{0,n}]-E[Q]=\frac{1}{m(m+1)}\Big(\frac{d(m-1)}{2\alpha}(\lambda_n^{\frac{2\alpha}{d}}-1)-(\lambda_n^{m-1}-1)\Big)\int |Q(x)|^{m+1}\, dx.
    \end{aligned}
\end{equation}
Thus, to show that $E[u_{0,n}]<E[Q]$, it suffices to obtain 
\begin{equation}\label{eqins2}
\frac{d(m-1)}{2\alpha}(\lambda_n^{\frac{2\alpha}{d}}-1)-(\lambda_n^{m-1}-1) < 0. \end{equation}
Since $\frac{2\alpha}{d}<m-1$, the function $f(x):=\frac{d(m-1)}{2\alpha}(x^{\frac{2\alpha}{d}}-1)-(x^{m-1}-1)$ is decreasing for $x> 1$; in particular, yielding $f(\lambda_n)< f(1)=0$, which is exactly \eqref{eqins2}. 
\end{proof}

To prove the theorem, we assume to the contrary that $Q$ is stable. Then for every $\omega>0$, Proposition \ref{propApr} assures that there exists $n(\omega)\geq 1$ integer such that 
\begin{equation}\label{stableassumeq1}
u_{n(\omega)}(t)\in U_{\omega},
\end{equation}
where $u_{n(\omega)}(t)$ solves \eqref{EQ:fKdV} with the initial condition $u_{0,n(\omega)}$. We note that $u_{n(\omega)}$ has cylindrical invariance, since its initial condition satisfies this property and \eqref{EQ:fKdV} is invariant under this property whenever $d\geq 2$. Thus, we take $\omega_0$ such that $0<\omega_0<\omega_2<\omega_1 \ll 1$, where $\omega_1> 0$ is given by Proposition \ref{propmodThe} and $\omega_2>0$ by Lemma \ref{lemmacontrompar} and Remark \ref{remarkregulModPar}. To simplify notation, we will omit the index $n(\omega_0)$ from now on.  By Definition \ref{epsidefi} there exists a function
\begin{equation}\label{epsdef2}
\epsilon(t)=\epsilon_{(z_1(t),\dots, z_d(t))}=u(x+z_1(t),\dots,x_d+z_d(t))-Q(x),
\end{equation}
$x=(x_1,\dots,x_d)\in \mathbb{R}^d$, which satisfies \eqref{eqMT3.1}, and from the second conclusion in Proposition \ref{propmodThe}, we also have $z_j(t)=0$ for all $t\geq 0$, $j=2,\dots,d$.

Since $u(t)\in U_{\omega_0}$, from Proposition \ref{propmodThe} and Lemma \ref{lemmacontrompar} we have
\begin{equation}\label{epsdef3}
\|\epsilon(t)\|_{H^{\frac{\alpha}{2}}}\lesssim \omega_0, \, \, \text{ and } \, \,  |z'(t)-1|\lesssim \omega_0.
\end{equation}
We first establish an upper bound for the operator $\mathcal{J}_A$, $A\geq 1$. 
We remark that the next result is valid for any dispersion $0<\alpha<2$ (there will be some technical restriction later). 
\begin{proposition}\label{boundJ}
Assume $0<\alpha<2$ and  $\frac{2\alpha}{d}+1<m< m_{\ast}$ is an integer.
Let $A\geq 1$ and $\mathcal{J}_A(t)$ be the functional defined in \eqref{Jfunct}. If $\omega_0>0$ in \eqref{epsdef3} is sufficiently small, then there exists a constant $M_0>0$ independent of $A$ such that
for all $t \geq 0$
\begin{equation}\label{E:J_A<}
|\mathcal{J}_A(t)|\leq M_0A^{\frac{1}{2}}.
\end{equation}
\end{proposition}

\begin{proof}
Let $d\geq 2$ and recall the support of the function $\sigma_A$. Then by the Cauchy-Schwarz inequality and \eqref{boundF0}, we upper estimate
\begin{equation}
\begin{aligned}
|\mathcal{J}_A(\epsilon)|
\lesssim & \int_{\mathbb{R}^{d-1}} \sup_{x_1\in \mathbb{R}}|F(x_1,x_2,\dots,x_d)|\Big(\int_{-2A}^{2A}|\epsilon(x_1,x_2,\dots,x_d,t)|\, dx_1\Big) \, dx_2\dots dx_d\\
&\lesssim A^{\frac{1}{2}}\int_{\mathbb{R}^{d-1}} \langle (x_2,\dots,x_d) \rangle^{-m_1}\Big(\int_{\mathbb{R}} |\epsilon(x_1,\dots,x_d,t)|^2\, dx_1\Big)^{\frac{1}{2}} \, dx_2\dots dx_d\\
&\lesssim A^{\frac{1}{2}}\|\langle (x_2,\dots,x_d) \rangle^{-m_1}\|_{L^2(\mathbb{R}^{d-1})}\|\epsilon(t)\|_{L^2(\mathbb{R}^d)},
\end{aligned}
\end{equation}
where we set $\frac{d}{2}-\frac{1}{2}<m_1<d+\alpha-1$. We remark that a similar estimate as the one above also holds when $d=1$. Noting that the last two norms above are bounded independently of $A$, completes the proof of the lemma. 
\end{proof}

\begin{theorem}\label{lowerboundThm}
Assume $\max\{0,1-\frac{d}{2}\}<\alpha<2$ and $\frac{2\alpha}{d}+1<m< m_{\ast}$ is an integer. If $\omega_0>0$ is sufficiently small and $A\geq 1$ is large enough, then there exists a constant $a_0>0$ such that for all $t \geq 0$
\begin{equation}
\Big|\frac{d}{dt}\mathcal{J}_A(t) \Big| \geq a_0>0.
\end{equation}
\end{theorem}

\begin{proof}
Since $z_j(t)=0$ for all $t\geq 0$, $j=2,\dots,d$, the identity \eqref{derivJ} reduces to
\begin{equation}
\frac{d}{dt}\mathcal{J}_A(t)=\beta \lambda_0 \int \epsilon \chi_0+ R_2(\epsilon),
\end{equation}
where
\begin{equation}
\begin{aligned}
R_2(\epsilon)=&\beta \lambda_0 \int \epsilon \chi_0(\sigma_A-1)+\int \epsilon Q \sigma_A-(z_1'(t)-1)\int \epsilon (\Lambda Q+\beta \chi_0)\sigma_{A}\\
&-(z_1'(t)-1)\int Q(\Lambda Q+\beta \chi_0)(\sigma_{A}-1)\\
&-\int \epsilon[L,\sigma_A] ( \Lambda Q+\beta \chi_0)-\frac{1}{A}\int \epsilon L\big(F\sigma'_A\big)\\
&-\frac{(z'_1(t)-1)}{A}\int (Q+\epsilon)F\sigma'_A-\int R(\epsilon)F_A\\
=:&\sum_{j=1}^8 R_{2,j},
\end{aligned}
\end{equation}
with $R(\epsilon)$ defined by \eqref{eqMT2.1}. We proceed to estimate each factor $R_{2,j}$ in the identity above. But first, we use the mean value inequality to deduce
\begin{equation}\label{diffapprox}
    |\sigma_A(x_1,\dots,x_d)-1|=|\widetilde{\sigma}_A(x_1)-\widetilde{\sigma}_A(0)|\lesssim \frac{\langle x_1 \rangle}{A},
\end{equation}
where the implicit constant above is independent of $A\geq 1$. Thus, the above expression, the Cauchy-Schwarz inequality, and the decay properties of $\chi_0$ yield
\begin{equation*}
    \begin{aligned}
    |R_{2,1}|\lesssim \frac{1}{A}\|\epsilon\|_{L^2}\|\langle x_1 \rangle \chi_0\|_{L^2}\lesssim  \frac{1}{A}\|\epsilon\|_{L^2}.
    \end{aligned}
\end{equation*}
Note that $\langle x_1\rangle \chi_0\in L^2(\mathbb{R}^d)$, if $\alpha>1-\frac{d}{2}$ (which holds true when $d\geq 2$), or if  $\alpha>\frac{1}{2}$ in the one-dimensional setting, $d=1$. 
For the second term we write
\begin{equation*}
R_{2,2}=\int \epsilon Q+\int \epsilon Q (\sigma_A-1).
\end{equation*}
By the $L^2$ conservation for the solutions of \eqref{EQ:fKdV} and \eqref{eqins2}, 
we get
\begin{equation*}
\int Q^2=\int u^2=\int Q^2+2\int \epsilon Q+\int\epsilon^2,
\end{equation*}
and thus,
\begin{equation}\label{eqinst4}
\Big|\int \epsilon Q\Big|\leq \frac{1}{2}\|\epsilon\|_{L^2}^2.
\end{equation}
Then, \eqref{eqinst4} and \eqref{diffapprox} imply
\begin{equation}
      \begin{aligned}
    |R_{2,2}|\lesssim \|\epsilon\|_{L^2}^2+\frac{1}{A}\|\epsilon\|_{L^2}\|\langle x_1 \rangle Q\|_{L^2}\lesssim \|\epsilon\|_{L^2}^2+\frac{1}{A}\|\epsilon\|_{L^2},
    \end{aligned}
\end{equation}
where for the last step we used the decay of $Q$ from \eqref{poldecayGS} and the fact that $\alpha>1-\frac{d}2$.
We use \eqref{eqMT6.0} and the Cauchy-Schwarz inequality to get
\begin{equation}\label{eqinst3}
\begin{aligned}
|R_{2,3}| &\lesssim \|\epsilon\|_{H^{\frac{\alpha}{2}}}\Big|\int \epsilon (\Lambda Q+\beta \chi_0)\sigma_A \Big| \\
&\lesssim \|\epsilon\|_{H^{\frac{\alpha}{2}}}\|\epsilon\|_{L^2}\| (\Lambda Q+\beta \chi_0) \|_{L^2}.
\end{aligned}
\end{equation}
Similarly, we use \eqref{diffapprox} to deduce
\begin{equation}
\begin{aligned}
|R_{2,4}| &\lesssim \|\epsilon\|_{H^{\frac{\alpha}{2}}}\Big|\int Q (\Lambda Q+\beta \chi_0)(\sigma_A-1) \Big| \\
&\lesssim \frac{1}{A}\|\epsilon\|_{H^{\frac{\alpha}{2}}}\|\langle x_1 \rangle Q\|_{L^2}\| (\Lambda Q+\beta \chi_0) \|_{L^2}.
\end{aligned}
\end{equation}
Given that $L=D^{\alpha}+1-Q^{m-1}$, by using that $[L,\sigma_A]=[D^{\alpha},\sigma_A]$, we write 
\begin{equation*}
\begin{aligned}
\left[L,\sigma_{A}\right](\Lambda Q+\beta \chi_0)=&\sum_{0<|\beta|\leq \alpha}\frac{1}{\beta!}\partial^{\beta}(\sigma_A)D^{\alpha,\beta}(\Lambda Q+\beta \chi_0)\\
&+R^{\ast}(\sigma_A,(\Lambda Q+\beta \chi_0)),    
\end{aligned}
\end{equation*}
where by Proposition \ref{fractionalDeriv},
\begin{equation*}
  \|R^{\ast}(\sigma_A,(\Lambda Q+\beta \chi_0))\|_{L^2} \lesssim \|D^{\alpha}(\sigma_A)\|_{L^{\infty}}\|(\Lambda Q+\beta \chi_0)\|_{L^2}\lesssim \frac{1}{A^{\alpha}}\|D^{\alpha}\sigma\|_{L^{\infty}}\|(\Lambda Q+\beta \chi_0)\|_{L^2}.
\end{equation*}
Above, we have used that $D^{\alpha}\sigma\in L^{\infty}(\mathbb{R}^d)$, which can be proven, by using the singular integral representation of $D^{\alpha}$ (see, \cite{NezzaPalatucciValdinoci2012}), and a change of variables to write
\begin{equation}\label{Doperaonsig}
\begin{aligned}
D^{\alpha}(\sigma_A)(x)=&\frac{1}{A^{\alpha}}D^{\alpha}(\sigma)(\frac{x}{A})\\
=&-\frac{c_{d,\alpha}}{2A^{\alpha}} \, p.v.\int_{\mathbb{R}^d} \frac{\sigma(\widetilde{x}+y)+\sigma(\widetilde{x}-y)-2\sigma(\widetilde{x})}{|y|^{d+\alpha}} \, dy,
\end{aligned}
\end{equation}
for some constant $c_{d,\alpha}$, and where $\widetilde{x}=\frac{x}{A}$ . Since $\sigma_A(x)=\widetilde{\sigma}_A(x_1)$, a second order expansion of $\sigma(x)$ assures the integrability at the origin of the integral on the right-hand side of \eqref{Doperaonsig}, while the integrability outside the origin follows from the fact that $\sigma\in L^{\infty}(\mathbb{R}^d)$. Consequently, we conclude
\begin{equation*}
    \|D^{\alpha}(\sigma_A)\|_{L^{\infty}}\lesssim \frac{\|D^{\alpha}\sigma\|_{L^{\infty}}}{A^{\alpha}}\lesssim \frac{1}{A^{\alpha}}.
\end{equation*}
Since $A\geq 1$,
\begin{equation*}
\begin{aligned}
\|\sum_{0<|\beta|\leq \alpha}\frac{1}{\beta!}\partial^{\beta}(\sigma_A)D^{\alpha,\beta}(\Lambda Q+\beta \chi_0)\|_{L^2} &\lesssim \sum_{0<|\beta|\leq \alpha}\|\partial^{\beta}(\sigma_A)\|_{L^{\infty}}\|D^{\alpha,\beta}(\Lambda Q+\beta \chi_0)\|_{L^2}\\
&\lesssim \frac{1}{A}\|(\Lambda Q+\beta \chi_0)\|_{H^{\alpha}}.
\end{aligned}
\end{equation*}
Summarizing, we obtain 
\begin{equation*}
 \|[L,\sigma_{A}](\Lambda Q+\beta \chi_0)\|_{L^2}\lesssim \Big(\frac{1}{A}+\frac{1}{A^{\alpha}}\Big)\|(\Lambda Q+\beta \chi_0)\|_{H^{\alpha}},
\end{equation*}
which allows us to conclude that
\begin{equation}
|R_{2,5}|\lesssim 
\|\epsilon\|_{L^2}\|[L,\sigma_A](\Lambda Q+\beta \chi_0)\|_{L^2} 
\lesssim 
\Big(\frac{1}{A}+\frac{1}{A^{\alpha}}\Big)\|\epsilon\|_{L^2}\|\Lambda Q+\beta \chi_0\|_{H^{\alpha}} 
\lesssim 
\Big(\frac{1}{A}+\frac{1}{A^{\alpha}}\Big)\|\epsilon\|_{L^2},
\end{equation}
where we used that $Q,\chi_0 \in H^{\infty}(\mathbb{R}^d)$ and the spatial decay properties of these functions. Next, we use the Cauchy-Schwarz inequality to find
\begin{equation*}
\begin{aligned}
|R_{2,6}|\lesssim & \frac{1}{A}\|\epsilon\|_{L^2}\|L(F\sigma_A')\|_{L^2} \\
\lesssim & \frac{1}{A} \|\epsilon\|_{L^2}\big(\|F\sigma_A'\|_{H^{\alpha}}+\|Q^{m-1}F\sigma_A'\|_{L^2}\big)\\ 
\lesssim & \frac{1}{A} \|\epsilon\|_{L^2}\big(\|F\sigma_A'\|_{H^{\alpha}}+\|F\|_{L^{\infty}} \|\sigma_A'\|_{L^{\infty}}  \|Q^{m-1}\|_{L^2}\big), 
\end{aligned}
\end{equation*}
where $\sigma_A'(x_1,\dots,x_d)=\widetilde{\sigma} \big(\frac{x_1}{A}\big)$,  where $\widetilde{\sigma}$ is given by \eqref{sigmaprim}.
Using $\alpha\leq 2$ and the estimate \eqref{boundF0}, similar reasoning as in \eqref{L2estF} implies 
\begin{equation*}
    \begin{aligned}
    \|F\sigma_A'\|_{H^{\alpha}}\lesssim 
    & {\|F\sigma_A'\|_{L^2} }
    +\sum_{j=1}^d\|\partial_{x_j}^2(F\sigma_A')\|_{L^2}\\
    \lesssim &1+\|\langle (x_2,\dots,x_d) \rangle^{-m_1}\|_{L^2(\mathbb{R}^{d-1})}\sum_{k=1}^3\|\sigma_A^{(k)}\|_{L^2(\mathbb{R})}\\
    \lesssim &  1+A^{\frac{1}{2}},
    \end{aligned}
\end{equation*}
with $\frac{d}{2}-\frac{1}{2}<m_1<d+\alpha-1$, the implicit constant above is independent of $A\geq 1$, and $\sigma^{(k)}_A(x_1,\dots,x_d)=(\widetilde{\sigma})^{(k)}\big(\frac{x_1}{A}\big)$, $k=1,2,3$.  
Hence, we have
\begin{equation*}
\begin{aligned}
|R_{2,6}|\lesssim \frac{(1+A^{\frac{1}{2}})}{A}\|\epsilon\|_{L^2}.    
\end{aligned}
\end{equation*}
A similar argument shows that 
\begin{equation*}
|R_{2,7}| \lesssim \frac{\|\epsilon\|_{H^{\frac{\alpha}{2}}}}{A}\big(\|Q\|_{L^2}+\|\epsilon\|_{L^2}\big)\|F\sigma_A'\|_{L^2} 
\lesssim \frac{\|\epsilon\|_{H^{\frac{\alpha}{2}}}}{A^{\frac{1}{2}}}.
\end{equation*}
By \eqref{eqMT2.1},  integrating by parts and using the Gagliardo–Nirenberg inequality (Proposition \ref{GNineq}), we get
\begin{equation}\label{eqinst2}
\begin{aligned}
|R_{2,8}|=\Big|\frac{1}{m} & \int \Big( \sum_{k=2}^{m} \binom{m}{k} Q^{m-k}\epsilon^k\Big)\partial_{x_1}(F_A)\Big| \lesssim \|\partial_{x_1}(F_A)\|_{L^{\infty}}\sum_{k=2}^{m}\|\epsilon\|_{L^k}^k \\
& \lesssim \|\partial_{x_1}(F_A)\|_{L^{\infty}} \sum_{k=2}^{m} \|D^{\frac{\alpha}{2}}\epsilon\|_{L^2}^{\frac{d(k-2)}{\alpha}}\|\epsilon\|_{L^2}^{k-\frac{d(k-2)}{\alpha}} 
\lesssim (1+\omega_0) \|\epsilon\|_{H^{\frac{\alpha}{2}}}^2,
\end{aligned}
\end{equation}
where the implicit constant is independent of $A\geq 1$. Gathering the estimates for $R_{2,j}$, $j=1,\dots,8$, together, we have
\begin{equation}\label{eqinst5}
|R_2(\epsilon)|\lesssim (\|\epsilon\|_{H^{\frac{\alpha}{2}}}+\frac{1}{A}+\frac{1}{A^{\alpha}}+\frac{1}{A^{\frac{1}{2}}})\|\epsilon\|_{H^{\frac{\alpha}{2}}},
\end{equation}
the implicit constant above is independent of $A\geq 1$, $0\leq \omega_0\leq 1$. Setting 
\begin{equation}
\theta(t)=\int \epsilon(t)\chi_0,
\end{equation}
by the last statement in  Theorem \ref{propL}, there exists $k_1,k_2>0$ such that
\begin{equation}
\theta(t)^2\geq \frac{k_1}{k_2}\|\epsilon(t)\|_{L^2}^2-\frac{1}{k_2}(L(\epsilon(t)),\epsilon(t)).
\end{equation}
From Lemma \ref{lemWF} and Proposition \ref{propApr}, it follows that 
\begin{equation}
\begin{aligned}
(L(\epsilon(t)),\epsilon(t))=&2W[u]-2W[Q]-2K[\epsilon]\\
=&2(E[u_0]-E[Q])+2(M[u_0]-M[Q])-2K[\epsilon]\\
= & -\delta_0-2K[\epsilon],
\end{aligned}
\end{equation}
with $\delta_0>0$. Consequently, there exists a constant $C>0$ such that
\begin{equation}
\begin{aligned}
\theta(t)^2 &\geq \frac{k_1}{k_2}\|\epsilon(t)\|_{H^{\frac{\alpha}{2}}}^2+\frac{1}{k_2}\big(\delta_0+2K[\epsilon]\big )\\
& \geq  \frac{k_1}{k_2}\|\epsilon(t)\|_{H^{\frac{\alpha}{2}}}^2+\frac{\delta_0}{k_2}-\frac{C}{k_2}
\sum_{k=3}^{m+1} \|D^{\frac{\alpha}{2}}\epsilon\|_{L^2}^{\frac{d(k-2)}{\alpha}}\|\epsilon\|_{L^2}^{k-\frac{d(k-2)}{\alpha}}.
\end{aligned}
\end{equation}
By following the argument in \eqref{eqinst2}, and using that $\epsilon$ satisfies the first inequality in  \eqref{epsdef3}, we can take $\omega_0$ small to obtain
\begin{equation}
\begin{aligned}
\theta(t)^2 \geq  \frac{k_1}{2k_2}\|\epsilon(t)\|_{H^{\frac{\alpha}{2}}}^2+\frac{\delta_0}{k_2}.
\end{aligned}
\end{equation}
Thus, the sign of $\theta(t)$ remains the same for all time $t\geq 0$. If $\theta(t)$ is positive, taking the square root of the above inequality and using \eqref{eqinst5}, we deduce
\begin{equation}
\begin{aligned}
\frac{d}{dt}\mathcal{J}_A(t)\geq & c \beta \lambda_0 \sqrt{\frac{k_1}{k_2}}\|\epsilon(t)\|_{H^{\frac{\alpha}{2}}}+c\beta \lambda_0 \sqrt{\frac{\delta_0}{k_2}}+R_{2}(\epsilon)\\
\geq & c \beta \lambda_0 \sqrt{\frac{k_1}{k_2}}\|\epsilon(t)\|_{H^{\frac{\alpha}{2}}}+c\beta \lambda_0 \sqrt{\frac{\delta_0}{k_2}}-c_1\big(\omega_0+\frac{1}{A}+\frac{1}{A^{\alpha}}+\frac{1}{A^{\frac{1}{2}}}\big)\|\epsilon\|_{H^{\frac{\alpha}{2}}}
\end{aligned}
\end{equation} 
for some $c>0$, and with $c_1>0$ independent of $A\geq 1$ and $\omega_0$. Thus, taking $0<\omega_0< 1$ small, and $A\geq 1$ large both in proportion to $c\beta \lambda_0 \sqrt{\frac{k_1}{k_2}}$ (which is a universal constant), we have
\begin{equation}
\begin{aligned}
\frac{d}{dt}\mathcal{J}_A(t) \geq  \frac{c \beta \lambda_0}{2} \sqrt{\frac{k_1}{k_2}}\|\epsilon(t)\|_{H^{\frac{\alpha}{2}}}+c\beta \lambda_0 \sqrt{\frac{\delta_0}{k_2}} >0,
\end{aligned}
\end{equation} 
for some constant $c>0$. We emphasize that in this part, we fix $A\geq 1$ such that the above inequality holds. If $\theta(t)\leq 0$, taking $0<\omega_0$ small enough, and $A\geq 1$ large, we can use a similar reasoning as above to assure the existence of a constant $a_0>0$ such that $\frac{d}{dt}\mathcal{J}_A(t)\leq -a_0<0$, completing the proof. 
\end{proof}

\begin{proof}[Proof of Theorem \ref{InstMainThm}] We fix $\omega_0>0$, and $A\geq 1$ such that Theorem \ref{lowerboundThm} is valid. If $\frac{d}{dt}\mathcal{J}_A(t)\geq a_0>0$, then integrating in the time variable $t$ both sides of this inequality yields
\begin{equation*}
\mathcal{J}_A(t)\geq a_0 t+\mathcal{J}_A(0), \, \, \text{ for all } \, t\geq 0,
\end{equation*}
which contradicts the boundedness of $\mathcal{J}_A(t)$ in Proposition \ref{boundJ}. A similar conclusion holds assuming that $\frac{d}{dt}\mathcal{J}_A(t)\leq -a_0<0$. We emphasize that the spatial decay properties of $Q$ in Theorem \ref{existTHR} yield
\begin{equation*}
\begin{aligned}
|\mathcal{J}_A(0)|\lesssim \|F_A\|_{L^{\infty}}\|Q\|_{L^1}\lesssim 1,
\end{aligned}    
\end{equation*}
where by Lemma \ref{Fproperties}, the implicit constant above is independent of $A\geq 1$.
\end{proof}


\section*{Appendix}\label{Appendix}

In this section we prove Lemmas \ref{decayQm}, \ref{conmident} and \ref{kernDecay}.

\begin{proof}[Proof of Lemma \ref{decayQm}]
By Leibniz's rule, we write
\begin{equation*}
    \partial^{\widetilde{\gamma}}\big(x^{\widetilde{\beta}}\partial^{\widetilde{\eta}}(Q^m)\big)(x)=\sum_{\substack{0\leq \widetilde{\gamma}_1\leq \widetilde{\gamma}\\ \widetilde{\eta}_1+\dots+\widetilde{\eta}_m=\widetilde{\gamma}-\widetilde{\gamma}_1+\widetilde{\eta}\\ |\widetilde{\gamma}_1|\leq |\widetilde{\beta}| }}c_{\widetilde{\gamma},\widetilde{\beta},\widetilde{\eta},\widetilde{\gamma}_1, \widetilde{\eta}_1,\dots,\widetilde{\eta}_m}\partial^{\widetilde{\gamma}_1}(x^{\widetilde{\beta}})\partial^{\widetilde{\eta}_1}Q\dots \partial^{\widetilde{\eta}_m}Q.
\end{equation*}
We use \eqref{decayinduction} to deduce
\begin{equation*}
\begin{aligned}
|\partial^{\widetilde{\gamma}_1}(x^{\widetilde{\beta}})(\partial^{\widetilde{\eta}_1}Q\dots \partial^{\widetilde{\eta}_m}Q)(x)|\lesssim & \langle x \rangle^{|\widetilde{\beta}|-|\widetilde{\gamma}_1|}\langle x \rangle^{-m\alpha-m d-\sum_{j=1}^m\min\{|\widetilde{\eta}_j|,\kappa-1\}}\\
\lesssim &  \big(\langle x \rangle^{|\widetilde{\beta}|-|\widetilde{\gamma}_1|}\langle x \rangle^{-(m-1)\alpha-(m-1) d-\sum_{j=1}^m\min\{|\widetilde{\eta}_j|,\kappa-1\}}\big)\langle x \rangle^{-\alpha-d}.
\end{aligned}
\end{equation*}
We set $H_{\widetilde{\gamma}_1, \widetilde{\eta}_1,\dots,\widetilde{\eta}_m}(x):=\langle x \rangle^{|\widetilde{\beta}|-|\widetilde{\eta}_1|}\langle x \rangle^{-(m-1)\alpha-(m-1) d-\sum_{j=1}^m\min\{|\widetilde{\eta}_j|,\kappa-1\}}$. Thus, to complete the proof of Lemma \ref{decayQm}, it is enough to show that $H_{\widetilde{\gamma}_1, \widetilde{\eta}_1,\dots,\widetilde{\eta}_m}(x)$ is bounded, continuous, and $H_{\widetilde{\gamma}_1, \widetilde{\eta}_1,\dots,\widetilde{\eta}_m}(x)\to 0$ as $|x|\to \infty$. Indeed, if for some $0\leq j'\leq m$, $\min\{|\widetilde{\eta}_{j'}|,\kappa-1\}=\kappa-1$, by using that $|\widetilde{\beta}|\leq \kappa$, we have
\begin{equation*}
\begin{aligned}
|H_{\widetilde{\gamma}_1, \widetilde{\eta}_1,\dots,\widetilde{\eta}_m}(x)|\lesssim & \langle x \rangle^{\kappa-|\widetilde{\gamma}_1|-(\kappa-1)}\langle x \rangle^{-(m-1)\alpha-(m-1) d-\sum_{\substack{1\leq j \leq m\\ j\neq j'}}\min\{|\widetilde{\eta}_j|,\kappa-1\}}\\
\lesssim & \langle x \rangle^{-(m-1)\alpha}\langle x \rangle^{-(m-1) d+1}\langle x \rangle^{-|\widetilde{\gamma}_1|-\sum_{\substack{1\leq j \leq m\\ j\neq j'}}\min\{|\widetilde{\eta}_j|,\kappa-1\}}.
\end{aligned}
\end{equation*}
Since $m\geq 2$, we have $-(m-1)d+1\leq 0$. Thus, it is clear that the above inequality implies the desired properties of the function $H_{\widetilde{\gamma}_1, \widetilde{\eta}_1,\dots,\widetilde{\eta}_m}(x)$.

Now, we assume that for all $j$,  $\min\{|\widetilde{\eta}_{j'}|,\kappa-1\}=|\widetilde{\eta}_{j'}|$, then since $|\widetilde{\eta}_1|+\dots+|\widetilde{\eta}_m|=|\widetilde{\gamma}|-|\widetilde{\gamma}_1|+|\widetilde{\eta}|$, we have
\begin{equation*}
\begin{aligned}
|H_{\widetilde{\gamma}_1, \widetilde{\eta}_1,\dots,\widetilde{\eta}_m}(x)|\lesssim & \langle x \rangle^{|\widetilde{\beta}|-|\widetilde{\gamma}_1|}\langle x \rangle^{-(m-1)\alpha-(m-1) d-\sum_{j=1}^m |\widetilde{\eta}_j|}\\
\lesssim & \langle x \rangle^{|\widetilde{\beta}|-|\widetilde{\gamma}|-|\widetilde{\eta}|-(m-1)\alpha-(m-1) d}.
\end{aligned}
\end{equation*}
The assumptions on $\widetilde{\beta}$, $\widetilde{\gamma}$, and $\widetilde{\eta}$ establish the required properties of $H_{\widetilde{\gamma}_1, \widetilde{\eta}_1,\dots,\widetilde{\eta}_m}(x)$, thus, finishing the proof.
\end{proof}

Next, we deduce Lemma \ref{conmident}.
\begin{proof}[Proof of Lemma \ref{conmident}]
We assume that $f\in \mathcal{S}(\mathbb{R}^d)$, the validity of identity \eqref{identdecay} under the general assumptions on $f$ stated in Proposition \ref{conmident} follows by approximating $f$ by Schwartz functions. By taking the Fourier transform of the commutator, and applying Leibniz rule, we get 
\begin{equation}\label{eqdecaQ1}
\begin{aligned}
    \mathcal{F}\big([x^{\beta},\partial^{\eta}(1+D^{\alpha})^{-1}]f\big)(\xi)=&i^{|\beta|}\partial^{\beta}\Big(\frac{i^{|\eta|}\xi^{\eta}}{(1+|\xi|^{\alpha})}\widehat{f}\Big)(\xi)-\frac{i^{|\beta|+|\eta|}\xi^{\eta}}{(1+|\xi|^{\alpha})}\partial^{\beta}\widehat{f}\\
    =&i^{|\beta|+|\eta|}\sum_{\substack{\beta_1+\beta_2=\beta\\ |\beta_1|\geq 1}}\partial^{\beta_1}\Big(\frac{\xi^{\eta}}{(1+|\xi|^{\alpha})}\Big)\partial^{\beta_2}\widehat{f}(\xi),
\end{aligned}
\end{equation}
where $\mathcal{F}$ denotes the Fourier transform operator. Recalling that $|\eta|=|\beta|$, by an inductive argument on the magnitude of the multi-index $|\beta_1|\geq 1$, with $1\leq |\beta_1|\leq |\beta|=|\eta|$, it is not hard to obtain the following identity
\begin{equation}\label{eqdecaQ2}
    \begin{aligned}
    \partial^{\beta_1}\Big(\frac{\xi^{\eta}}{(1+|\xi|^{\alpha})}\Big)=&\sum_{\substack{1\leq |\widetilde{\beta_1}|\leq |\beta_1|}}\sum_{1\leq l\leq |\widetilde{\beta_1}|}c_{\widetilde{\beta_1},l}\frac{|\xi|^{\alpha l-2|\widetilde{\beta}_1|}P_{l,2|\widetilde{\beta}_1|}(\xi)}{(1+|\xi|^{\alpha})^{1+|\widetilde{\beta}_1|}}P_{l,|\beta|-|\beta_1|}(\xi)\\
    &+\frac{c_{\beta,
    \beta_1}}{(1+|\xi|^{\alpha})}P_{0,|\beta|-|\beta_1|}(\xi),
    \end{aligned}
\end{equation}
where given $\widetilde{m}\geq 0$ integer, $P_{l,\widetilde{m}}(\xi)$, and $P_{0,\widetilde{m}}(\xi)$ denote homogeneous polynomials of degree $\widetilde{m}$, and the above identity holds for some constants $c_{l,\widetilde{\beta}_1},c_{\beta,\beta_1}$. Plugging \eqref{eqdecaQ2} into \eqref{eqdecaQ1}, then taking the inverse Fourier transform, yields the desired result.
\end{proof}

We conclude this section with the proof of Lemma \ref{kernDecay}. This result follows by standard arguments, though for the sake of completeness, we include it here.

\begin{proof}[Proof of Lemma \ref{kernDecay}]
Since $\psi \in C^{\infty}_c(\mathbb{R}^d)$, $|\gamma|=2\rho$, and $1\leq l\leq \rho$, we find
\begin{equation*}
\begin{aligned}
|K_{\rho,l,\gamma}(x)|= \Big|\int \frac{|\xi|^{\alpha l-2\rho}\xi^{\gamma}}{(\lambda+|\xi|^{\alpha})^{1+\rho}}\psi(\xi)e^{-i x\cdot \xi}\, d\xi \Big|\lesssim  \|\psi\|_{L^1}.
\end{aligned}    
\end{equation*}
Thus, we will assume that $|x|\geq 1$. Let $R=|x|^{-1}$, we write
\begin{equation*}
\begin{aligned}
K_{\rho,l,\gamma}(x)=& i^{|\gamma|}\int \frac{|\xi|^{\alpha l-2\rho}\xi^{\gamma}}{(\lambda+|\xi|^{\alpha})^{1+\rho}}\psi(\xi)\psi\big(\frac{\xi}{R}\big)e^{-i x\cdot \xi}\, d\xi   \\
&+i^{|\gamma|}\int \frac{|\xi|^{\alpha l-2\rho}\xi^{\gamma}}{(\lambda+|\xi|^{\alpha})^{1+\rho}}\psi(\xi)\big(1-\psi\big(\frac{\xi}{R}\big)\big)e^{-i x\cdot \xi}\, d\xi \\
=:&K_{\rho,l,\gamma,1}(x)+K_{\rho,l,\gamma,2}(x).
\end{aligned}
\end{equation*}
Using the fact that $1\leq l\leq \rho$, it follows that
\begin{equation*}
\begin{aligned}
|K_{\rho,l,\gamma,1}(x)|\lesssim & \int_{|\xi|\leq 2R} |\xi|^{\alpha} \frac{|\xi|^{\alpha (l-1)}}{(\lambda+|\xi|^{\alpha})^{1+\rho}}|\psi(\xi)|\, d\xi\\
\lesssim & \int_{|\xi|\leq 2R} |\xi|^{\alpha}\, d\xi\\
\lesssim & R^{\alpha+d}\sim |x|^{-(\alpha+d)}.
\end{aligned}    
\end{equation*}
On the other hand, let $j=1,\dots,d$ be any fixed direction such that $|x_j|\sim |x|$, where the implicit constant depends on $d$ and is independent of $j$ and $x$. We consider an integer $N\geq 1$ with $N>\alpha+d$. By using that $\partial_{\xi_j}^N(e^{-i x\cdot \xi})=(-i x_j)^N e^{-i x\cdot \xi}$, and integrating by parts, we deduce
\begin{equation*}
\begin{aligned}
K_{\rho,l,\gamma,2}(x)=&\frac{i^{|\gamma|}}{(i x_j)^N}\sum_{l_1+l_2+l_3+l_4+l_5=N}c_{l_1,\dots,l_5}\int \partial_{\xi_j}^{l_1}(|\xi|^{\alpha-2\rho})\partial_{\xi_j}^{l_2}(\xi^{\gamma})\partial_{\xi_j}^{l_3}\Big(\frac{|\xi|^{\alpha (l-1)}}{(\lambda+|\xi|^{\alpha})^{1+\rho}}\Big)\\
&\hspace{5cm} \times\partial_{\xi_j}^{l_4}\psi(\xi)\partial_{\xi_j}^{l_5}\big(1-\psi\big(\frac{\xi}{R}\big)\big)e^{-i x\cdot \xi}\, d\xi.
\end{aligned}
\end{equation*}
Now, since 
\begin{equation*}
    \big|\partial_{\xi_j}^{l_3}\Big(\frac{|\xi|^{\alpha (l-1)}}{(\lambda+|\xi|^{\alpha})^{1+\rho}}\Big)\big|\lesssim |\xi|^{-l_3},
\end{equation*}
$\xi \neq 0$, it is seen that
\begin{equation*}
\begin{aligned}
|K_{\rho,l,\gamma,2}(x)|\lesssim & \frac{1}{|x|^N} \sum_{l_1+l_2+l_3+l_4+l_5=N}\int  |\xi|^{\alpha-2\rho+|\gamma|-l_1-l_2-l_3}|\partial_{\xi_j}^{l_4}\psi(\xi)| \big|\partial_{\xi_j}^{l_5}\big(1-\psi\big(\frac{\xi}{R}\big)\big)\big|\, d\xi\\
\lesssim & \frac{1}{|x|^N} \sum_{l_1+l_2+l_3+l_4+l_5=N}\int  |\xi|^{\alpha-N}||\xi|^{l_4}\partial_{\xi_j}^{l_4}\psi(\xi)| \big||\xi|^{l_5}\partial_{\xi_j}^{l_5}\big(1-\psi\big(\frac{\xi}{R}\big)\big)\big|\, d\xi   \\
\lesssim & \frac{1}{|x|^N} \int_{|\xi|\geq 2R}  |\xi|^{\alpha-N}\, d\xi\\
\lesssim & \frac{R^{\alpha+d-N}}{|x|^N}\sim |x|^{-\alpha-d},
\end{aligned}
\end{equation*}
which completes the proof of \eqref{Kdecay1}. 

On the other hand, the condition $\mu>\max\{d-\alpha,0\}$ assures
\begin{equation*}
|\widetilde{K}_{\mu,\rho,l,\gamma}(x)|\lesssim \int_{|\xi|\geq 1}\frac{|\xi|^{\alpha l}}{(\lambda+|\xi|^{\alpha})^{1+\rho}}\frac{1}{|\xi|^{\mu}}\, d\xi\leq \int_{|\xi|\geq 1} \frac{1}{|\xi|^{\mu+\alpha}}\, d\xi\lesssim 1.
\end{equation*}
Using that $1-\psi(\xi)$ is supported on the set $\{|\xi|\geq 1\}$, one can use integration by parts and similar arguments as for proving the estimate for $K_{\rho,l,\gamma,2}(x)$ above, to deduce \eqref{Kdecay2}. This completes the proof of the lemma. 
\end{proof}


{\bf Data Availability.} 
Data sharing not applicable to this article as no datasets were generated or analyzed during the current study.


\bibliographystyle{abbrv}
\bibliography{ref_fKdV}

\end{document}